\documentclass[11pt,final]{amsart}

\usepackage{lmodern}

\usepackage[active]{srcltx}
\usepackage[utf8]{inputenc}
\usepackage[english]{babel}
\usepackage{graphicx}
\usepackage[T1]{fontenc}
\usepackage{amsmath}
\usepackage{amssymb} 
\usepackage{amsthm}
\usepackage{mathabx}
\usepackage{bbm} % Pour faire \mathbb{1}
\usepackage{bm} % Pour faire des lettres grecques en italique
\usepackage{eqnarray}
\usepackage{tikz}
\usepackage[numbers]{natbib}
\usepackage{array}
\usepackage[left=2.05cm,right=2.05cm,top=2.95cm,bottom=2.8cm]{geometry}
\usepackage[foot]{amsaddr} % Pour mettre l'auteur en dessous

\usepackage{setspace}
\usepackage[pdfborder={0 0 0}]{hyperref}

\numberwithin{equation}{section} % Numérotation des équations

\let\oldtocsection=\tocsection
 
\let\oldtocsubsection=\tocsubsection
 
\renewcommand{\tocsection}[2]{\hspace{0em}\oldtocsection{#1}{#2}}
\renewcommand{\tocsubsection}[2]{\hspace{3em}\oldtocsubsection{#1}{#2}}

\newcommand{\R}{\mathbb{R}}

\newcommand{\N}{\mathbb{N}}
\newcommand{\C}{\mathbb{C}} % Complex numbers
\newcommand{\Hil}{\mathcal{H}} % H Hibert space
\newcommand{\Hili}{\Hil_{\text{i}}} % Hilbert monotone
\newcommand{\Prob}{\mathcal{P}}
\newcommand{\Probec}{\Prob_{\mathrm{ec}}} % Location-scatter family
\newcommand{\Leb}{\mathcal{L}} % Lebesgue measure
\newcommand{\M}{\mathcal{M}} % Measures
\newcommand{\A}{\mathcal{A}} % A caligraphié
\newcommand{\g}{\mathfrak{g}} % Tenseur métrique

\newcommand{\Sph}{\mathbb{S}} % Sphère unité
\newcommand{\U}{\mathbb{B}} % Disque unité 

\newcommand{\F}{\mathcal{F}} % Ensemble de fonctions

\newcommand{\BT}{\mathrm{BT}} % Boundary term 
\newcommand{\mubf}{\bm{\mu}}
\newcommand{\nubf}{\bm{\nu}}
\newcommand{\vbf}{\mathbf{v}} % v bold
\newcommand{\wbf}{\mathbf{w}} % w bold
\newcommand{\Abf}{\mathbf{A}} % A bold
\newcommand{\Bbf}{\mathbf{B}} % A bold
\newcommand{\Cbf}{\mathbf{C}} % A bold
\newcommand{\Dbf}{\mathbf{D}} % A bold
\newcommand{\Em}{\mathbf{E}} % E momentum

\newcommand{\Or}{\mathring{\Omega}} % Intérieur de \Omega
\newcommand{\Dr}{\mathring{D}} % Intérieur de \Omega

\newcommand{\nD}{\mathbf{n}_D} % Normale extérieure à D
\newcommand{\nO}{\mathbf{n}_{\Omega}} % Normale extérieure à Omega

\DeclareMathOperator*{\esssup}{ess\,sup} % Sup essentiel
\DeclareMathOperator*{\essinf}{ess\,inf} % Infimum essentiel
\DeclareMathOperator*{\argmin}{arg\,min} % argmin essentiel

\newcommand{\Dir}{\mathrm{Dir}}

\newcommand{\ddr}{\mathrm{d}} % d droit pour les intégrales
\newcommand{\dr}{\partial}

\newcommand{\Id}{\mathrm{Id}}

\newcommand{\Tr}{\mathrm{Tr}} % Trace

\newcommand{\cov}{\mathrm{cov}}

\newcommand{\1}{\mathbbm{1}} 

\newcommand{\dst}[1]{\displaystyle{#1}}
\newcommand{\bsl}{\backslash}
\newtheorem{theo}{Theorem}[section]
\newtheorem*{theo*}{Theorem}
\newtheorem{prop}[theo]{Proposition}
\newtheorem{crl}[theo]{Corollary}
\newtheorem{lm}[theo]{Lemma}
\newtheorem{defi}[theo]{Definition}

\newtheorem*{asmp*}{Assumption}

\begin{document}

\title{Harmonic mappings valued in the Wasserstein space}

\date{\today}
\author{Hugo Lavenant}
\address{Laboratoire de Math\'ematiques d'Orsay, Univ. Paris-Sud, CNRS, Universit\'e Paris-Saclay, 91405 Orsay Cedex, France}
\email{hugo.lavenant@u-psud.fr}

\begin{abstract}
We propose a definition of the Dirichlet energy (which is roughly speaking the integral of the square of the gradient) for mappings $\mubf : \Omega \to (\Prob(D), W_2)$ defined over a subset $\Omega$ of $\R^p$ and valued in the space $\Prob(D)$ of probability measures on a compact convex subset $D$ of $\R^q$ endowed with the quadratic Wasserstein distance. Our definition relies on a straightforward generalization of the Benamou-Brenier formula (already introduced by Brenier) but is also equivalent to the definition of Korevaar, Schoen and Jost as limit of approximate Dirichlet energies, and to the definition of Reshetnyak of Sobolev spaces valued in metric spaces. 

We study harmonic mappings, i.e. minimizers of the Dirichlet energy provided that the values on the boundary $\dr \Omega$ are fixed. The notion of constant-speed geodesics in the Wasserstein space is recovered by taking for $\Omega$ a segment of $\R$. As the Wasserstein space $(\Prob(D), W_2)$ is positively curved in the sense of Alexandrov we cannot apply the theory of Korevaar, Schoen and Jost and we use instead arguments based on optimal transport. We manage to get existence of harmonic mappings provided that the boundary values are Lipschitz on $\dr \Omega$, uniqueness is an open question. 

If $\Omega$ is a segment of $\R$, it is known that a curve valued in the Wasserstein space $\Prob(D)$ can be seen as a superposition of curves valued in $D$. We show that it is no longer the case in higher dimensions: a generic mapping $\Omega \to \Prob(D)$ cannot be represented as the superposition of mappings $\Omega \to D$. 

We are able to show the validity of a maximum principle: the composition $F \circ \mubf$ of a function $F : \Prob(D) \to \R$ convex along generalized geodesics and a harmonic mapping $\mubf : \Omega \to \Prob(D)$ is a subharmonic real-valued function. 

We also study the special case where we restrict ourselves to a given family of elliptically contoured distributions (a finite-dimensional and geodesically convex submanifold of $(\Prob(D), W_2)$ which generalizes the case of Gaussian measures) and show that it boils down to harmonic mappings valued in the Riemannian manifold of symmetric matrices endowed with the distance coming from optimal transport. 
\end{abstract}

\keywords{Wasserstein space; harmonic maps; Dirichlet problem}

\maketitle

\setcounter{tocdepth}{2}
\tableofcontents

\section{Introduction}

\subsection{Harmonic mappings}

If $f : \Omega \to \R$ is a real-valued function defined on a subset $\Omega$ of $\R^p$, one says that $f$ is \emph{harmonic} if
\begin{equation}
\label{equation_Laplace}
\Delta f = 0,
\end{equation}
where $\Delta = \sum_{\alpha = 1}^p \dr_{\alpha \alpha}$ denotes the Laplacian operator. Although this equation can be traced back to physics (for instance it corresponds to the equation satisfied by the electric potential in the absence of charge, or the one satisfied by the temperature in some homogeneous and isotropic medium when the permanent regime is reached), it has revealed to have its own mathematical interest \cite{Helein2008}. In particular it is associated to a concept of \emph{equilibrium}, as for an harmonic function $f$, the value of $f$ at a point $\xi \in \Omega$ is always equal to the mean of the values of $f$ on a ball centered at $\xi$. A whole line of research has been devoted to define harmonic mappings $f : X \to Y$ where $X$ and $Y$ are spaces without a structure as strong as the Euclidean one. If $X$ and $Y$ are Riemannian manifolds, one can define an analogue of \eqref{equation_Laplace} which involves the metric tensors of both $X$ and $Y$ (see for instance \cite{Eells1964} or, for a modern presentation, \cite{Jost2008, Helein2008}). The standard hypothesis to get existence results and nice properties of harmonic mappings is that $X$ has a positive curvature and $Y$ has a negative curvature. In the 90s, Korevaar and Schoen \cite{Korevaar1993} on one side and Jost \cite{Jost1994} on the other side, presented independently a more general setting and showed that one can define harmonic mappings $f : \Omega \to Y$ provided that $\Omega$ is a compact Riemannian manifold and $Y$ is a metric space with negative curvature in the sense of Alexandrov \cite[Section 2.1]{Korevaar1993}.

The most robust point of view for the definition of harmonic mappings valued in metric spaces is related to the Dirichlet problem. Indeed, if we go back to the case where $Y = \R$, a function $f : \Omega \to \R$ is harmonic if and only if it is a minimizer of the Dirichlet energy 
\begin{equation*}
\Dir(g) := \int_\Omega \frac{1}{2} |\nabla g(\xi)|^2 \ddr \xi
\end{equation*} 
among all functions $g : \Omega \to \R$ having the same values as $f$ on $\dr \Omega$ the boundary of $\Omega$. The main advantage of this formulation is that it involves only first order derivatives, and most of the concepts involving first order derivatives can be defined on metric spaces even without any vectorial structure \cite{Ambrosio2003}. Korevaar, Schoen and Jost proved that for every separable metric space $Y$, one can define the analogue of the Dirichlet energy of any mapping $f : \Omega \to Y$. Then under the assumption that $Y$ has a negative curvature in the sense of Alexandrov, they proved existence and uniqueness of a minimizer of the Dirichlet energy (provided that the values at the boundary $\dr \Omega$ are fixed), interior and boundary regularity of the minimizer and lots of other properties similar to harmonic mappings between manifolds. Most of the proofs mimic the ones in the Euclidean case and rely only on the curvature properties of the target space $Y$. To quote Korevaar and Schoen: ``\emph{We find the generality, elegance, and simplicity of the proofs presented here to be an indication that we have found the proper framework for their expression}'' \cite[p. 614]{Korevaar1993}. 

In this article, our goal is to define and to study harmonic mappings defined over a compact domain $\Omega$ of $\R^p$ and valued in the space of probability measures endowed with the distance coming from optimal transport, the so-called quadratic Wasserstein space \cite{Villani2003, Ambrosio2008, OTAM}. If $D$ is a convex compact domain of $\R^q$, and if $\mu, \nu$ are two probability measures on $D$ (the set of probability measures on $D$ is denoted by $\Prob(D)$) then the (quadratic) Wasserstein distance $W_2(\mu, \nu)$ between the two is defined as
\begin{equation*}
W_2(\mu, \nu) := \inf_\pi \sqrt{\iint_{D \times D} |x-y|^2 \ddr \pi(x,y)},
\end{equation*}  
where the infimum is taken over all transport plans $\pi \in \Prob(D \times D)$ whose marginals are $\mu$ and $\nu$. We will define the Dirichlet energy for mappings $\mubf : \Omega \to (\Prob(D), W_2)$ and study its minimizers under the constraint that the values at the boundary $\dr \Omega$ are fixed. It is known \cite[Section 7.3]{Ambrosio2008} that $(\Prob(D), W_2)$ is a positively curved space in the sense of Alexandrov, hence the whole theory of Korevaar, Schoen and Jost does not apply: we have to leave the world of ``\emph{generality, elegance and simplicity}''. Though we manage to develop a fairly satisfying theory of Dirichlet energy and harmonic mappings valued in the Wasserstein space, it is \emph{ad hoc}: it intensively relies on specific properties of $(\Prob(D), W_2)$ and is hardly generalizable to other positively curved spaces.  

\subsection{Related works}

This work can be seen as an extension of an article written by Brenier \cite{Brenier2003} almost 15 years ago. Recently, few articles \cite{Solomon2013, Solomon2014, Vogt2017, Lu2017} have been published on related topics even though none of them seems aware of Brenier's work.

In Section 3 of \cite{Brenier2003}, Brenier proposed a definition of what he called \emph{generalized harmonic functions} which is the same thing as our harmonic mappings valued in the Wasserstein space. He defined the Dirichlet energy for such mappings; proved the existence of harmonic mappings in some special cases and gave an explicit solution in the very special case where all measures on $\dr \Omega$ are Dirac masses; indicated the formulation of the dual problem; and formulated some conjectures. In the present article, we will rely on the same definition of Dirichlet energy as in Brenier's article, but we push the analysis much further: we provide a rigorous functional analysis framework; link the Dirichlet energy with already known notions of analysis in metric spaces (in particular with the definition of Korevaar, Schoen and Jost); prove the existence of harmonic mappings in a more general context; and answer Brenier's conjectures. 

In \cite{Solomon2013}, the authors study \emph{soft maps} (which are nothing more than maps $\Omega \to \Prob(D)$ except that $\Omega$ and $D$ are surfaces, i.e. Riemannian manifolds of dimension $2$) and define a Dirichlet energy in the same way as Korevaar, Schoen and Jost. These maps are seen as relaxations of ``classical'' maps $\Omega \to D$, and they focus on numerical computation and visualization of theses soft maps, see also \cite{Solomon2014} for applications to supervised learning. On the other hand, they do not analyze in detail the theoretical properties of the Dirichlet energy and harmonic mappings, which in contrast is the main topic of the present article. In \cite{Lu2017}, the author provides some theoretical analysis of soft maps by focusing on the cases where the boundary measures on $\dr \Omega$ are either Dirac masses or Gaussian measures.  

Finally, in \cite{Vogt2017} the authors also study mappings valued in the space of probability measures, but are rather interested in the bounded variation norm (the integral of the norm of the gradient) than in the Dirichlet energy. Their provide applications to the denoising of measure-valued images.   

Apart from these articles, let us underline the interest of our work by relating it to other already known concepts:  
\begin{itemize}
\item[•] It is well known that harmonic mappings defined over an interval of $\R$ and valued in a geodesic space are precisely the constant-speed geodesics, and it is the case with our definition. Thus our work can be seen as extending the definition of geodesics in the Wasserstein space, the latter being an object which is now well understood. 
\item[•] As we said above, our definition of Dirichlet energy coincides with the one of Korevaar, Schoen and Jost. In particular, our work shows that their definition can be applied to positively curved spaces and still get some non trivial result, even though we rely on the very special structure of the Wasserstein space.
\item[•] To study the regularity of minimal surfaces, Almgren proposed the notion of $Q$-valued functions (see \cite{Almgren2000} or \cite{DeLellis2011} for a clear and self-contained reference), which can be seen (up to renormalization) as mappings defined on $\Omega \subset \R^p$ and valued in the subset $\A_Q(D)$ (where $Q \geqslant 1$ is an integer) of the Wasserstein space $(\Prob(D), W_2)$ defined as 
\begin{equation*}
\A_Q(D) := \left\{  \frac{1}{Q} \sum_{i=1}^Q \delta_{x_i} \ : \ (x_1, x_2, \ldots, x_Q) \in D^Q \right\}. 
\end{equation*}
In other words, $\A_Q(D)$ is the set of probability measures which are combinations of at most $Q$ Dirac masses with weights which are multiples of $1/Q$, and is endowed with the Wasserstein distance $W_2$. To put it shortly, a $Q$-function is a function which in every point takes $Q$ unordered different values (counted with multiplicity). There exists a beautiful existence and regularity theory for harmonic $Q$-functions. As $\bigcup_{Q \geqslant 1} \A_Q(D)$ is dense in $\Prob(D)$, it would be tempting to see the Dirichlet problem for mappings valued in the Wasserstein space $\Prob(D)$ as the limit as $Q \to + \infty$ of the Dirichlet problem for $Q$-functions. However, it is not so obvious that this limit really holds, and most of the results in the theory of $Q$-functions are proved by induction on $Q$ through clever decompositions and combinatorial arguments, hence they depend heavily on $Q$ and not much can be passed to the limit $Q \to + \infty$. Notice that the space $\A_Q(D)$ is also positively curved in the sense of Alexandrov (the example in \cite[Section 7.3]{Ambrosio2008} lives in $\A_2(D)$), hence the theory of $Q$-functions is a theory of harmonic mappings valued in a positively curved space. However, it is known \cite[Theorem 2.1]{DeLellis2011} that $\A_Q(D)$ is in a bilipschitz bijection with a subset of $\R^N$ for some large $N$: with $Q$-functions we stay in the finite-dimensional world. On the contrary, in the present article, the target space $(\Prob(D), W_2)$ will be both positively curved and genuinely infinite-dimensional.      
\end{itemize}

\subsection{Main definitions and results}

Let us go into the details and summarize the content of this article as well as the key insights. In this discussion we will stay informal, with sometimes sloppy or non rigorous statements. 

\bigskip

In Section \ref{section_preliminairies}, we give our notations and collect some known facts about the Wasserstein space, which can be found in standard textbooks. In particular, we take $\Omega$ and $D$ two compact domains of respectively $\R^p$ and $\R^q$ and assume that $D$ is convex.

\bigskip

Section \ref{section_Dirichlet_energy} is devoted to the definition and properties of the Dirichlet energy of a mapping $\mubf : \Omega \to \Prob(D)$. 

The idea is to start from curves valued in the Wasserstein space and the so-called Benamou-Brenier formula \cite{Benamou2000}. If $I$ is a segment of $\R$ and $\mubf : I \to \Prob(D)$ is an absolutely continuous curve, then its Dirichlet energy, which is nothing else than the integral of the square of its metric derivative \cite[Section 1.1]{Ambrosio2008} is equal to 
\begin{equation*}
\Dir(\mubf) = \inf_{\vbf} \left\{ \int_I \left(\int_D \frac{1}{2} |\vbf(t,x)|^2 \mubf(t, \ddr x) \right) \ddr t \ : \ \vbf : I \times D \to \R^q \ \text{ and } \dr_t \mubf + \nabla \cdot(\mubf \vbf) = 0 \right\},
\end{equation*} 
which means that one minimizes the integral over time of the kinetic energy among all velocity fields $\vbf$ such that the continuity equation $\dr_t \mubf + \nabla \cdot(\mubf \vbf) = 0$ is satisfied. What Benamou and Brenier understood is that the correct variable is the momentum $\Em = \vbf \mubf$. Indeed, the continuity equation $\dr_t \mubf + \nabla \cdot \Em = 0$ becomes a linear constraint and 
\begin{equation*}
\int_I \left(\int_D \frac{1}{2} |\vbf(t,x)|^2 \mubf(t, \ddr x) \right) \ddr t = \iint_{I \times D} \frac{|\Em|^2}{2 \mubf}
\end{equation*}
is a convex function of the pair $(\mubf, \Em)$. In particular, to find the constant-speed geodesic between $\mu$ and $\nu \in \Prob(D)$, assuming that $I = [0,1]$, one minimizes the convex Dirichlet energy over the pairs $(\mubf, \Em)$ with linear constraints given by the continuity equation, that $\mubf(0) = \mu$ and that $\mubf(1) = \nu$. 

As noticed in \cite[Section 3]{Brenier2003}, this formulation can be directly extended to the case where the source space is no longer of dimension $1$: if $\Omega$ is a subset of $\R^p$, one can define a (generalized) continuity equation for the pair $\mubf : \Omega \to \Prob(D)$ and $\Em : \Omega \times D \to \R^{pq}$ by 
\begin{equation}
\label{equation_continuity_equation}
\nabla_\Omega \mubf + \nabla_D \cdot \Em = 0, 
\end{equation} 
where $\nabla_\Omega$ stands for the gradient w.r.t. variables in $\Omega$ and $\nabla_D \cdot$ stands for the divergence w.r.t. variables in $D$. More precisely if $(\Em^{\alpha i})_{1 \leqslant \alpha \leqslant p, 1 \leqslant i \leqslant q}$ denote the components of $\Em$, and if the derivatives w.r.t. $\Omega$ (resp. $D$) are denoted by $(\dr_\alpha)_{1 \leqslant \alpha \leqslant p}$ (resp. $(\dr_i)_{1 \leqslant i \leqslant q}$) then the the continuity equation reads: for any $\alpha \in \{ 1,2, \ldots, p \}$, 
\begin{equation*}
\dr_\alpha \mubf + \sum_{i=1}^q \dr_i \Em^{\alpha i} = 0. 
\end{equation*}
The Dirichlet energy of the pair $(\mubf, \Em)$ is defined as 
\begin{equation*}
\iint_{\Omega \times D} \frac{|\Em|^2}{2 \mubf} = \iint_{\Omega \times D} \sum_{\alpha = 1}^p \sum_{i=1}^q \frac{|\Em^{i \alpha}|^2}{2 \mubf},
\end{equation*}
and $\Dir(\mubf)$, the Dirichlet energy of $\mubf$, is the minimal Dirichlet energy of the pairs $(\mubf, \Em)$ among all $\Em$ such that the continuity equation is satisfied (Definition \ref{definition_dirichlet_mu}). It is a straightforward copy of the classical proofs of optimal transport to show that there exists a unique optimal momentum $\Em$ (which we call the tangent momentum) which is written $\Em = \vbf \mubf$ for some velocity field $\vbf : \Omega \times D \to \R^{pq}$, and that $\Dir$ is convex and lower semi-continuous (l.s.c.). 

We will prove that for $\mubf : \Omega \to \Prob(D)$, one has $\Dir(\mubf) < + \infty$ if and only if for any $u : \Prob(D) \to \R$ which is $1$-Lipschitz, one has that $u \circ \mubf$ belongs to $H^1(\Omega)$ with $|\nabla (u \circ \mubf)| \leqslant g$, where $g \in L^2(\Omega)$ is independent of $u$. Moreover, the minimal $g$ will be shown to be controlled from above and below by 
\begin{equation*}
\sqrt{\int_D |\vbf(\cdot, x)|^2 \mubf(\cdot, \ddr x)} \in L^2(\Omega),
\end{equation*} 
where $\Em = \vbf \mubf$ is the tangent momentum (Theorem \ref{theorem_equivalence_sobolev_metric}). This precisley shows that the space $\{ \mubf : \Omega \to \Prob(D) \ : \ \Dir(\mubf) < + \infty \}$ coincides with the set $H^1(\Omega, \Prob(D))$, where the latter is defined in the sense of Reshetnyak \cite{Reshetnyak1997}, and that the gradient of $\mubf$ in the sense of Reshetnyak (the minimal $g$ above) is related to the tangent velocity field $\vbf$. The Dirichlet energy is not equal to the $L^2$ norm of $g$, as it is already the case in the classical framework \cite{Chiron2007}: if we see $\vbf : \Omega \times D \to \R^{pq}$ as a matrix-valued field, the Benamou-Brenier definition measures the magnitude of $\vbf$ with the Hilbert-Schmidt norm, whereas the optimal $g$ from the definition of Reshetnyak is rather related to the operator norm of the matrices. Nevertheless, it implies for instance, it implies that Lipschitz mappings $\mubf : \Omega \to \Prob(D)$ (i.e. such that $W_2(\mubf(\xi), \mubf(\eta)) \leqslant C |\xi - \eta|$ for any $\xi, \eta \in \Omega$) have a finite Dirichlet energy.  

We will also prove that our Dirichlet energy coincides with the one of Korevaar and Schoen, as well as Jost. The idea of theses authors goes as follows: if $f : \Omega \to \R$ is smooth, then for any $\xi \in \R^p$, 
\begin{equation*}
|\nabla f(\xi)|^2 = \lim_{\varepsilon \to 0} C_p \int_{B(\xi, \varepsilon)} \frac{|f(\eta) - f(\xi)|^2}{\varepsilon^{p+2}} \ddr \eta,
\end{equation*}
for some constant $C_p$ which depends on the dimension $p$ of $\Omega$, where $B(\xi, \varepsilon)$ is the ball of center $\xi$ and radius $\varepsilon$. Thus, if $\varepsilon > 0$ is small, a good approximation of the Dirichlet energy of $f$ would be 
\begin{equation*}
\Dir(f) = \int_\Omega \frac{1}{2} |\nabla f(\xi)|^2 \ddr \xi \simeq C_p \iint_{\Omega \times \Omega} \frac{|f(\xi) - f(\eta)|^2}{2 \varepsilon^{p+2}} \1_{|\xi - \eta| \leqslant \varepsilon} \ddr \xi \ddr \eta.
\end{equation*}
Notice that the right hand side (r.h.s.) involves only metric quantities, thus its definition can be extended if $f : \Omega \to Y$ where $(Y,d)$ is an arbitrary metric space by replacing $|f(\xi) - f(\eta)|^2$ by $d(f(\xi), f(\eta))^2$: this is what is done and extensively studied in \cite[Section 1]{Korevaar1993} (curvature assumptions on $Y$ are not required for the definition of the Dirichlet energy, but are used to derive existence, uniqueness and properties of the minimizers). The counterpart in our case is to define the $\varepsilon$-Dirichlet energy of a mapping $\mubf : \Omega \to \Prob(D)$ by
\begin{equation*}
\Dir_\varepsilon(\mubf) := C_p \iint_{\Omega \times \Omega} \frac{W_2^2(\mubf(\xi), \mubf(\eta))}{2 \varepsilon^{p+2}} \1_{|\xi - \eta| \leqslant \varepsilon} \ddr \xi \ddr \eta. 
\end{equation*}  
We are able show that $\Dir_\varepsilon$ converges to $\Dir$ as $\varepsilon \to 0$: it holds pointwisely but also in the sense of $\Gamma$-convergence (Theorem \ref{theorem_equivalence_KS}). For both the equivalence with the definition of Korevaar, Schoen and Jost, or with the one of Reshetnyak, the difficulty is not to guess them (they are fairly simple at the formal level) but to conduct careful approximation arguments. 

To conclude the section, we will show how one can define values on $\dr \Omega$ for mappings $\mubf : \Omega \to \Prob(D)$ with finite Dirichlet energy. There already exists a trace theory in \cite{Korevaar1993}, however in view of the dual formulation for the Dirichlet problem, we prefer to define trace values by extending the continuity equation up to the boundary of $\Omega$. Indeed, multiplying \eqref{equation_continuity_equation} by a test function $\varphi \in C^1(\Omega \times D, \R^p)$ valued in $\R^p$, we get the following weak formulation: 
\begin{equation*}
\iint_{\Omega \times D} \nabla_\Omega \cdot \varphi \ddr \mubf + \iint_{\Omega \times D} \nabla_D \varphi \cdot \ddr \Em = \int_{\dr \Omega} \left( \int_D \varphi(\xi, x) \cdot \nO(\xi) \mubf(\xi, \ddr x) \right) \sigma(\ddr \xi),
\end{equation*} 
where $\nO$ is the outward normal to $\dr \Omega$ and $\sigma$ the surface measure. We will show that, if $\Dir(\mubf) < + \infty$, then the r.h.s. can always be defined as a finite vector-valued measure acting on $\varphi$ called $\BT_{\mubf}$ (Theorem \ref{theorem_boundary_values}). Two mappings will have the same values on the boundary $\dr \Omega$ if, by definition, they have the same boundary term. 

\bigskip

In Section \ref{section_Dirichlet_problem} we define the Dirichlet problem and establish its dual formulation. This is fairly classic in optimal transport theory, our proofs do not bring any new ideas.  

To define the Dirichlet problem, we assume that a mapping $\mubf_b : \Omega \to \Prob(D)$ with finite Dirichlet energy is given and we study
\begin{equation*}
\min_{\mubf} \{ \Dir(\mubf) \ : \ \mubf = \mubf_b \text{ on } \dr \Omega \}.
\end{equation*}
Thanks to the Benamou-Brenier formulation, existence of a solution is a straightforward application of the direct method of calculus of variations (Theorem \ref{theorem_existence_solution_Dirichlet}). As we discuss it in the core of the article, we do not know if uniqueness holds. Only in some particular case where the boundary values belong to a family of elliptically contoured distributions we are able to prove uniqueness (see below in the introduction).  

In the formulation of the Dirichlet problem, we define the boundary conditions through a mapping $\mubf_b$ defined on the whole $\Omega$. A natural question arises: if $\mubf_b : \dr \Omega \to \Prob(D)$ is given, is it possible to extend it on $\Omega$ in such a way that $\Dir(\mubf_b) < + \infty$? We will show that the answer to this question is positive if $\mubf_b$ is Lipschitz on $\dr \Omega$, indeed in this case one can extend it as a Lipschitz mapping on $\Omega$. The question of the existence of a Lipschitz extension for mappings $f : Z \to Y$, where $Z \subset X$ and $X,Y$ are metric spaces has been intensively studied, see for instance \cite{Lang1997,Ohta2009} and references therein. The general philosophy is that lower bounds on the curvature are required for the source space $X$, whereas upper bounds on the curvature are required for the target space $Y$. In our case, there are no upper bounds for the curvature of the target space $\Prob(D)$, hence we cannot apply classical results. However, we use the fact that we want to extend Lipschitz mappings defined not on an arbitrary closed subset of $\Omega$, but on the boundary $\dr \Omega$ which has some regularity. By some \emph{ad hoc} construction, we are able to treat the case where $\Omega$ is a ball, but we cannot control the Lipschitz constant of the extension on $\Omega$ by the Lipschitz constant of the mapping on $\dr \Omega$. Nevertheless, we can conclude for smooth domains, as they can be cut in a finite number of pieces, each piece being in a bilipschitz bijection with a ball (Theorem \ref{theorem_lipschitz_extension}).  

Let us establish here the dual formulation \emph{via} a formal $\inf - \sup$ exchange, it was already done in \cite{Brenier2003}. Indeed, given the definition of $\Dir$ and the weak formulation of the continuity equation, 
\begin{align*}
\min_{\mubf} \{ \Dir(\mubf)&  \ : \ \mubf = \mubf_b \text{ on } \dr \Omega \} \\
& = \inf_{\mubf, \vbf} \left[ \iint_{\Omega \times D} \frac{1}{2} |\vbf|^2 \mubf + \sup_{\varphi \in C^1(\Omega \times D, \R^p)} \left( \BT_{\mubf_b}(\varphi) - \iint_{\Omega \times D} \nabla_\Omega \cdot \varphi \ddr \mubf - \iint_{\Omega \times D} \nabla_D \varphi \cdot \vbf \mubf \right) \right] \\
& = \sup_{\varphi \in C^1(\Omega \times D, \R^p)} \left[ \BT_{\mubf_b}(\varphi) + \inf_{\mubf, \vbf} \iint_{\Omega \times D} \left(  \frac{1}{2} |\vbf|^2 - \nabla_D \varphi \cdot \vbf - \nabla_\Omega \cdot \varphi \right) \mubf \right].
\end{align*}
Optimizing in $\vbf$, we have that $\vbf = \nabla_D \varphi$, and then the infimum in $\mubf$ is translated into the constraint $\nabla_\Omega \cdot \varphi + \frac{1}{2} |\nabla_D \varphi|^2 \leqslant 0$. Hence, we have (formally, and it is proved rigorously in the core of the article, see Theorem \ref{theorem_duality}) the following identity: 
\begin{equation*}
\sup_\varphi \Bigg\{ \BT_{\mubf_b}(\varphi) \ : \ \varphi \in C^1(\Omega \times D, \R^p) \ \text{ and } \ \nabla_\Omega \cdot \varphi + \frac{|\nabla_D \varphi|^2}{2} \leqslant 0 \Bigg\} 
 = \min_{\mubf} \{ \Dir(\mubf)  \ : \ \mubf = \mubf_b \text{ on } \dr \Omega \}.
\end{equation*}
We do not have an existence result for solutions $\varphi$ of the dual problem. Notice that $\varphi$ is a vector-valued function, but there is only a scalar constraint on it: the dual problem looks harder than in the case where $\Omega$ is a segment of $\R$. Formally, as it is done in \cite{Brenier2003}, one can get optimality conditions out of the dual formulation. Indeed, we have that $\vbf = \nabla_D \varphi$ and, from the optimization in $\mubf$, that $\nabla_\Omega \cdot \varphi + \frac{1}{2} |\nabla_D \varphi|^2 = 0$ $\mubf$-a.e. If we assume that $\mubf$ is strictly positive a.e., we end up with the following system for $\vbf$ (the first equation is just a rewriting of the fact that $\vbf$ is a gradient, the second one is obtained by differentiating $\nabla_\Omega \cdot \varphi + \frac{1}{2} |\nabla_D \varphi|^2 = 0$ w.r.t. $D$): 
\begin{equation*}
\begin{cases}
\dr_i \vbf^{\alpha j} = \dr_j \vbf^{\alpha i} & \text{for } \alpha \in \{ 1,2, \ldots, p \} \text{ and } i,j \in \{1,2, \ldots,q \}, \\
\dst{\sum_{\alpha = 1}^p \dr_\alpha \vbf^{\alpha i} + \sum_{\alpha = 1}^p \sum_{j = 1}^q \vbf^{\alpha j} \dr_j \vbf^{\alpha i}} = 0 & \text{for } i \in \{1,2, \ldots,q \}.
\end{cases}
\end{equation*} 
However, we will not push the analysis further and try to derive a rigorous version of theses optimality conditions, it might be the topic of an other study.  

\medskip

In Section \ref{section_superposition_principle}, we answer to a a problem formulated by Brenier \cite[Problem 3.1]{Brenier2003}. The question is the following: if $\mubf : \Omega \to \Prob(D)$, does there exists a probability $Q$ over functions $f : \Omega \to D$ such that $\mubf$ is represented by $Q$, i.e. 
\begin{equation*}
\int_D a(x) \mubf(\xi, \ddr x) = \int a(f(\xi)) Q(\ddr f)
\end{equation*}
for all $a \in C(D)$ continuous and $\xi \in \Omega$; and such that the Dirichlet energy is the mean of the Dirichlet energy of the $f$: 
\begin{equation*}
\Dir(\mubf) = \int \left( \int_\Omega \frac{1}{2} |\nabla f(\xi)|^2 \ddr \xi \right) Q(\ddr f)?
\end{equation*}
If $\Omega$ is a segment of $\R$ the answer is positive as shown in \cite[Section 8.2]{Ambrosio2008} (it is known as the probabilistic representation or the superposition principle). However, as soon as $\Omega$ is two or more dimensional (in fact it already fails if $\Omega$ is a circle), the answer becomes negative. We will provide a counterexample and explain the obstruction. 

The main consequence is the following: there is no Lagrangian formulation for mappings $\mubf : \Omega \to \Prob(D)$. There can be no static formulation of the Dirichlet problem analogue to transport plans or multimarginal formulation. One is forced to work only with the Eulerian formulation, namely the Benamou-Brenier formula. It explains why it is substantially more difficult to study mappings $\mubf : \Omega \to \Prob(D)$ as soon as the dimension of $\Omega$ is larger than $2$, as most of the difficult results of optimal transport are proved thanks to the Lagrangian point of view. 

\bigskip

In Section \ref{section_Ishihara}, we prove a maximum principle (more specifically a Ishihara-type property) for harmonic mappings, meaning roughly speaking that \emph{harmonic mappings reach their maximum on the boundary of the domain $\Omega$}. Of course, there is no canonical order on the Wasserstein space, thus this assertion does not really make sense: only the composition of a (real-valued) geodesically convex function over $\Prob(D)$ with an harmonic mapping will satisfy the maximum principle. In particular, it allows us to give a positive answer to \cite[Conjecture 3.1]{Brenier2003}

If $f : \Omega \to \R$ is a real-valued harmonic function, then $(F \circ f) : \Omega \to \R$ is a subharmonic function for every $F : D \to \R$ convex, which means that $\Delta(F \circ f) \geqslant 0$. It can be checked by a direct computation using the chain rule. If we take $f : X \to Y$, where $X$ and $Y$ are two Riemannian manifolds, then the result still holds (provided that harmonicity, subharmonicity and convexity are properly defined through the Riemannian structures) and it is even a characterization of harmonic mappings: this was first remark by Ishihara \cite{Ishihara1979} (hence the denomination ``Ishihara type property''), one can find a statement and a proof in \cite[Corollary 8.2.4]{Jost2008}. Extensions of this result when the target is a metric space with negative curvature are available, see for instance \cite[Section 7]{Sturm2005}. 

In the Wassertein space, mappings which are convex w.r.t. the metric structure, which means convex along geodesics, are well understood (see for instance \cite[Chapter 9]{Ambrosio2008} or \cite[Chapter 7]{OTAM}). Actually, we will need something a little stronger, which is convexity along generalized geodesics \cite[Section 9.2]{Ambrosio2008}. In our case the Ishihara property reads: if $F : \Prob(D) \to \R$ is convex along generalized geodesics and if $\mubf : \Omega \to \Prob(D)$ is a solution of the Dirichlet problem, then $(F \circ \mubf ) : \Omega \to \R$ is subharmonic (Theorem \ref{theorem_ishihara}). 

The proof of geodesic convexity usually relies on the Lagrangian formulation, which, as we said above, is not available in our case. To overcome this difficulty, we use the approximate Dirichlet energies $\Dir_\varepsilon$ as a substitute for $\Dir$. Indeed, as explained by Jost \cite{Jost1994}, if $\mubf_\varepsilon$ is a minimizer of $\Dir_\varepsilon$ (with for instance fixed values around the boundary $\dr \Omega$), then for a.e. $\xi \in \Omega$, $\mubf_\varepsilon(\xi)$ is a minimizer of 
\begin{equation*}
\nu \mapsto \int_{B(\xi, \varepsilon)} W_2^2(\nu, \mubf_\varepsilon(\eta)) \ddr \eta,
\end{equation*}
in other words $\mubf_\varepsilon(\xi)$ is a barycenter of the $\mubf_\varepsilon(\eta)$, for $\eta \in B(\xi, \varepsilon)$ (for barycenters in the Wasserstein space, see \cite{Agueh2011} for the finite case and \cite{Bigot2012, Kim2017} for the infinite case). Notice that if $f : \Omega \to \R$ is real-valued and harmonic, then for any $\varepsilon > 0$ $f(\xi)$ is the barycenter of $f(\eta)$ for $\eta \in B(\xi, \varepsilon)$, while in the metric case this property only holds asymptotically as $\varepsilon \to 0$. For barycenters in the Wasserstein space, there exists a generalized Jensen inequality: it was already proved in the finite case by Agueh and Carlier \cite[Proposition 7.6]{Agueh2011} under the assumption that $F$ is convex along generalized geodesics, and in a more general case (in particular with an infinite numbers of measures defined on a compact manifold, whereas Agueh and Carlier worked in the Euclidan space) by Kim and Pass \cite[Section 7]{Kim2017}, but with rather strong regularity assumptions on the measures. We reprove this Jensen inequality in a case adapted to our context by letting the barycenter $\mubf_\varepsilon(\xi)$ follow the gradient of the functional $F$ (for gradients flows in the Wasserstein space see \cite{Ambrosio2008}) and use the result as a competitor: through arguments first advanced in \cite{Matthes2009} in a very different context under the name of \emph{flow interchange}, one can show (estimating the derivative of the Wasserstein distance along the flow of $F$ with the so-called (EVI) inequality) that for a.e. $\xi \in \Omega$
\begin{equation}
\label{equation_subharmonic_approximate_intro}
\int_{B(\xi, \varepsilon)} [F(\mubf_\varepsilon(\eta)) - F(\mubf_\varepsilon(\xi))] \ddr \eta \geqslant 0.
\end{equation}  
Then, as $\Dir_\varepsilon$ $\Gamma$-converges to $\Dir$, one knows that $\mubf_\varepsilon$ converges to $\mubf$ a solution of the Dirichlet problem. Passing in the limit \eqref{equation_subharmonic_approximate_intro}, one concludes that $(F \circ \mubf)$ is subharmonic in the sense of distributions. Actually, for technical reasons, we do not minimize exactly the $\varepsilon$-Dirichlet energy, hence using the \emph{flow interchange} leads to an inequality reminiscent of Jensen's one, but not exactly the same.  

Let us make a few comments. The main drawback of the proof, as we proceed by approximation and that uniqueness in the Dirichlet problem is not known, is that we are only able to show subharmonicity of $F \circ \mubf$ for one solution of the Dirichlet problem (which moreover depends on $F$), and not for all. To overcome this limitation, the best thing to do would be to prove uniqueness in the Dirichlet problem. Let us also discuss the regularity that we need on $F$. Either we require $F$ to be continuous (which is very restrictive: it excludes the internal energies); or, if $F$ is only lower semi-continuous, we need $F$ to be bounded on bounded subsets of $L^\infty(D) \cap \Prob(D)$ (which is not very restrictive), but we also need the weak lower semi-continuity of 
\begin{equation*}
\mubf \mapsto \int_\Omega F(\mubf(\xi)) \ddr \xi.
\end{equation*}
More precisely, a mapping $\mubf : \Omega \to \Prob(D)$ can be seen as an element of $\Prob(\Omega \times D)$ (by ``fubinization'') and we require lower semi-continuity of $\mubf \mapsto \int_\Omega (F \circ \mubf)$ w.r.t. the weak convergence on $\Prob(\Omega \times D)$. This weak lower semi-continuity holds heuristically if $F$ is convex for the usual (and not geodesic) convexity on $\Prob(D)$. Indeed, even if the Dirichlet energy has a nice behavior w.r.t. geodesic convexity, the approximate Dirichlet energies $\Dir_\varepsilon$ behave well w.r.t. usual convexity. At the end of the day, the Ishihara property works for potential energies (for a convex, $L^1$ and lower semi-continuous potential), for internal energies (which have a super linear growth and satisfy McCann's conditions) and for the interaction energies (but only for a convex continuous interaction potential). Indeed, for a generic lower semi-continuous potential, the interaction energy $W$ is itself lower semi-continuous on $(\Prob(D), W_2)$, but $\mubf \mapsto \int_\Omega (W \circ \mubf)$ is not. Finally, notice that we do not have the converse statement: we do not know if the fact that $F \circ \mubf$ is subharmonic for any $F$ convex along generalized geodesics is enough to prove that $\mubf$ is harmonic. To prove such a result, one would need a better understanding of the optimality conditions of the Dirichlet problem.  

\bigskip

Finally, we conclude in Section \ref{section_examples} with two examples. 

The first one is the case where the set $D$, on which the target space $\Prob(D)$ is modeled, is a segment of $\R$. In this case, the Wasserstein space $(\Prob(D), W_2)$ is in an isometric bijection with a convex subset of the Hilbert space $L^2([0,1])$. Hence, the Dirichlet problem reduces to the study of the Dirichlet problem for mappings valued in a Hilbert space, which is fairly simple. 

The second one is the case where we restrict our attention to a family of \emph{elliptically contoured distributions}. This terminology comes from \cite{Gelbrich1990} and denotes a generalization of the family of Gaussian measures. In statistics this type of family is sometimes called a \emph{location-scatter family}. More precisely, we take $\rho \in L^1(\R^q)$ a positive and compactly supported function such that the measure $\rho(x) \ddr x$ has a unit mass, zero mean, and the identity matrix as covariance matrix. The family of elliptically contoured distributions built on $\rho$ is nothing else than the sets of measures obtained as image measures from $\rho(x) \ddr x$ by symmetric positive linear transformations. For instance, if $\rho$ is the indicator function of a ball, the family of elliptically contoured distributions built on $\rho$ consists in probability measures uniformly distributed on centered ellipsoids (in general the level sets of the density are ellipsoids, hence the terminology). The Gaussian case would be obtained by taking for $\rho(x) \ddr x$ a centered standard Gaussian, but this probability measure is not compactly supported (recall that we work in $\Prob(D)$ where $D \subset \R^q$ is compact). As in the Gaussian case, the elements of the family of elliptically contoured distributions are parametrized by their covariance matrix. Notice that it is already known that the geodesic between Gaussian measures and more generally the barycenter of Gaussian measures stay in the Gaussian family \cite[Section 6.3]{Agueh2011}. If the boundary values $\mubf : \dr \Omega \to \Prob(D)$ are valued in a family of elliptically contoured distributions, we show that there exists at least one solution of the Dirichlet problem which takes values in the same family everywhere on $\Omega$ (Theorem \ref{theorem_gaussian_existence}). 

Under the additional assumption that the covariance matrices on the boundary $\dr \Omega$ are non singular we are able to show much more (Theorem \ref{theorem_dirichlet_problem_location_scatter}). It implies that there is a solution of the Dirichlet problem with covariance matrices non singular everywhere in $\Omega$: to prove it we use the maximum principle for the Boltzmann entropy, which translates in a minimum principle for the determinant of the covariance matrices. From this we are able to derive the Euler-Lagrange equation satisfied by the covariance matrix. Moreover we can show the uniqueness of the solution to the Dirichlet problem even in the class of mappings not necessarily valued in the family of elliptically contoured distributions. Let us give the structure of the proof as it is almost the only case where we know how to prove uniqueness. The observation is that all solutions of the Dirichlet problem must have the same tangent velocity field. Indeed, if $\varphi$ is a solution of the dual problem, from optimality the tangent velocity field to any solution must be equal to $\nabla_D \varphi$. Now, if the velocity field $\nabla_D \varphi$ is regular enough (namely Lipschitz w.r.t. variables in $D$), then the solution of the ($1$-dimensional) continuity equation with velocity field $\nabla_D \varphi$ is unique. As the (generalized) continuity equation implies the $1$-dimensional one, and as all solutions of the Dirichlet problem coincide on $\dr \Omega$ they must be equal everywhere. In the case of a family of elliptically contoured distributions the tangent velocity field is linear w.r.t. variables in $D$ with some uniform bounds which allow us to make this argument rigorous. If we leave the world of families of elliptically contoured distributions, we do not think that we could get enough regularity on the tangent velocity field for this strategy to work. 

Still under this additional assumption, we are also able to show the regularity of the minimizer: as the problem boils down to the study of Dirichlet minimizing mappings valued in a Riemannian manifold, the only thing to show, following the theory of Schoen and Uhlenbeck \cite{Schoen1982, Schoen1983} is the absence of non-constant tangent minimizing mappings. We prove the latter property with the help of the maximum principle: even though the Wasserstein space is positively curved, there is a lot of functionals convex along geodesics defined on it.

In summary, under the assumption that the covariance matrices on the boundary $\dr \Omega$ are non singular we are able to give a full solution to the problem: existence, uniqueness regularity and Euler-Lagrange equation.

\bigskip

Let us comment on the somehow restrictive framework that we have chosen. The compactness assumption of $\Omega$ and $D$ allows to simplify proofs by avoiding tails estimates: we believe that there is enough technical difficulties and non trivial statements even in this case, and that the key features of the Dirichlet problem are captured, which is the reason why we have restricted ourselves to the compact case. Although we have stuck to the Euclidean case, we see no deep reason which would prevent our definitions and results to be applied to the case where $\Omega$ and $D$ are compact Riemannian manifolds. In particular, our regularization procedures rely on heat flows which are available in Riemannian manifolds. Finally, we have stick to the quadratic Wasserstein distance. We believe that if $p \in (1, + \infty)$ is given, the machinery that we use can be adapted in a straightforward way to define 
\begin{equation*}
\int_\Omega \frac{1}{p} |\nabla \mubf|^p,
\end{equation*}      
where $\mubf : \Omega \to \Prob(D)$ but $\Prob(D)$ is endowed with the $p$-Wassertsein distance. However the Ishihiara type property is related to the Riemannian framework; also the explicit computations in the case of a family of elliptically contoured distributions are no longer avalaible. The case $p=1$ which corresponds the \emph{total variation} of $\mubf : \Omega \to \Prob(D)$ (where $\Prob(D)$ is equipped with the $1$-Wasserstein distance) has been defined and studied very recently \cite{Vogt2017} in the context of image denoising. 

\bigskip

To conclude the introduction, let us explain the connection between the different parts of the paper. If one just wants to understand the definition of the Dirichlet problem, then Subsections \ref{subsection_BB}, \ref{subsection_boundary_values} and Section \ref{section_Dirichlet_problem} are enough. Section \ref{section_superposition_principle}, about the failure of the superposition principle, can be read independently from the rest of the article (except for Subsection \ref{subsection_BB} to get the definition of the objects involved). To have the full proof of the Ishihara property in Section \ref{section_Ishihara}, one needs also to read entirely Section \ref{section_Dirichlet_energy} and Subsections \ref{subsection_Dirichlet_problem}, \ref{subsection_Lipschitz_extension} as some necessary results are proved there. To understand the examples in Section \ref{section_examples}, the reading of Section \ref{section_Dirichlet_energy} is advised.   

\section{Preliminaries}
\label{section_preliminairies}

\subsection{Notations}

Let $p$ and $q$ be two integers larger than $1$. The space $\R^p$ and $\R^q$ are endowed with their Euclidean structure: the scalar product is denoted by $\cdot$ and the norm by $| \cdot| $. The closed ball of center $\xi$ and radius $r$ is denoted by $B(\xi,r)$. We will take $\Omega \subset \R^p$ and $D \subset \R^q$ two \emph{compact} domains, their interior, assumed to be non empty, are denoted by $\Or$ and $\Dr$. The outward normal vector to $\dr \Omega$ (resp. $\dr D$) is denoted by $\nO$ (resp. $\nD$). In general, all elements related to $\Omega$ will be denoted with Greek letters, and those related to $D$ with Latin ones. For instance, points in $\Omega$ (resp. $D$) will be denoted by $\xi, \eta$ (resp. $x,y$), and $(e_\alpha)_{1 \leqslant \alpha \leqslant p}$ (resp. $(e_i)_{1 \leqslant i \leqslant q}$) is the canonical basis of $\R^p$ (resp. $\R^q$). We make the following regularity assumptions:

\begin{asmp*}
We assume that $\Omega$ is a connected compact subset of $\R^p$. Moreover, $\dr \Omega$ is assumed to be Lipschitz, which means that around any point of $\dr \Omega$, up to a rotation, $\Omega$ is the epigraph of a Lipschitz function. 

We assume that $D$ is a convex compact subset of $\R^q$ . 
\end{asmp*}

\noindent Notice that we assume more regularity on $D$ than on $\Omega$. We will consider mappings $\Omega \to \Prob(D)$ with prescribed values on $\dr \Omega$, the regularity of the latter is important. On the contrary, we assume that $D$ is convex, which translates in the fact that $(\Prob(D), W_2)$ is a geodesic space: in some sense, the boundary $\dr D$ of $D$ will be invisible.

The restriction of the Lebesgue measure on $\R^p$ (resp. $\R^q$) to $\Omega$ (resp. $D$) will be denoted by $\Leb_\Omega$ (resp. $\Leb_D$). To avoid normalization constants, we assume that $\Omega$ has unit mass, thus $\Leb_\Omega$ is a probability measure.

If $X$ is a polish space (metric, complete and separable), it is endowed with its Borel $\sigma$-algebra. We define $\Prob(X)$ as the space of Borel positive measure with unit mass. It is endowed with the topology of weak convergence, which means convergence in duality with $C(X)$ the space of continuous bounded and real-valued functions defined on $X$. We also define $\M(X, \R^n)$, for $n \geqslant 1$ as the space of Borel (vectorial) measures valued in $\R^n$ with finite mass, still endowed with the topology of weak convergence. In the case $n=1$, we use the shortcut $\M(X) := \M(X,\R)$. If $\mu \in \Prob(X)$ or $\M(X, \R^n)$, integration w.r.t. $\mu$ is denoted by $\ddr \mu$, or by $\mu(\ddr x)$ if the variable cannot be omitted. If no measure is specified or simply $\ddr \xi$ or $\ddr x$ is used, it means that the integration is performed w.r.t. the Lebesgue measure. If $x \in X$, the Dirac mass at point $x$ is denoted by $\delta_x$. The indicator function of a set $X$ will be denoted by $\1_X$.   

If $T : X \to Y$ is a measurable application between two measurable spaces $X$ and $Y$ and $\mu$ is a measure on $X$, then the image measure of $\mu$ by $T$, denoted by $T \# \mu$, is the measure defined on $Y$ by $(T \# \mu)(B) = \mu(T^{-1}(B))$ for any measurable set $B \subset Y$. It can also be defined by 
\begin{equation*}
\int_Y a(y) (T \# \mu)(\ddr y) := \int_X a(T(x)) \mu(\ddr x), 
\end{equation*} 
this identity being valid as soon as $a : Y \to \R$ is an integrable function \cite[Section 5.2]{Ambrosio2008}.

If $(X, \mu)$ is a measured space and $(Y,d)$ is any metric separable space, $L^2_\mu(X,Y)$ will denote the space of measurable mappings $f : X \to Y$ for which $d(f,y)^2$ integrable w.r.t. $\mu$ for some $y \in Y$. If $Y = \R$, then the letter $Y$ is omitted, and if $\mu$ is the Lebesgue measure, then the letter $\mu$ is omitted. If $Y$ is an Euclidean space, then we set 
\begin{equation*}
\| f \|^2_{L^2_\mu(X, Y)} := \int_X |f(x)|^2 \mu(\ddr x). 
\end{equation*}
If $X$ is an Euclidean space, the space $H^1(\Omega,X)$ is the set of functions $f : \Omega \to X$ such that both $f$ and $\dr_\alpha f$, for $1 \leqslant \alpha \leqslant p$ are in $L^2(\Omega,X)$. 

If $X$ and $Y$ are two subsets of Euclidean spaces, the $L^\infty$ norm of a measurable function $f : X \to Y$ is defined as $\| f \|_\infty := \esssup_{x \in X} |f(x)|$, where the essential supremum is taken w.r.t. the Lebesgue measure.   

If $X$ and $Y$ are two subsets of Euclidean spaces, $C(X,Y)$ and $C^1(X,Y)$ will denote respectively the continuous and $C^1$ functions defined on $X$ and valued in $Y$. If $Y=\R$, then the target space is omitted and we use $C(X)$ or $C^1(X)$. On the space $C^1(\Omega \times D, Y)$ the following differential operators can be defined. The derivatives w.r.t. variables in $\Omega$ will be denoted by $\nabla_\Omega$, or simply $(\dr_\alpha)_{1 \leqslant \alpha \leqslant p}$, and those w.r.t. variables in $D$ by $\nabla_D$, or simply $(\dr_i)_{1 \leqslant i \leqslant q}$. If $X$ is of dimension $1$, the derivative of a function $f$ will be denoted $\dot{f}$. The gradient will be denoted by $\nabla$, and the divergence by $\nabla \cdot$. As an example, if $\varphi \in C^1(\Omega \times D, \R^p)$, with components $(\varphi^{\alpha})_{1 \leqslant \alpha \leqslant p}$, then $\nabla_\Omega \cdot \varphi \in C(\Omega \times D)$ is defined as 
\begin{equation*}
\nabla_\Omega \cdot \varphi(\xi, x) = \sum_{\alpha = 1}^p \dr_\alpha \varphi^\alpha(\xi, x),
\end{equation*}   
for all $\xi \in \Omega$ and $x \in D$; and $\nabla_D \varphi \in C(\Omega \times D, \R^{pq})$ is defined as, for any $\alpha \in \{ 1,2, \ldots, p\}$ and $i \in \{ 1,2, \ldots, q \}$, 
\begin{equation*}
(\nabla_D \varphi)^{\alpha i} (\xi, x) = \dr_i \varphi^\alpha(\xi, x) \in \R.  
\end{equation*}
The notation $C^1_c(\Or \times D, Y)$ will stand for the smooth functions which are compactly supported in $\Or$ but not necessarily in $D$ (and valued in $Y$): if $\varphi \in C^1_c(\Or \times D, Y)$, it means that there exists a compact set $X \subset \Or$ such that $\varphi(\xi, x) = 0$ as soon as $\xi \notin X$.

\subsection{The Wasserstein space}

We recall well known facts about the Wasserstein space. All these results can be found in classical books like \cite{Villani2003, OTAM, Ambrosio2008}.

We endow the space $\Prob(D)$ with the $L^2$-Wasserstein distance $W_2$. If $\mu$ and $\nu$ are elements of $\Prob(D)$, then 
\begin{equation*}
W_2(\mu, \nu) := \sqrt{ \min_\pi \left\{ \iint_{D \times D} |x-y|^2 \pi(\ddr x, \ddr y) \ : \ \pi \in \Pi(\mu, \nu) \right\}},
\end{equation*} 
where $\Pi(\mu, \nu)$ is the set of transport plans, i.e. of probability measures on $D \times D$ which have $\mu$ and $\nu$ as marginals. There exists at least one $\pi \in \Pi(\mu, \nu)$ realizing the infimum, it is called an optimal transport plan. The Wasserstein distance admits a dual formulation which reads 
\begin{equation*}
\frac{W^2_2(\mu, \nu)}{2} = \max_{\varphi, \psi} \left\{ \int_D \varphi(x) \mu(\ddr x) + \int_D \psi(x) \nu(\ddr x) \ : \ \varphi, \psi \in C(D) \text{ and } \forall x,y \in D,  \varphi(x) + \psi(y) \leqslant \frac{|x-y|^2}{2}  \right\}.
\end{equation*} 
Notice that we have inserted a factor $2$, it slightly simplifies the expressions in the sequel. There exists at least one solution of the dual problem, and any pair $(\varphi, \psi)$ which is a solution is called a pair of Kantorovicth potentials. Moreover, if $\mu$ has a density w.r.t. $\Leb_D$ and $(\varphi, \psi)$ is a solution of the dual problem, there exists a unique optimal transport plan $\pi$ and it is given by $\pi = (\Id, \Id - \nabla \varphi) \# \mu$. Notice that thanks to the dual formulation, we see that $W_2^2 : \Prob(D) \times \Prob(D) \to \R$ is the supremum of continuous affine functionals, hence it is convex for the affine structure on $\Prob(D)$ (and continuous by definition). In particular, there is a Jensen's inequality: if $\mubf$ and $\nubf$ are measurable mappings defined on $\Omega$ and valued in $\Prob(D)$, and if $f : \Omega \to \R$ is a weight, i.e. a positive measurable function whose integral is $1$, then 
\begin{equation*}
W^2_2 \left( \int_\Omega  \mubf(\xi) f(\xi) \ddr \xi, \int_\Omega \nubf(\xi) f(\xi) \ddr \xi \right) \leqslant \int_\Omega W_2^2(\mubf(\xi), \nubf(\xi)) f(\xi) \ddr \xi. 
\end{equation*} 
In the formula above the integral $\int_\Omega \mubf f \in \Prob(D)$ is defined according to the affine structure on $\Prob(D)$ for instance by duality: for any $a \in C(D)$, 
\begin{equation}
\label{equation_definition_integral_affine}
\int_D a \ddr \left[ \int_\Omega  \mubf(\xi) f(\xi) \ddr \xi \right] := \int_\Omega \left( \int_D a \ddr \mubf(\xi) \right) f(\xi) \ddr \xi. 
\end{equation}

The space $(\Prob(D), W_2)$ is a metric space whose topology is the one of weak convergence. In particular, according to Prokhorov's theorem, it is a compact separable space. The space $(\Prob(D), W_2)$ is a geodesic space. If $\mu, \nu \in \Prob(D)$ and $\pi \in \Pi(\mu, \nu)$ is an optimal transport plan between $\mu$ and $\nu$, then a constant speed geodesic $\mubf : [0,1] \to \Prob(D)$ joining $\mu$ to $\nu$ is given by $\mubf(t) := f_t \# \pi$ where $f_t : (x,y) \in D \times D \mapsto (1-t)x + ty \in D$ (Notice that we have assumed $D$ to be convex). 

We will briefly use the $1$-Wasserstein distance $W_1$ in the proof of Proposition \ref{proposition_Rellich_approximate}. The definition by duality will be enough: if $\mu$ and $\nu$ are probability measures on $D$, 
\begin{equation*}
W_1(\mu, \nu) := \max_\varphi \left\{ \int_D \varphi(x) \mu(\ddr x) - \int_D \varphi(x) \nu(\ddr x) \ : \ \varphi \in C(D) \text{ and } \varphi \text{ is } 1\text{-Lipschitz} \right\}.
\end{equation*} 
Moreover, as $D$ is compact, there exists a constant $C$ such that $W_2 \leqslant C \sqrt{W_1}$. 

\subsection{Absolutely continuous curves in the Wasserstein space}

A central tool when one is studying the infinitesimal properties of the Wasserstein space is the concept of ($2$-)absolutely continuous curves valued in the Wasserstein space. Let $I$ be a segment of $\R$. A curve $\mubf : I \to \Prob(D)$ is said to be absolutely continuous if there exists $g \in L^2(I)$ such that for any $s < t$ elements of $I$, 
\begin{equation}
\label{equation_IAF_1D}
W_2(\mubf(t), \mubf(s)) \leqslant \int_s^t g(r) \ddr r. 
\end{equation} 
Let us recall the following result, which holds in fact for absolutely continuous curves valued in arbitrary metric spaces, see \cite[Theorem 1.1.2]{Ambrosio2008}. 
\begin{theo}
If $\mubf : I \to \Prob(D)$ is an absolutely continuous curve, then the quantity 
\begin{equation*}
|\dot{\mubf}|(t) := \lim_{h \to 0} \frac{W_2(\mubf(t+h), \mubf(t))}{|h|}
\end{equation*}
exists and is finite for a.e. $t \in I$. Moreover, $|\dot{\mubf}| \leqslant g$ a.e. on $I$ for all $g$ such that \eqref{equation_IAF_1D} holds.
\end{theo}

In the Wasserstein space, absolutely continuous curves are related to solutions of the continuity equation: see \cite[Chapter 8]{Ambrosio2008}. 
\begin{theo}
\label{theorem_AGS_831}
Let $\mubf : I \to \Prob(D)$ an absolutely continuous curve. If $(\vbf_t)_{t \in I}$ is a measurable family such that $\int_I \| \vbf_t \|_{L^2_{\mubf(t)}(D, \R^q)} \ddr t < + \infty$ and the equation $\dr_t \mubf + \nabla_D \cdot(\vbf \mubf) = 0$ is satisfied in a weak sense on $I \times D$ with no-flux boundary conditions on $D$, then one has 
\begin{equation}
\label{equation_metric_derivative_v_1D}
|\dot{\mubf}|(t) \leqslant \sqrt{\int_D |\vbf_t|^2 \ddr \mubf(t)}
\end{equation}
for a.e. $t \in I$. 

Moreover, there exists a unique (for a.e. $t \in I$, $\vbf_t$ is unique $\mubf(t)$ a.e.) family $(\vbf_t)_{t \in I}$ for which equality holds in \eqref{equation_metric_derivative_v_1D} for a.e. $t \in I$. This optimal family is characterized by the fact that for a.e. $t \in I$, there exists a sequence $(\psi_n)_{n \in \N}$ of elements of $C^1(D)$ such that $(\nabla \psi)_{n \in \N}$ converges to $\vbf_t$ in $L^2_{\mubf(t)}(D, \R^q)$.  
\end{theo}

\subsection{Gradient flows}

At some point (in Section \ref{section_Ishihara}) we will need the notion of gradient flow in the Wassertsein space. Roughly speaking, if $F : \Prob(D) \to \R \cup \{ + \infty \}$ is a given functional, a gradient flow is a curve $\mubf : [0, + \infty) \to \Prob(D)$ along which $F$ decreases ``the most'' w.r.t. the Wasserstein distance, in a formal way it can be written 
\begin{equation}
\label{equation_gradient_flow_euclidian}
\frac{\ddr \mubf}{\ddr t}(t) = - \nabla F(\mubf(t)).
\end{equation}
Of course nor the notion of gradient or of time derivative make sense as vectors in the Wassertsein space. In \cite{Ambrosio2008} (see also \cite[Chapter 8]{OTAM}), it is shown how the notion of gradient flow can still be defined through the use of metric quantities only. 

A standard assumption to ensure the existence and uniqueness of a gradient flow with a given value $\mubf(0)$ is that $F$ is convex along generalized geodesic. If $\mu, \nu$ and $\mu_0$ are three probability measures on $D$, one can always build a transport plan $\pi \in \Pi(\mu_0, \mu, \nu) \subset \Prob(D \times D \times D)$ such that the $2$-marginals are optimal transport plans between $\mu_0, \mu$ on the one hand and $\mu_0, \nu$ on the other hand (notice that in general the last $2$-marginal is not an optimal plan between $\mu$ and $\nu$). Then, the generalized geodesic $\mubf : [0,1] \to \Prob(D)$ between $\mu$ and $\nu$ with base point $\mu_0$ is defined as $\mubf(t) := f_t \# \pi$, with $f_t : (x,y,z) \in D^3 \mapsto (1-t) y + tz \in D$. A functional $F : \Prob(D) \to \R \cup \{ + \infty \}$ is said convex along generalized geodesics if for any points $\mu_0, \mu$ and $\nu$, there exists a generalized geodesic $\mubf$ joining $\mu$ to $\nu$ with base point $\mu_0$ such that $F \circ \mubf : [0,1] \to \R \cup \{ + \infty \}$ is a convex function.   

The only result that we will need is called the Evolution Variational Inequality (EVI) formulation of gradient flows (which is a way to make sense of \eqref{equation_gradient_flow_euclidian} in the metric framework). It is summarized in the following theorem, whose proof can be found in \cite[Theorem 11.2.1]{Ambrosio2008}.

\begin{theo}
\label{theorem_EVI}
Let $F : \Prob(D) \to \R \cup \{ + \infty \}$ a functional l.s.c. and convex along generalized geodesics. Then, for any $\mu \in \Prob(D)$ such that $F(\mu) < + \infty$, there exists an absolutely continuous curve $t \in [0, + \infty) \mapsto S^F_t \mu \in \Prob(D)$ such that $S^F_0 \mu = \mu $ and for any $t \geqslant 0$ and any $\nu$ such that $F(\nu) < + \infty$ 
\begin{equation*}
\limsup_{h \to 0, \ h>0} \frac{W_2^2(S^F_{t+h} \mu, \nu) - W^2_2(S_t^F \mu, \nu)}{2h} \leqslant F(\nu) - F(S_t^F \mu).  
\end{equation*}
Moreover, the function $t \mapsto F \left( S^F_t \mu \right)$ is decreasing. 
\end{theo}

\noindent The curve $S^F \mu$ (which can be shown to be unique) is nothing else than the gradient flow of $F$ source form $\mu$.

\subsection{Heat flow}

To regularize probability measures the main tool will be the heat flow. We recall in this subsection classical results that we will use in the sequel. We will denote by $\Phi^D : [0, + \infty) \times \Prob(D) \to \Prob(D)$ the heat flow with Neumann boundary conditions acting on $D$. For a proper definition, one can view $\Phi^D$ as the gradient flow of the Boltzmann entropy, which is convex along generalized geodesics because $D$ is convex, see \cite{Ambrosio2008}. If $\mu \in \Prob(D)$ and $t > 0$, then $\Phi^D_t \mu \in \Prob(D)$ is defined as the measure $u(t,x) \ddr x$ with a density $u : (0, + \infty) \times D \to \R$ which is the solution of the Cauchy Problem
\begin{equation*}
\begin{cases}
\dr_s u(s,x) = \Delta u(s,x) & \text{if } (s,x) \in (0, + \infty) \times \Dr, \\
\nabla u(s,x) \cdot \mathbf{n}_D(x) = 0 & \text{if } (s,x) \in (0, + \infty) \times \dr D, \\
\dst{\lim_{s \to 0}} [ u(s,x) \ddr x] = \mu & \text{in } \Prob(D), 
\end{cases}
\end{equation*}
where $\mathbf{n}_D$ is the outward normal to $D$.

\begin{prop}
\label{proposition_heat_flow_properties}
The heat flow $\Phi^D$ satisfies the following properties: 
\begin{itemize}
\item[(i)] For any $\mu \in \Prob(D)$ and any $t > 0$, the measure $\Phi^D_t \mu$ has a density w.r.t. $\Leb_D$ which is bounded from below by a strictly positive constant and belongs to $C^1(\Dr)$. 

\item[(ii)] For any $t> 0$, the density of $\Phi^D_t \mu$ w.r.t. $\Leb_D$ is bounded in $L^\infty(D)$ by a constant that depends on $t$, but not on $\mu \in \Prob(D)$.

\item[(iii)] For a fixed $t > 0$ and for any $\mu \in \Prob(D)$ and $a \in C(D)$, one has 
\begin{equation*}
\int_D a \ddr \left( \Phi^D_t \mu \right) = \int_D  \left( \Phi^D_t a \right) \ddr \mu. 
\end{equation*} 

\item[(iv)] For any $\mu, \nu \in \Prob(D)$ and any $t \geqslant 0$, 
\begin{equation}
\label{I_equation_heat_flow_contraction}
W_2(\Phi^D_t \mu, \Phi^D_t \nu) \leqslant W_2(\mu, \nu).
\end{equation} 
\end{itemize}
\end{prop}

\begin{proof}
Point (i) is standard interior parabolic regularity. Point (ii) comes from $L^\infty - L^1$ estimates for the Neumann Laplacian, see \cite[Section 7]{Arendt2004}. Point (iii) just states that the heat flow is self-adjoint. Point (iv) comes from the convexity along generalized geodesics of the Boltzmann entropy and the fact that the heat flow is the gradient flow of the latter, see \cite[Example 11.2.4 and Theorem 11.2.1]{Ambrosio2008}. 
\end{proof}

\noindent We mention that except for (iv), all of the statements of Proposition \ref{proposition_heat_flow_properties} remain true if we drop the convexity assumption on $D$, and only assume that $D$ is connected and has a Lipschitz boundary. 

With the help of the last point, we can prove this uniform estimate about the behavior of the heat flow for small values of $t$. 

\begin{prop}
\label{proposition_heat_fow_uniform_0}
There exists a function $\omega^D : [0,+ \infty) \to \R$, continuous and with $\omega(0) = 0$ such that, for any $\mu \in \Prob(D)$ and any $t \geqslant 0$, 
\begin{equation*}
W_2(\Phi^D_t \mu, \mu) \leqslant \omega^D(t).
\end{equation*}
\end{prop}

\begin{proof}
The only thing to check is that $\omega^D$ is continuous in $0$. Assume by contradiction that it is not the case. We can find $(\mu_n)_{n \in \N}$ a sequence in $\Prob(D)$ and $(t_n)_{n \in \N}$ a sequence that tends to $0$ such that, for some $\delta > 0$, there holds $W_2(\Phi^D_{t_n} \mu_n, \mu_n) \geqslant \delta$. Up to extraction, we can assume that $\mu_n$ converges to some limit $\mu$. We can write
\begin{equation*}
W_2(\Phi^D_{t_n} \mu_n, \mu_n) \leqslant W_2(\Phi^D_{t_n} \mu_n, \Phi^D_{t_n} \mu) + W_2(\Phi^D_{t_n} \mu,  \mu) + W_2(\mu, \mu_n) \leqslant  W_2(\Phi^D_{t_n} \mu,  \mu) + 2W_2(\mu, \mu_n), 
\end{equation*}
where we have used the last point of Proposition \ref{proposition_heat_flow_properties}. But then it is clear that the two terms of the r.h.s. tend to $0$, which is a contradiction.
\end{proof}

\section{The Dirichlet energy and the space \texorpdfstring{$H^1(\Omega, \Prob(D))$}{H1(Omega,P(D))}}
\label{section_Dirichlet_energy}

In this section, we define the Dirichlet energy of a function $\mubf \in L^2(\Omega, \Prob(D))$ following the idea of \cite[Section 3]{Brenier2003}. We relate the space of $\mubf$ with finite Dirichlet energy with $H^1(\Omega, \Prob(D))$ using the theory of Sobolev spaces valued into metric spaces of Reshetnyak \cite{Reshetnyak1997, Reshetnyak2004}, and we also prove that this Dirichlet energy coincides with the limit of $\varepsilon$-Dirichlet energies introduced by Korevaar, Schoen and Jost \cite{Korevaar1993, Jost1994}. 

Let us first define the space $L^2(\Omega, \Prob(D))$. As $\Prob(D)$ is bounded, it coincides with the measurable mappings valued in $\Prob(D)$.

\begin{defi}
We denote by $L^2(\Omega, \Prob(D))$ the quotient space of measurable mappings $\mubf : \Omega \to \Prob(D)$ by the equivalence relation of being equal $\Leb_\Omega$-a.e. This space is endowed with the distance $d_{L^2}$ defined by: for any $\mubf$ and $\nubf$ in $L^2(\Omega, \Prob(D))$, 
\begin{equation*}
d_{L^2}^2(\mubf, \nubf) := \int_\Omega W_2^2(\mubf(\xi), \nubf(\xi)) \ddr \xi.
\end{equation*} 
\end{defi}

If $\mubf \in L^2(\Omega, \Prob(D))$, we can define a probability measure on $\Omega \times D$, that we will call temporary $\bar{\mubf}$, in the following way: for any $a \in C(\Omega \times D)$, 
\begin{equation}
\label{equation_definition_disintegration}
\iint_{\Omega \times D} a \ddr \bar{\mubf} := \int_\Omega \left( \int_D a(\xi, \cdot) \ddr \mubf(\xi) \right) \ddr \xi.
\end{equation}
As we have assumed (without any loss of generality) that the Lebesgue measure of $\Omega$ is $1$, the measure $\bar{\mubf}$ is an actual probability measure. If we take a function $a \in C(\Omega)$ which depends only on variables in $\Omega$, one can see that 
\begin{equation}
\label{equation_marginal_equals_Lebesgue}
\iint_{\Omega \times D} a \ddr \bar{\mubf} = \int_\Omega a(\xi) \ddr \xi.
\end{equation}
In other words, the marginal of $\bar{\mubf}$ is the Lebesgue measure $\Leb_\Omega$. We will denote by $\Prob_0(\Omega \times D)$ the subspace of $\Prob(\Omega \times D)$ such that \eqref{equation_marginal_equals_Lebesgue} is satisfied for all $a \in C(\Omega)$. Thanks to the disintegration Theorem \cite[Theorem 5.3.1]{Ambrosio2008}, one can see that, reciprocally, to each $\bar{\mubf} \in \Prob_0(\Omega \times D)$, one can associate a unique element $\mubf$ of $L^2(\Omega, \Prob(D))$ such that \eqref{equation_definition_disintegration} holds. In all the sequel, we will drop the ``bar'' on $\bar{\mubf}$ and use the same letter $\mubf$ to denote an element of $L^2(\Omega, \Prob(D))$ and its counterpart in $\Prob_0(\Omega \times D)$ through the bijection that we have just described. Any $\mubf \in L^2(\Omega, \Prob(D))$ can be seen in two different ways: either as a mapping $\Omega \to \Prob(D)$, or as a probability measure on $\Omega \times D$, and we will very often switch between the two points of view. To clarify the notations:
\begin{itemize}
\item[•] if $\mubf \in L^2(\Omega, \Prob(D))$, then $\mubf(\xi)$ or $\mubf(\xi, \ddr x)$, which is an element of $\Prob(D)$, will denote the mapping $\mubf$ evaluated at $\xi$; 
\item[•] $\mubf(\ddr \xi, \ddr x)$ will indicate that we consider $\mubf$ as an element of $\Prob_0(\Omega \times D)$, integration on $\Omega \times D$ will be denoted by $\ddr \mubf$ or $\mubf(\ddr \xi, \ddr x)$, notice that we have the following relation: $\mubf(\ddr \xi, \ddr x) = \mubf(\xi, \ddr x) \ddr \xi$; 
\item[•] the mapping $\mubf \in L^2(\Omega, \Prob(D))$ is said continuous (resp. Lipschitz) if there is one representative of $\mubf$ such that $W_2(\mubf(\xi), \mubf(\eta))$ goes to $0$ if $\eta \to \xi$ (resp. is bounded by $C|\xi - \eta|$ for some $C < + \infty$).   
\end{itemize}
The topologies on $L^2(\Omega, \Prob(D))$ are defined as follows. 

\begin{defi}
The strong topology on $L^2(\Omega, \Prob(D))$ is the one induced by the distance $d_{L^2}$, and the weak topology is the one induced on $\Prob_0(\Omega \times D)$ by the weak topology on $\Prob(\Omega \times D)$.
\end{defi}

\begin{prop}
W.r.t. the strong topology, $L^2(\Omega, \Prob(D))$ is a polish space. W.r.t. the weak topology, $L^2(\Omega, \Prob(D))$ is a separable compact space. Moreover, the strong topology is finer than the weak topology. 
\end{prop}

\begin{proof}
The statement concerning the strong topology is a consequence of the fact that $\Prob(D)$ is itself a polish space, see for instance \cite[Section 1.1]{Korevaar1993}. As $\Prob_0(\Omega \times D)$ is closed in $\Prob(\Omega \times D)$, for the second statement we simply use the fact that $\Prob(\Omega \times D)$ is itself a separable compact space. 

To compare the topologies we take a sequence $(\mubf_n)_{n \in \N}$ which converges strongly to some $\mubf \in L^2(\Omega, \Prob(D))$. Up to extraction, we know that we can assume that $\mubf_n(\xi)$ converges in $\Prob(D)$ to $\mubf(\xi)$ for a.e. $\xi \in \Omega$. In particular, if $a \in C(\Omega \times D)$, we have that $\int_D a(\xi, \cdot) \ddr \mubf_n(\xi)$ converges to $\int_D a(\xi, \cdot) \ddr \mubf(\xi)$ for a.e. $\xi \in \Omega$. With the help of Lebesgue dominated convergence Theorem, we see that 
\begin{equation*}
\lim_{n \to + \infty} \iint_{\Omega \times D} a \ddr \mubf_n = \lim_{n \to + \infty} \int_\Omega \left( \int_D a(\xi, \cdot) \ddr \mubf_n(\xi) \right) \ddr \xi =  \int_\Omega \left( \int_D a(\xi, \cdot) \ddr \mubf(\xi) \right) \ddr \xi = \iint_{\Omega \times D} a \ddr \mubf.  
\end{equation*} 
As $a$ is arbitrary, this allows us to conclude that $(\mubf_n)_{n \in \N}$ converges to $\mubf$ for the weak topology.
\end{proof}

\subsection{A Benamou-Brenier type definition}
\label{subsection_BB}

We are now ready to define the Dirichlet energy. The first step is to define the (generalized) continuity equation. Recall that $C^1_c(\Or \times D, \R^p)$ is the set of $C^1$ functions defined on $\Omega \times D$ and valued in $\R^p$, whose support is compactly included in $\Or$, but not necessarily in $D$, and $\M(\Omega \times D, \R^{pq})$ denotes the space of vector-valued measures on $\Omega \times D$ with finite mass.  

\begin{defi}
If $\mubf \in L^2(\Omega, \Prob(D))$ and if $\Em \in \M(\Omega \times D, \R^{pq})$, we say that the pair $(\mubf, \Em)$ satisfies the continuity equation if, for every $\varphi \in C^1_c(\Or \times D, \R^p)$, one has 
\begin{equation*}
\iint_{\Omega \times D} \nabla_\Omega \cdot \varphi \ddr \mubf + \iint_{\Omega \times D} \nabla_D \varphi \cdot \ddr \Em = 0. 
\end{equation*}
\end{defi}

\noindent In other words, the pair $(\mubf, \Em)$ satisfies the continuity equation if the equation
\begin{equation*}
\nabla_\Omega \mubf + \nabla_D \cdot \Em = 0.
\end{equation*}
with no-flux boundary conditions on $\dr D$ is satisfied in a weak sense. If we develop in coordinates, it means that for every $\alpha \in \{ 1,2, \ldots, p \}$, one has $\dr_\alpha \mubf + \sum_{i=1}^q \partial_i \Em^{i \alpha} = 0$. If the pair $(\mubf, \Em)$ satisfies the continuity equation, we want to define its Dirichlet energy by $\iint_{\Omega \times D} \frac{|\Em|^2}{2\mubf}$. It is well known in optimal transport that this definition can be made by duality. 

\begin{defi}
\label{definition_dirichlet_mu_E}
If $(\mubf, \Em)$ satisfies the continuity equation, we define its Dirichlet energy $\Dir(\mubf, \Em)$ by
\begin{equation*}
\Dir(\mubf, \Em) := \sup_{a,b} \left\{ \iint_{\Omega \times D} a \ddr \mubf + \iint_{\Omega \times D} b \cdot \ddr \Em  \ : \ (a,b) \in C(\Omega \times D, \mathcal{K}) \right\}, 
\end{equation*}
where $\mathcal{K} \subset \R^{1+pq}$ is the set of pair $(x,y)$ with $x \in \R$ and $y \in \R^{pq}$ such that $x + \frac{1}{2} |y|^2 \leqslant 0$. 
\end{defi}

\noindent Note that $|y|$ is the euclidean norm of $y \in \R^{pq}$. In other words, if $y$ is seen as $p \times q$ matrix, $|y|$ is the Hilbert-Schmidt norm of the matrix. The following proposition is identical to the case of the Benamou-Brenier formula.

\begin{prop}
\label{proposition_v_density_E}
If $(\mubf, \Em)$ satisfies the continuity equation and $\Dir(\mubf, \Em) < + \infty$, then $\Em$ is absolutely continuous w.r.t. $\mubf$, and if $\vbf : \Omega \times D \to \R^{pq}$ is the density of $\Em$ w.r.t. $\mubf$, then one has 
\begin{equation*}
\Dir(\mubf, \Em) = \Dir(\mubf, \vbf \mubf) = \iint_{\Omega \times D} \frac{1}{2} |\vbf|^2 \ddr \mubf.
\end{equation*}
\end{prop}

\begin{proof}
There is nothing to add to the proof of this when $\Omega$ is $1$-dimensional, and such a proof can be found for instance in \cite[Proposition 5.18]{OTAM}. 
\end{proof}

\begin{defi}
\label{definition_dirichlet_mu}
Let $\mubf \in L^2(\Omega, \Prob(D))$. Its Dirichlet energy $\Dir(\mubf)$ is defined by 
\begin{equation*}
\Dir(\mubf) := \inf_{\Em} \left\{ \Dir(\mubf, \Em) \ : \ \Em \in \M(\Omega \times D, \R^{pq}) \text{ and } (\mubf, \Em) \text{ satisfies the continuity equation} \right\}.
\end{equation*}
\end{defi}

\noindent Let us underline that if there exists no $\Em \in \M(\Omega \times D, \R^{pq})$ such that $(\mubf, \Em)$ satisfies the continuity equation, then by convention $\Dir(\mubf) = + \infty$. To be sure that it is written somewhere, let us state the following proposition which identifies the Dirichlet energy if $\Omega$ is a segment of $\R$. It is a consequence of Theorem \ref{theorem_AGS_831} and of the above definitions (see \cite[Theorem 5.28]{OTAM}). 

\begin{prop}
\label{proposition_Benamou_Brenier_1D}
Assume that $I$ is a segment of $\R$ and let $\mubf \in L^2(I, \Prob(D))$. Then $\Dir(\mubf) < + \infty$ if and only if $\mubf$ is absolutely continuous, and in this case
\begin{equation*}
\Dir(\mubf) = \int_I \frac{1}{2} |\dot{\mubf}|^2(t) \ddr t.
\end{equation*}
\end{prop}

Now, let us show easy properties, which are straightforward adaptations on the ones for curves.

\begin{prop}
\label{proposition_existence_optimal_v}
If $\mubf \in L^2(\Omega, \Prob(D))$ is such that $\Dir(\mubf) < + \infty$, then there exists a unique $\Em \in \M(\Omega \times D, \R^{pq})$ such that $(\mubf, \Em)$ satisfies the continuity equation and $\Dir(\mubf) = \Dir(\mubf, \Em)$. 
\end{prop}

\begin{defi}
If $\mubf \in L^2(\Omega, \Prob(D))$ and if $\Em = \vbf \mubf$ is such that $(\mubf, \Em)$ satisfies the continuity equation and $\Dir(\mubf) = \Dir(\mubf, \Em) < + \infty$, then $\Em$ and $\vbf$ are said \emph{tangent} to $\mubf$. 
\end{defi}

\noindent The terminology \emph{tangent} comes from \cite{Ambrosio2008}. As in the case of absolutely continuous curves, there is a characterization of the tangent velocity field $\vbf$ which looks like the one of Theorem \ref{theorem_AGS_831}. 

\begin{prop}
\label{proposition_characterization_v_tangent}
Let $\mubf \in L^2(\Omega, \Prob(D))$ such that $\Dir(\mubf)< + \infty$ and $\vbf \in L^2_{\mubf}(\Omega \times D, \R^{pq})$ such that $(\mubf, \vbf \mubf)$ satisfies the continuity equation. Then $\vbf$ is tangent to $\mubf$ if and only if there exists a sequence $(\psi_n)_{n \in \N}$ in $C^1(\Omega \times D, \R^p)$ such that $(\nabla_D \psi_n)_{n \in \N}$ converges to $\vbf$ in $L^2_{\mubf}(\Omega \times D, \R^{pq})$.
\end{prop}

\begin{proof}[Proof of Proposition \ref{proposition_existence_optimal_v} and Proposition \ref{proposition_characterization_v_tangent}]
In the Hilbert space $L^2_{\mubf}(\Omega \times D, \R^{pq})$ the set $X$ of $\vbf$ such that $(\mubf, \vbf \mubf)$ satisfies the continuity equation is clearly an affine set, and it is not empty as $\Dir(\mubf) < + \infty$. Denoting by $Y = \{ \nabla \psi \ : \ \psi \in C^1(\Or \times D, \R^p) \}$, it is clear that $X$ is parallel to $Y^\perp$.

Thanks to Proposition \ref{proposition_v_density_E}, the problem of calculus of variations in Definition \ref{definition_dirichlet_mu} corresponds to finding the orthogonal projection of the vector $0 \in L^2_{\mubf}(\Omega \times D, \R^{pq})$ on the set of $X$, i.e. Proposition \ref{proposition_existence_optimal_v} is proved. 

It is well known that the projection $\vbf$ is characterized by the fact that $\vbf$ is orthogonal to any vector in the linear space parallel to $X$. In other words, $\vbf$ is characterized (beside the fact that it satisfies the continuity equation) by $\vbf \in X^\perp = (Y^\perp)^\perp$. The latter is nothing else than the closure in $L^2_{\mubf}(\Omega \times D, \R^{pq})$ of $Y$. An easy argument involving cutoff functions shows that this closure is the same as the closure of the set of $\nabla_D \psi$ for $\psi \in C^1(\Omega \times D, \R^p)$, hence Proposition \ref{proposition_characterization_v_tangent} is proved.    
\end{proof}

As an immediate corollary, Proposition \ref{proposition_characterization_v_tangent} implies a localization property: the tangent velocity field $\vbf$, depends only locally on the values of $\mubf$. In the next proposition, $ \mubf |_{\tilde{\Omega}}$ and $\vbf |_{\tilde{\Omega}}$ will denote the restrictions of $\mubf$ and $\vbf$ to a subset $\tilde{\Omega}$ of $\Omega$.  

\begin{crl}
\label{corollary_localization}
Let $\mubf \in L^2(\Omega, \Prob(D))$ such that $\Dir(\mubf) < + \infty$ and let $\vbf \in L^2_{\mubf}(\Omega \times D, \R^{pq})$ be tangent to $\mubf$. Then, if $\tilde{\Omega}$ is any subdomain compactly supported in $\Or$, $\vbf |_{\tilde{\Omega}}$ is tangent to $\mubf |_{\tilde{\Omega}}$.   
\end{crl} 

Still building from Proposition \ref{proposition_characterization_v_tangent}, we can build some sort of dual representation for the Dirichlet energy. Namely, we can say that 
\begin{equation}
\label{equation_dirichlet_energy_dual}
\Dir(\mubf) = \sup_\varphi \left\{ - \iint_{\Omega \times D} \left( \nabla_\Omega \cdot \varphi + \frac{1}{2} |\nabla_D \varphi|^2 \right) \ddr \mubf \ : \ \varphi \in C^1_c(\Or \times D, \R^p) \right\}.
\end{equation}
Indeed, if $\vbf$ is the tangent velocity field to $\mubf$, given the continuity equation and elementary algebra, 
\begin{align*}
- \iint_{\Omega \times D} \left( \nabla_\Omega \cdot \varphi + \frac{1}{2} |\nabla_D \varphi|^2 \right) \ddr \mubf & =  \iint_{\Omega \times D} \left( \nabla_D \varphi \cdot \vbf  - \frac{1}{2} |\nabla_D \varphi|^2 \right) \ddr \mubf \\
& = \Dir(\mubf) - \frac{1}{2} \iint_{\Omega \times D} |\nabla_D \varphi - \vbf|^2 \ddr \mubf.
\end{align*}
Hence the l.h.s. is always smaller than $\Dir(\mubf)$, and we can make the discrepancy arbitrary small thanks to Proposition \ref{proposition_characterization_v_tangent}.

\begin{prop}
\label{proposition_lsc_convex_Dir}
The mapping $\Dir : L^2(\Omega, \Prob(D)) \to \R \cup \{ + \infty \}$ is l.s.c. w.r.t. weak convergence. Moreover it is convex: for any $\mubf$ and $\nubf$ in $L^2(\Omega, \Prob(D))$ and any $t \in [0,1]$, 
\begin{equation*}
\Dir((1-t) \mubf + t \nubf) \leqslant (1-t) \Dir(\mubf) + t \Dir(\nubf).
\end{equation*}
\end{prop}

\begin{proof}
From \eqref{equation_dirichlet_energy_dual}, we see that $\Dir$ is the supremum of linear and continuous (w.r.t. weak convergence) functionals on $L^2(\Omega, \Prob(D))$. Hence it is convex and continuous.
\end{proof}

We will conclude this subsection by showing the following approximation result, which will be useful to prove the equivalences with the metric definitions. We will not be able to regularize up to the boundary of $\Omega$, though it will be sufficient for our purpose. 

\begin{theo}
\label{theorem_approximation}
Fix $\tilde{\Omega} \subset \Or$ compactly embedded in $\Or$. Let $\mubf \in L^2(\Omega, \Prob(D))$ with $\Dir(\mubf) < + \infty$. Then there exists a sequence $\mubf_n \in L^2(\tilde{\Omega}, \Prob(D))$ with the following properties: 
\begin{itemize}
\item[(i)] For any $n \in \N$, $\mubf_n (\ddr \xi, \ddr x) = \rho_n(\xi, x) \ddr \xi \ddr x$, where the density $\rho_n$ of $\mubf_n$ w.r.t. to $\Leb_{\tilde{\Omega}} \otimes \Leb_D$ satisfies $\rho_n \in C^\infty(\tilde{\Omega}, L^\infty(D))$ and $\essinf_{\tilde{\Omega} \times D} \rho_n > 0$.
\item[(ii)] The sequence $(\mubf_n)_{n \in \N}$ converges weakly to $\mubf$ in $L^2(\tilde{\Omega}, \Prob(D))$.
\item[(iii)] There holds
\begin{equation*}
\lim_{n \to + \infty} \Dir( \mubf_n ) = \Dir(  \mubf |_{\tilde{\Omega}} ).
\end{equation*}
\end{itemize}
\end{theo}

\noindent Notice that $\mubf_n$ is defined only on $\tilde{\Omega}$, i.e. not on the full domain $\Omega$. 

\begin{proof}
On $\Omega$, we will regularize with a convolution kernel $\chi$. Specifically, we fix $\chi : \R^p \to [0,1]$ a smooth function, radial, compactly supported in $B(0,1)$ and of total integral $1$, and we set $\chi_n(\xi) = n^p \chi(n \xi )$. On the other hand, on $D$ we will regularize with the heat flow that we denote by $\Phi^D$. We set $\tilde{\mubf}_n(\xi) := [\Phi^D_{1/n}] [\mubf(\xi)]$ for any $\xi \in \Omega$. Hence $\tilde{\mubf}_n \in L^2(\Omega, \Prob(D))$ is defined on the whole $\Omega$. For $n$ large enough and $\xi \in \tilde{\Omega}$ we define 
\begin{equation*}
\mubf_n (\xi) := \int_{\Omega} \chi_n(\xi - \eta) \tilde{\mubf}_n(\eta) \ddr \eta,
\end{equation*}   
where here we do the usual (Euclidean) mean of probability measures. In short, $\mubf_n = \chi_n \star_\Omega\tilde{\mubf}_n$. We need $n$ such that the support of $\chi_n$ is small compared to the distance between $\tilde{\Omega}$ and $\dr \Omega$. 

Assertion (i) holds because of the regularization properties of the convolution and the lower bound on the solution of the heat flow.  

Assertion (ii) is standard: if we fix $a \in C(\tilde{\Omega} \times D)$, given the self-adjacency of the heat flow and the symmetry of the heat kernel, 
\begin{equation*}
\iint_{\tilde{\Omega} \times D} a \ddr \mubf_n = \iint_{\tilde{\Omega} \times D} \Phi^D_{1/n} [ \chi_n \star_\Omega a ] ~ \ddr \mubf
\end{equation*} 
and the r.h.s. converges strongly to the integral of $a$ against $\mubf$ because of standard functional analysis. 

Assertion (iii) is slightly trickier. As we have already seen in Proposition \ref{proposition_heat_flow_properties}, applying the heat flow decreases the Wasserstein distance. Provided we admit the representation given below by Theorem \ref{theorem_equivalence_KS} and the contraction property of the heat flow, it is straightforward that we should have $\Dir(\tilde{\mubf}_n) \leqslant \Dir(\mubf)$. But the current theorem will be used to prove Theorem \ref{theorem_equivalence_KS}, hence we cannot invoke it. We adopt a different strategy: we start with the ``dual'' representation for the Dirichlet energy given by \eqref{equation_dirichlet_energy_dual}. We want to show that $\Dir(\tilde{\mubf}_n) \leqslant \Dir(\mubf)$. For any fixed $\varphi \in C^1_c(\Or \times D, \R^p)$, and given that the heat flow is self-adjoint, 
\begin{equation*}
\iint_{\Omega \times D} \left( \nabla_\Omega \cdot \varphi + \frac{1}{2} |\nabla_D \varphi|^2 \right) \ddr \tilde{\mubf}_n = \iint_{\Omega \times D} \left( \nabla_\Omega \cdot (\Phi^D_{1/n}\varphi) + \Phi^D_{1/n} \left( \frac{1}{2} |\nabla_D \varphi|^2 \right) \right) \ddr \mubf.
\end{equation*}
Notice that we used the property that the heat flow acting on $D$ commutes with $\nabla_\Omega \cdot$. Now, the key point is the so-called Bakery-\'Emery estimate 
\begin{equation*}
\frac{1}{2} \left| \nabla \left( \Phi^D_{1/n} \varphi \right) \right|^2 \leqslant \Phi^D_{1/n} \left( \frac{1}{2} |\nabla_D \varphi|^2 \right)
\end{equation*}
which is valid because $D$ is a convex domain \cite[Equation (2.4)]{Gladbach2018}. Hence
\begin{equation*}
\iint_{\Omega \times D} \left( \nabla_\Omega \cdot \varphi + \frac{1}{2} |\nabla_D \varphi|^2 \right) \ddr \tilde{\mubf}_n \geqslant \iint_{\Omega \times D} \left( \nabla_\Omega \cdot (\Phi^D_{1/n}\varphi) + \frac{1}{2} \left| \nabla \left( \Phi^D_{1/n} \varphi \right) \right|^2 \right) \ddr \mubf \geqslant - \Dir(\mubf),
\end{equation*} 
where the last inequality comes from \eqref{equation_dirichlet_energy_dual}. Taking the supremum in $\varphi$ and using the representation formula \eqref{equation_dirichlet_energy_dual} we conclude that $\Dir(\tilde{\mubf}_n) \leqslant \Dir(\mubf)$. Now we want to control the Dirichlet energy of $\mubf_n$ with the one of $\tilde{\mubf}_n$. Recall that $\Dir$ is a convex function. But $\mubf_n$ is the average, w.r.t. to the weights $\chi_n(\eta)$, of the mappings $\xi \mapsto  \tilde{\mubf}_n(\xi - \eta)$. Hence, by Jensen's inequality, 
\begin{equation*}
\Dir( \mubf_n ) \leqslant \int_{B(0,1/n)} \chi_n(\eta) \Dir\left( \left. \tilde{\mubf}_n \right|_{\tilde{\Omega}} (\cdot - \eta) \right). 
\end{equation*}
Hence, calling $\Omega_n$ the set of points which are distant at most $1/n$ from $\tilde{\Omega}$, one has $\Dir(\mubf_n) \leqslant \Dir( \tilde{\mubf}_n |_{\Omega_n} )$. Sending $n$ to $+ \infty$ and using the lower semi-continuity of $\Dir$ and assertion (ii) to get the reverse inequality, we get (iii).
\end{proof}

\subsection{The smooth case}
\label{subsection_smooth_case}

In this subsection, we will briefly study the smooth case, i.e. the one where $\mubf$ has a smooth and strictly positive density w.r.t. $\Leb_\Omega \otimes \Leb_D$. It will help us to understand the meaning of the continuity equation and we will use it in the sequel when reasoning by approximation. 

\begin{defi}
\label{definition_smooth}
A mapping $\mubf \in L^2(\Omega, \Prob(D))$ with $\Dir(\mubf)< + \infty$ is said smooth if it admits a density $\rho$ w.r.t. $\Leb_\Omega \otimes \Leb_D$ satisfying $\rho \in C^\infty(\Omega, L^\infty(D))$ and uniformly bounded from below.
\end{defi}

\noindent In particular, it implies $\rho$ is uniformly bounded (from above) on the closed set $\Omega$. Notice that Theorem \ref{theorem_approximation} says that any $\mubf \in L^2(\Omega, \Prob(D))$ with finite Dirichlet energy can be approximated by a sequence of smooth functions (only in the interior of $\Omega$) according to Definition \ref{definition_smooth}. Let us start by explaining how, in the smooth case, one can compute the tangent velocity field. 

\begin{prop}
\label{proposition_smooth_velocity_def}
Let $\mubf \in L^2(\Omega, \Prob(D))$ be smooth. Then, for every $\xi \in \Or$, there exists a unique $\varphi(\xi, \cdot) \in H^1(D,\R^p)$ with $0$-mean solution to the elliptic equation 
\begin{equation}
\label{equation_elliptic}
\begin{cases}
\nabla_D \cdot (  \rho(\xi,\cdot) \nabla_D \varphi(\xi,\cdot)  ) = - \nabla_\Omega \rho(\xi,\cdot) & \text{in } \Dr \\
\nabla_D \varphi(\xi, \cdot ) \cdot \nD = 0 & \text{on } \dr D.
\end{cases}
\end{equation}
Moreover $\nabla_D \varphi \in L^2_{\mubf}(\Omega \times D, \R^{pq})$ is the tangent velocity field to $\mubf$ and it is continuous as a mapping from $\Omega$ to $L^2(D, \R^{pq})$.
\end{prop}

\begin{proof}
The existence of a unique solution to the elliptic equation \eqref{equation_elliptic} derives from standard arguments. Notice that $\nabla_\Omega \rho(\xi, \cdot)$ has always $0$-mean on $D$, hence the equation is well-posed. In particular, as $\rho$ is bounded from below, the equation is uniformly elliptic. We have the usual estimate 
\begin{equation*}
\| \nabla_D \varphi(\xi, \cdot) \|_{L^2(D, \R^{pq})} \leqslant C \| \nabla_D \varphi(\xi, \cdot) \|_{L^2_{\rho(\xi, \cdot)}(D, \R^{pq})} \leqslant C \| \nabla_\Omega \rho(\xi, \cdot) \|_\infty,
\end{equation*} 
which tells us that $\nabla_D \varphi(\xi, \cdot)$ is uniformly bounded (w.r.t. $\xi$) in $L^2(D, \R^{pq})$. By construction, $\vbf := \nabla_D \varphi$ is such that $(\mubf, \vbf \mubf)$ satisfies the continuity equation. 

To prove continuity of $\xi \mapsto \nabla_D \varphi(\xi, \cdot)$, let us fix $\xi \in \Or$ and a sequence $\xi_n$ which converges to $\xi$. We use momentarily the compact notations $\bar{\varphi} = \varphi(\xi, \cdot) \in H^1(D, \R^{pq})$ and $\varphi_n = \varphi(\xi_n, \cdot) \in H^1(D, \R^{pq})$. Similarly, we set $\bar{\rho} = \rho(\xi, \cdot)$ and $\rho_n = \rho(\xi_n, \cdot)$. The r.h.s. of the elliptic equations will be $\bar{h} = -\nabla_\Omega \rho(\xi, \cdot)$ and $h_n = -\nabla_\Omega \rho(\xi_n, \cdot)$. We want to show that $\varphi_n$ converges to $\bar{\varphi}$ in $H^1(D, \R^{pq})$, while we know that $\bar{\rho}, \rho_n$ are uniformly bounded from below and above, and that $\rho_n$ (resp. $h_n$) converges to $\bar{\rho}$ (resp. $\bar{h}$) in $L^\infty(D)$. Clearly, $\varphi_n - \bar{\varphi}$ satisfies the elliptic equation 
\begin{equation*}
\nabla_D \cdot ( \bar{\rho} \nabla_D ( \varphi_n - \bar{\varphi} )  ) = h_n - \bar{h} + \nabla_D \cdot ( (\rho_n - \bar{\rho}) \nabla_D  \varphi_n   )
\end{equation*} 
with Neumann boundary conditions. Testing this equation against $\varphi_n - \bar{\varphi}$, we deduce that 
\begin{equation*}
\| \nabla_D( \varphi_n - \bar{\varphi}  ) \|_{L^2(D, \R^{pq})} \leqslant C \left( \| h_n - \bar{h} \|_{L^2(D)} + \| \rho_n - \bar{\rho} \|_\infty \| \nabla_D \varphi_n \|_{L^2(D, \R^{pq})} \right).
\end{equation*}
We can use the convergence of $\rho_n$ to $\bar{\rho}$, $h_n$ to $\bar{h}$ and the fact that $\| \nabla_D \varphi_n \|_{L^2(D, \R^{pq})}$ is uniformly bounded in $n$ to conclude that the l.h.s. goes to $0$ as $n \to + \infty$. 
\end{proof}

Now take $\mubf \in L^2(\Omega, \Prob(D))$ smooth and denote by $\vbf = \nabla_D \varphi$ its tangent velocity field. If $\gamma : I \to \Or$ is a smooth curve going from an interval of $\R$ to $\Or$, then, multiplying \eqref{equation_elliptic} by $\dot{\gamma}$, one can see that $\mubf^\gamma = \mubf \circ \gamma : I \to \Prob(D)$ defines a curve valued in the Wasserstein space for which the (classical) continuity equation $\dr_t \mubf^\gamma + \nabla \cdot (\vbf^\gamma \mubf^\gamma)$ with Neumann boundary conditions is satisfied (at least in a weak sense), provided that we define $\vbf^\gamma := \vbf \cdot \dot{\gamma} : I \times D \to \R^q$. More precisely, if $i \in \{1,2, \ldots, q \}$, the $i$-the component of $\vbf^\gamma$ at time $t \in I$ and at the point $x \in D$ is 
\begin{equation*}
\left(\vbf^\gamma(t,x) \right)^i = \sum_{\alpha = 1}^p \vbf(\gamma(t),x)^{i \alpha} \dot{\gamma}^\alpha(t).
\end{equation*} 
In other words, the (generalized) continuity equation implies that we get (classical) continuity equation for every curve of $\Omega$. In some sense, the (generalized) continuity equation is much stronger in higher dimensions.

As we recalled previously, the velocity field $\vbf^\gamma$ is related to the metric derivative of the curve $\mubf^\gamma$ in the Wasserstein space. As the tangent velocity field $\vbf \in L^2(\Omega \times D, \R^{pq})$ is the gradient of a function $\nabla_D \varphi$, by Proposition \ref{proposition_characterization_v_tangent} $\vbf^\gamma$ is the tangent velocity field to the curve $\mubf^\gamma$. Using Theorem \ref{theorem_AGS_831}, we see that for all $s \leqslant t \in I$, 
\begin{equation}
\label{equation_control_distance_curve}
\frac{W^2_2(\mubf(\gamma(t), \mubf(\gamma(s)) )}{t-s} \leqslant \int_s^t \int_D |\vbf^\gamma(r,x)|^2 \mubf(\gamma(r), \ddr x) \ddr r. 
\end{equation} 
But in fact, we can say more and go from a global estimate to a local one, this is the object of the following proposition.

\begin{prop}
\label{proposition_upper_gradient_smooth}
Let $\mubf \in L^2(\Omega, \Prob(D))$ be smooth and let $\vbf \in C(\Omega, L^2(D, \R^{pq}))$ its tangent velocity field. 

Then the function $\mubf$ is Lipschitz. Moreover, if $\xi \in \Or$ and $\eta \in \R^p$, 
\begin{equation}
\label{equation_limsup_smooth_case}
\lim_{\varepsilon \to 0} \frac{W_2(\mubf(\xi + \varepsilon \eta), \mubf(\xi))}{|\varepsilon|} = \sqrt{\int_D |\vbf(\xi,x) \cdot \eta|^2 \mubf(\xi, \ddr x)}. 
\end{equation} 
\end{prop}

\noindent The important point of this proposition is that the estimate holds for \emph{all} points of $\Omega$, there is no ``almost everywhere'' in the statement. 

\begin{proof}
We fix $\xi \in \Or$ and use $\gamma(t) := \xi + t \eta$ which is defined for $t$ sufficiently close to $0$. Notice that $\vbf^\gamma(t,x) = \vbf(\xi + t \eta, x) \cdot \eta$. 

We denote by $\rho$ the density of $\mubf$ w.r.t. $\Leb_\Omega \times \Leb_D$. To prove that $\mubf$ is Lipschitz, we use \eqref{equation_control_distance_curve} and the fact that $\rho \in L^\infty(\Omega \times D)$ and $\vbf \in L^\infty(\Omega, L^2(D, \R^{pq}))$.

The fact the l.h.s. of \eqref{equation_limsup_smooth_case} (provided the $\lim$ is replaced by a $\limsup$) is bounded by the r.h.s. comes directly from \eqref{equation_control_distance_curve} and the continuity of $\vbf : \Omega \to L^2(D, \R^{pq})$.

To prove the reverse inequality, take a sequence $(\varepsilon_n)_{n \in \N}$ realizing the $\liminf$ for the l.h.s. of \eqref{equation_limsup_smooth_case}. Call $\psi_n \in C^0(D)$ the function with $0$-mean such that $\varepsilon_n \psi_n$ is the Kantorovich potential from $\mubf(\xi)$ to $\mubf(\xi + \varepsilon_n \eta)$, it is unique because $\mubf(\xi)$ is supported on the whole $D$, see for instance \cite[Proposition 7.18]{OTAM}. As $\Id - \varepsilon_n \nabla_D \psi_n$ is the optimal transport map from $\mubf(\xi)$ onto $\mubf(\xi + \varepsilon_n \eta)$, there holds
\begin{equation*}
\varepsilon_n \| \nabla_D \psi_n \|_{L^2_{\mubf(\xi)}(D)} =  W_2(\mubf(\xi), \mubf(\xi + \varepsilon_n \eta )) \leqslant C \varepsilon_n,
\end{equation*}
where $C$ is the Lipschitz constant of $\mubf$. In particular, using the lower bound on $\rho$, one sees that, up to a subsequence, $(\psi_n)_{n \in \N}$ converges weakly in $H^1_{\mubf(\xi)}(D)$ to some function $\psi$ such that
\begin{equation*}
\sqrt{\int_D | \nabla_D \psi(x)|^2 \mubf(\xi, \ddr x)} =  \| \nabla_D \psi \|_{L^2_{\mubf(\xi)}(D)} \leqslant \liminf_{n \to + \infty} \frac{W_2(\mubf(\xi), \mubf(\xi + \varepsilon_n e ))}{\varepsilon_n}.
\end{equation*}
Thus, to conclude, it is enough to show that $\nabla_D \psi = -\vbf(\xi, \cdot) \cdot \eta$.

As $\Id - \varepsilon_n \nabla_D \psi_n$ transports $\mubf(\xi)$ onto $\mubf(\xi + \varepsilon_n \eta)$, for any $f \in C^1(D)$, one has 
\begin{equation*}
\int_D f(x - \varepsilon_n \nabla_D \psi_n(x)) \rho(\xi, x) \ddr x = \int_D f(x) \rho(\xi + \varepsilon_n \eta, x) \ddr x.
\end{equation*} 
Using a Taylor expansion on $f$ and dividing by $\varepsilon_n$, 
\begin{equation*}
\left| \int_D \nabla_D \psi_n(x) \cdot \nabla_D f(x) \rho(\xi,x) \ddr x + \int_D \frac{\rho(\xi +  \varepsilon_n \eta ,x) - \rho(\xi,x)}{\varepsilon_n} f(x) \ddr x \right| \leqslant C \varepsilon_n \int_D |\nabla_D \psi_n(x)|^2 \ddr x,
\end{equation*}
where the constant $C$ is a bound on the second derivative of $f$. Using the $H^1$ bound on $\psi_n$ and the weak convergence to $\psi$, as well as the fact that $\rho$ is differentiable w.r.t. variables in $\Omega$, we conclude that $\psi$ solves weakly the elliptic equation
\begin{equation*}
\nabla_D \cdot ( \rho(\xi, \cdot) \nabla_D \psi  ) = - \nabla_\Omega \rho(\xi, \cdot) \cdot \eta.
\end{equation*}
Using the uniqueness (recall that $\psi_n$ has $0$-mean, hence $\psi$ too) result for equation \eqref{equation_elliptic}, this allows to conclude that $\nabla_D \psi = \vbf(\xi, \cdot) \cdot \eta$ where $\vbf$ is the tangent velocity field to $\mubf$, hence the proposition is proved. 
\end{proof}

\subsection{Equivalence with Sobolev spaces valued in metric spaces} 
\label{subsection_equivalence_Sobolev_metric}

Until now, we have not discussed the existence of solutions to the (generalized) continuity equation: this notion could be too strong or too loose. In this subsection, we will show that the set of $\mubf$ with finite Dirichlet energy coincides with an already known definition of Sobolev spaces valued in metric spaces given by Reshetnyak \cite{Reshetnyak1997, Reshetnyak2004}. This definition is restricted to the case where the source space has a smooth structure (which is precisely our framework), but can be seen as particular case of a more general definition given by Haj\l asz (a pedagogic and clear introduction to the latter can be found in \cite[Chapter 5]{Ambrosio2003}).

\begin{defi}
\label{definition_sobolev_R}
Let $\mubf \in L^2(\Omega, \Prob(D))$. For any $\nu \in \Prob(D)$, define $[\mubf]_\nu \in L^2(\Omega)$ by $[\mubf]_\nu(\xi) := W_2(\mubf(\xi), \nu)$. We say that $\mubf \in H^1(\Omega, \Prob(D))$ if there exists a countable family $(\nu_n)_{n \in \N}$ dense in $\Prob(D)$ such that $[\mubf]_{\nu_n} \in H^1(\Omega)$ for all $n \in \N$ and there exists a function $g \in L^2(\Omega)$ such that, for every $n \in \N$, 
\begin{equation}
\label{equation_definition_Sobolev_R}
|\nabla [\mubf]_{\nu_n}| \leqslant g 
\end{equation}
a.e. on $\Omega$. The smallest $g$ for which \eqref{equation_definition_Sobolev_R} holds is called the metric gradient of $\mubf$ and is denoted by $g_{\mubf}$.  
\end{defi}  

\noindent Notice that $g_{\mubf} = \sup_n |\nabla [\mubf]_{\nu_n}|$. The definition looks slightly different than in \cite{Reshetnyak1997}. However, it is equivalent because of the following result: 
\begin{prop}
Let $\mubf \in H^1(\Omega, \Prob(D))$ and $g_{\mubf} \in L^2(\Omega)$ be its metric gradient. Then for all mappings $u : \Prob(D) \to \R$ which are $C$-Lipschitz, $u \circ \mubf \in H^1(\Omega)$ and $|\nabla (u \circ \mubf)| \leqslant C g_{\mubf}$ a.e. on $\Omega$.
\end{prop} 

\begin{proof}
Is is enough to copy the proof of \cite[Theorem 5.1]{Reshetnyak1997}. Indeed, in this proof, one only uses the functions $[\mubf]_\nu$ for measures $\nu$ belonging to a dense and countable subset of $\Prob(D)$.  
\end{proof}

\noindent In particular, if $\mubf \in H^1(\Omega, \Prob(D))$, then $[\mubf]_\nu \in H^1(\Omega)$ with gradient bounded by $g_{\mubf}$ for all $\nu \in \Prob(D)$. Notice that the definition above can be stated for mappings valued in arbitrary metric spaces (separability of the target space is required). The main theorem of this subsection is the following, which states that the framework that we have developed coincides with the one of Reshetnyak. 

\begin{theo}
\label{theorem_equivalence_sobolev_metric}
Let $\mubf \in L^2(\Omega, \Prob(D))$. Then $\mubf \in H^1(\Omega, \Prob(D))$ if and only if $\Dir(\mubf)<+\infty$. Moreover, if $\mubf \in H^1(\Omega, \Prob(D))$ and if $\vbf$ is tangent to $\mubf$, then for a.e. $\xi \in \Omega$, 
\begin{equation*}
g_{\mubf}(\xi) \leqslant \sqrt{\int_D |\vbf(\xi, x)|^2 \mubf(\xi, \ddr x)} \leqslant \sqrt{p} g_{\mubf}(\xi). 
\end{equation*}
\end{theo} 

The inequalities are sharp. The function $g_{\mubf}$ measures the norm of the gradient of $\mubf$ as an operator norm, whereas the norm of the velocity field $\vbf$ is measured with an Hilbert-Schmidt norm, which explains the discrepancy, see \cite{Chiron2007} for a more detailed discussion. 

We will prove this theorem in three steps. The first one is to prove it if $\Omega$ is a segment of $\R$ (Proposition \ref{proposition_equivalence_sobolev_1D}). It is just a rewriting of the definition of Reshetnyak and does not rely of the special structure of the Wasserstein space. The second step is to say that, roughly speaking, a function is in $H^1(\Omega)$ if it is in $H^1$ for a.e. lines, with some uniform control on the gradients. It enables us to get the result if $\Omega$ is a cube (Proposition \ref{proposition_equivalence_sobolev_cube}). The third step is simply to write that every domain can be written as a (countable) union of cubes.

\begin{prop}
\label{proposition_equivalence_sobolev_1D}
Theorem \ref{theorem_equivalence_sobolev_metric} holds if $\Omega$ is a segment of $\R$.
\end{prop}

\begin{proof}
Assume $\Omega = I$ is a segment of $\R$. The set of curves with finite Dirichlet energy coincides with the set of absolutely continuous curves, see Proposition \ref{proposition_Benamou_Brenier_1D}. Given Theorem \ref{theorem_AGS_831}, we want to prove the equality $g_{\mubf} = |\dot{\mubf}|$ a.e. on $I$. 

Assume that $\Dir(\mubf) < + \infty$ and take $\nu \in \Prob(D)$. Then, as $W_2(\cdot, \nu)$ is $1$-Lipschitz, for all $s < t$ elements of $I$, 
\begin{equation*}
|[\mubf]_\nu(t) - [\mubf]_\nu(s)| \leqslant W_2(\mubf(t), \mubf(s)) \leqslant \int_s^t |\dot{\mubf}|(r) \ddr r.
\end{equation*} 
It shows that the function $[\mubf]_\nu$ is in $H^1(I)$ and its gradient is smaller than $|\dot{\mubf}|$. Hence, as $\nu$ is arbitrary, $\mubf \in H^1(I, \Prob(D))$ and $g_{\mubf} \leqslant |\dot{\mubf}|$.  

Reciprocally, assume $\mubf \in H^1(I, \Prob(D))$, take $(\nu_n)_{n \in \N}$ countable and dense in $\Prob(D)$ such that $[\mubf]_{\nu_n} \in H^1(I)$ for every $n \in \N$ with gradient bounded by $g_{\mubf}$. In particular, for any $n \in \N$ and any $s < t$ elements of $I$, 
\begin{equation*}
|[\mubf]_{\nu_n}(t) - [\mubf]_{\nu_n}(s)|\leqslant \int_s^t g_{\mubf}(r) \ddr r. 
\end{equation*}
Then we choose $\nu_n$ arbitrary close to $\mubf(t)$: the r.h.s. is unchanged and the l.h.s. is arbitrary close to $W_2(\mubf(t), \mubf(s))$. Hence we conclude that 
\begin{equation*}
W_2(\mubf(s), \mubf(t)) \leqslant \int_t^s g_{\mubf}(r) \ddr r,
\end{equation*}
which is enough to say that $\mubf$ is an absolutely continuous curve and $|\dot{\mubf}| \leqslant g_{\mubf}$ a.e. on $I$ by minimality of $|\dot{\mubf}|$.  
\end{proof}

Now we will prove Theorem \ref{theorem_equivalence_sobolev_metric} at least locally, which means in the case where $\Omega$ is a cube. Up to an isometry and a dilatation, we can assume that $\Omega$ is the unit cube of $\R^p$. Recall that $(e_\alpha)_{1 \leqslant \alpha \leqslant p}$ is the canonical basis of $\R^p$. In the sequel, we will denote by $\Omega_\alpha \subset \R^p$ the $\alpha$-face of the cube, which means the set of $(\xi^1, \ldots, \xi^{\alpha - 1}, 0, \xi^{\alpha + 1}, \ldots, \xi^p)$, with $0 \leqslant \xi^\beta \leqslant 1$ for all $\beta \neq \alpha$. The measure on $\Omega_\alpha$ will be the $p-1$-dimensional Lebesgue measure. If $f : \Omega \to X$ is a given mapping (where $X$ is any set) and $\xi \in \Omega_\alpha$ is fixed, then $f_\xi : [0,1] \to X$ is defined by $f_\xi(t) = f(\xi + t e_\alpha)$: it is the restriction of $f$ to a line directed by $e_\alpha$ and crossing $\Omega_\alpha$ at $\xi$. Recall the following characterization for real-valued mappings: 

\begin{prop}
\label{proposition_sobolev_line}
Assume $\Omega$ is the unit cube of $\R^p$ and let $f \in L^2(\Omega)$ be a given function. The function $f$ belongs to $H^1(\Omega)$ if and only if for any $\alpha \in \{ 1,2, \ldots, p \}$, for a.e. $\xi \in \Omega_\alpha$, the function $f_{\xi}$ is in $H^1([0,1])$ and 
\begin{equation*}
\int_{\Omega_\alpha} \left( \int_0^1 |\dot{f}_\xi(t)|^2 \ddr t \right) \ddr \xi < + \infty.
\end{equation*}
Moreover, for a.e. $\xi \in \Omega_\alpha$ and a.e. $t \in [0,1]$, 
\begin{equation*}
(\dr_\alpha f)(\xi + t e_\alpha) = \dot{f}_\xi(t). 
\end{equation*}
\end{prop}

\begin{proof}
One can look at \cite[Section 4.9]{Evans1992}.  
\end{proof}

\begin{prop}
\label{proposition_equivalence_sobolev_cube}
Theorem \ref{theorem_equivalence_sobolev_metric} holds if $\Omega$ is the unit cube of $\R^p$.
\end{prop}

\begin{proof}

\emph{Implication $\Dir(\mubf) < + \infty \Rightarrow \mubf \in H^1(\Omega, \Prob(D))$}. Assume first that $\mubf \in L^2(\Omega, \Prob(D))$ is such that $\Dir(\mubf) < + \infty$ and take $\vbf \in L^2_{\mubf}(\Omega \times D, \R^{pq})$ the velocity field tangent to $\mubf$. Fix $\alpha \in \{ 1,2, \ldots, p \}$. Take two compactly supported test functions $\psi \in C^1_c(]0,1[ \times D)$ and $a \in C^1_c(\Omega_\alpha)$. As a test function $\varphi \in C^1_c(\Omega \times D, \R^p)$ in the weak formulation of the continuity equation, choose $\varphi(\xi + t e_\alpha, x) := (0, 0, \ldots, 0, \psi(t, x) a(\xi), 0, \ldots, 0)$ for $\xi \in \Omega_\alpha$ and $t \in [0,1]$ (only the $\alpha$-th component of $\varphi$ is not $0$). If we expand we find that $\nabla_\Omega \cdot \varphi = a \dr_t \psi$ hence 
\begin{align*}
0 = \iint_{\Omega \times D}  \nabla_\Omega \cdot \varphi \ddr \mubf + \iint_{\Omega \times D} \nabla_D & \varphi \cdot  \vbf \ddr \mubf  = \int_{\Omega_\alpha} \left( \iint_{[0,1] \times D} \dr_t \psi(t,x) \ddr t \mubf(\xi + t e_\alpha, \ddr x)  \right) a(\xi) \ddr \xi \\
& + \int_{\Omega_\alpha} \left( \iint_{[0,1] \times D} \nabla_D \psi(t,x) \cdot (\vbf(\xi + t e_\alpha, x) \cdot e_\alpha) \ddr t \mubf(\xi + t e_\alpha, \ddr x)  \right) a(\xi) \ddr \xi. 
\end{align*}  
Using the arbitrariness of $a$, we deduce that for a.e. $\xi \in \Omega_\alpha$, and for a fixed $\psi \in C^1_c(]0,1[ \times D, \R^p)$, 
\begin{equation}
\label{equation_continuity_1D}
\iint_{[0,1] \times D} \dr_t \psi(t,x) \ddr t \mubf(\xi + t e_\alpha, \ddr x) + \iint_{[0,1] \times D} \nabla_D \psi(t,x) \cdot (\vbf(\xi + t e_\alpha, x) \cdot e_\alpha) \ddr t \mubf(\xi + t e_\alpha, \ddr x) = 0.  
\end{equation}
Now, taking a sequence $(\psi_n)_{n \in \N}$ which is dense in $C^1_c(]0,1[ \times D, \R^p)$, we can say that for a.e. $\xi \in \Omega_\alpha$, for all $\psi \in C^1_c(]0,1[ \times D, \R^p)$, \eqref{equation_continuity_1D} holds. For $\xi \in \Omega_\alpha$ define $\mubf_\xi : [0,1] \to \Prob(D)$ by $\mubf_\xi(t) = \mubf(\xi + t e_\alpha)$ and $\vbf_\xi : [0,1] \times D \to \R^q$ by $\vbf_\xi(t,x) = \vbf(\xi + t e_\alpha,x) \cdot e_\alpha$. By Fubini's theorem, for a.e. $\xi \in \Omega_\alpha$, $\vbf_\xi \in L^2_{\mubf_\xi}([0,1] \times D, \R^q)$. Hence \eqref{equation_continuity_1D} rewrites as: for a.e. $\xi \in \Omega_\alpha$, the curve $\mubf_\xi$ is an absolutely continuous curve in the Wasserstein space with a velocity field given by $\vbf_{\xi}$. By Proposition \ref{proposition_equivalence_sobolev_1D}, if $\nu \in \Prob(D)$, then the function $[\mubf_\xi]_{\nu}$ is in $H^1([0,1])$ and 
\begin{equation*}
|\dr_t [\mubf_\xi]_{\nu}(t)| \leqslant \sqrt{\int_D |\vbf_\xi(t,x)|^2 \mubf_\xi(t, \ddr x)} = \sqrt{\int_D |\vbf(\xi + t e_\alpha,x) \cdot e_\alpha|^2 \mubf(\xi + t e_\alpha, \ddr x)}.
\end{equation*} 
As the r.h.s. is integrable over $[0,1] \times \Omega_\alpha$ and $\alpha$ is arbitrary, we can use Proposition \ref{proposition_sobolev_line} to see that $[\mubf]_\nu \in H^1(\Omega)$. Moreover, taking the square of the previous equation and summing over $\alpha \in \{ 1,2, \ldots, p \}$, we see that for a.e. $\xi \in \Omega$ 
\begin{equation*}
|\nabla [\mubf]_\nu(\xi)|^2 \leqslant \int_D |\vbf(\xi, x)|^2 \mubf(\xi, \ddr x). 
\end{equation*}
Thus, we conclude that $\mubf \in H^1(\Omega, \Prob(D))$ and for a.e. $\xi \in \Omega$, 
\begin{equation}
\label{equation_zz_aux_17}
g_{\mubf}(\xi) \leqslant \sqrt{\int_D |\vbf(\xi, x)|^2 \mubf(\xi, \ddr x)}. 
\end{equation}

\medskip

\emph{Implication $\mubf \in H^1(\Omega, \Prob(D)) \Rightarrow \Dir(\mubf) < + \infty$}. Let $\mubf \in H^1(\Omega, \Prob(D))$. Take $(\nu_n)_{n \in \N}$ a sequence which is dense in $\Prob(D)$. For any $n \in \N$, the function $[\mubf]_{\nu_n}$ belongs to $H^1(\Omega)$. Fix $\alpha \in \{ 1,2, \ldots, p \}$. For any $n \in \N$, for a.e. $\xi \in \Omega_\alpha$, the function $[\mubf_\xi]_{\nu_n} : t \mapsto W_2(\mubf(\xi + t e_\alpha), \nu_n)$ is in $H^1([0,1])$ with a gradient bounded by $g_{\mubf}(\xi + t e_\alpha)$. As $\N$ is countable, we can exchange the ``for a.e. $\xi \in \Omega_\alpha$'' and the ``for all $n \in \N$''. Hence, for a.e. $\xi \in \Omega_\alpha$, the function $\mubf_\xi : [0,1] \to \Prob(D)$ belongs to $H^1([0,1], \Prob(D))$ with a gradient bounded by $g_{\mubf}(\xi + t e_\alpha)$. For a given $\xi \in \Omega_\alpha$, we can use Proposition \ref{proposition_equivalence_sobolev_1D} and Theorem \ref{theorem_AGS_831} to get the existence of a velocity field $\wbf_\xi^\alpha \in L^2_{\mubf_\xi}([0,1] \times D, \R^q)$ such that $(\mubf_\xi, \wbf_\xi^\alpha \mubf_\xi)$ satisfies the ($1$-dimensional) continuity equation and for a.e. $t \in [0,1]$, 
\begin{equation}
\label{equation_zz_aux_15}
\sqrt{\int_D |\wbf^\alpha_\xi(t,x)|^2 \mubf(\xi + t e_\alpha, \ddr x)} \leqslant |\dot{\mubf}_\xi(t)|=  |g_{\mubf_\xi}(t)| \leqslant g_{\mubf}(\xi + t e_\alpha). 
\end{equation} 
Now, do this for a.e. $\xi \in \Omega_\alpha$ and then for any $\alpha \in \{1,2, \ldots, p \}$. Define the velocity field $\vbf : \Omega \times D \to \R^{pq}$ component by component, the $\alpha$-th component at the point $\xi + t e_\alpha$ (with $\xi \in \Omega_\alpha$) being defined as $\wbf_\xi^\alpha(t)$. To justify that $\vbf$ is measurable, notice that $\wbf^\alpha_\xi$ is the solution of an optimization problem \cite[Equation (8.3.11)]{Ambrosio2008} which depends in a measurable way of $\xi$, thus one can apply Proposition \ref{proposition_Aliprantis1819}. By the bound \eqref{equation_zz_aux_15}, it is clear that $\vbf \in L^2_{\mubf}(\Omega \times D, \R^{pq})$. Moreover, if $\varphi \in C^1_c(\Omega \times D, \R^p)$, 
\begin{align*}
\iint_{\Omega \times D} \nabla_\Omega \cdot \varphi \ddr \mubf & = \sum_{\alpha = 1}^p \iint_{\Omega \times D} \dr_\alpha \varphi^\alpha(\xi, x) \mubf(\ddr \xi, \ddr x) \\
& = \sum_{\alpha = 1}^p \int_{\Omega_\alpha} \left( \int_0^1 \dr_\alpha \varphi^\alpha(\xi + t e_\alpha, x) \mubf(\xi + t e_\alpha, \ddr x) \ddr t \right) \ddr \xi \\
& = - \sum_{\alpha = 1}^p \int_{\Omega_\alpha} \left( \int_0^1 \nabla_D \varphi^\alpha(\xi + t e_\alpha, x) \cdot \wbf^{\alpha}_\xi(t,x) \mubf(\xi + t e_\alpha, \ddr x) \ddr t \right) \ddr \xi \\
& = - \sum_{\alpha = 1}^p \int_{\Omega_\alpha} \left( \int_0^1 \nabla_D \varphi^\alpha(\xi + t e_\alpha, x) (\vbf(\xi + t e_\alpha, x) \cdot e_\alpha) \mubf(\xi + t e_\alpha, \ddr x) \ddr t \right) \ddr \xi \\
& = - \iint_{\Omega \times D} \nabla_D \varphi \cdot \vbf \ddr \mubf. 
\end{align*} 
(The second and last inequalities are Fubini's theorem and the third one comes from the $1$-dimensional continuity equations). Hence, we see that $(\mubf, \vbf \mubf)$ satisfies the continuity equation.

To conclude, we need to show a control of $\vbf$ by $g_{\mubf}$. If $\alpha \in \{ 1,2, \ldots, p \}$, for a.e. $\xi \in \Omega_\alpha$ and a.e. $t \in [0,1]$, one has, by definition of $g_{\mubf_\xi}$ and Proposition \ref{proposition_equivalence_sobolev_1D},
\begin{equation*}
\sqrt{\int_D |\wbf^\alpha_\xi(t,x)|^2 \mubf(\xi + t e_\alpha, \ddr x)} =  g_{\mubf_\xi}(t) = \sup_{n \in \N} \left| \dr_\alpha [\mubf]_{\nu_n}(\xi + t e_\alpha) \right|,  
\end{equation*} 
which can be rewritten as: for a.e. $\xi \in \Omega$, for all $\alpha \in \{ 1,2, \ldots, p \}$, 
\begin{equation}
\label{equation_zz_aux_20}
\sqrt{\int_D |\vbf (\xi,x) \cdot e_\alpha|^2 \mubf(\xi, \ddr x)} = \sup_{n \in \N} \left|  \nabla [\mubf]_{\nu_n}(\xi) e_\alpha \right| \leqslant g_{\mubf}(\xi). 
\end{equation}
Squaring, summing on $\alpha$ and taking the square root, we see that for a.e. $\xi \in \Omega$
\begin{equation*}
\sqrt{\int_D |\vbf(\xi,x)|^2 \mubf(\xi, \ddr x)} \leqslant \sqrt{p} g_{\mubf}(\xi).
\end{equation*}
Even though one could prove that $\vbf$ is the tangent velocity field (using the fact that the $\wbf^\alpha$ are and the characterization given in Proposition \ref{proposition_characterization_v_tangent}), it is enough to use Corollary \ref{corollary_localization} to see that the l.h.s. is a.e. larger than the $L^2_{\mubf(\xi)}(D,\R^{pq})$-norm of the tangent velocity field.
\end{proof}

\noindent To conclude the proof of the theorem, we just have to justify that we can put the pieces together. 

\begin{proof}[Proof of Theorem \ref{theorem_equivalence_sobolev_metric}]
The domain $\Or$ can be cut in a (countable) number of cubes $(\Omega_m)_{m \in \N}$. The boundary $\dr \Omega$ does not play any role as $\Leb_\Omega(\dr \Omega) = 0$. 

\emph{Implication $\Dir(\mubf) < + \infty \Rightarrow \mubf \in H^1(\Omega, \Prob(D))$}. Assume first that $\mubf \in L^2(\Omega, \Prob(D))$ is such that $\Dir(\mubf) < + \infty$ and take $\vbf \in L^2_{\mubf}(\Omega \times D, \R^{pq})$ the velocity field tangent to $\mubf$. Fix $n \in \N$. On each cube $\Omega_m$, we know that the function $[\mubf]_{\nu_n}$ is in $H^1(\Omega_m)$ with a gradient which is bounded by a function which does not depend on $n$ and is in $L^2(\Omega)$, which is sufficient to say that $[\mubf]_{\nu_n} \in H^1(\Omega)$ with a gradient bounded by a function which does not depend on $n \in \N$. 

\medskip

\emph{Implication $\mubf \in H^1(\Omega, \Prob(D)) \Rightarrow \Dir(\mubf) < + \infty$}. Assume that $\mubf \in H^1(\Omega, \Prob(D))$. For any $m \in \N$, one can construct a tangent velocity field $\vbf \in L^2_{\mubf}(\Omega_m \times D, \R^{pq})$. Combining Proposition \ref{proposition_existence_optimal_v} giving the uniqueness $\mubf$-a.e. of the tangent velocity field and Corollary \ref{corollary_localization} which enables to localize, one sees that if $\Omega_{m_1} \cap \Omega_{m_2} \neq \emptyset$, then the tangent velocity fields $\vbf_1 \in L^2_{\mubf}(\Omega_{m_1} \times D, \R^{pq})$ and $\vbf_2 \in L^2_{\mubf}(\Omega_{m_2} \times D, \R^{pq})$ coincide $\mubf$-a.e. on $\Omega_{m_1} \cap \Omega_{m_2}$. Thus, one can define a velocity field $\vbf$ on the whole $\Omega$, and it is straightforward to check that $\vbf$ is tangent to $\mubf$.   
\end{proof}

\subsection{Equivalence with Dirichlet energy in metric spaces}
\label{subsection_equivalence_Dirichlet_metric}

In this subsection we will show that our definition coincides with the one of Korevaar, Schoen, and Jost \cite{Korevaar1993, Jost1994}. As explained in the introduction, their formulation is related to the following object.

\begin{defi}
Let $\varepsilon > 0$ and $\mubf \in L^2(\Omega, \Prob(D))$. We define the $\varepsilon$-Dirichlet energy of $\mubf$ by 
\begin{equation*}
\Dir_\varepsilon(\mubf) := C_p \iint_{\Omega \times \Omega} \frac{W_2^2(\mubf(\xi), \mubf(\eta))}{2\varepsilon^{p+2}} \1_{|\xi - \eta| \leqslant \varepsilon} \ddr \xi \ddr \eta,
\end{equation*}
where the normalization constant $C_p$ is defined as $C_p := |\eta|^2 \left( \int_{B(0,1)} |\xi \cdot \eta|^2 \ddr \xi \right)^{-1}$.
\end{defi} 

One can notice that the $\varepsilon$-Dirichlet energy is always finite as $\Prob(D)$ has a finite diameter, but it can blow up when $\varepsilon \to 0$. The goal is to prove that $\Dir_\varepsilon$ is a good approximation of $\Dir$ if $\varepsilon$ is small enough. Before stating the main result, let us do the following observation, which will be useful in the sequel. 

\begin{prop}
\label{proposition_Dir_eps_lsc}
Let $\varepsilon > 0$ be fixed. Then the functional $\Dir_\varepsilon : L^2(\Omega, \Prob(D)) \to \R$ is continuous w.r.t. strong convergence and l.s.c. w.r.t. the weak convergence.
\end{prop}

\begin{proof}
The continuity w.r.t. strong convergence is simple: recall that $\Prob(D)$ has a finite diameter, thus Lebesgue dominated convergence theorem is enough. The lower semi-continuity relies on the fact that $W_2^2$ is a supremum of continuous linear functionals, thus is l.s.c. and convex. 

More precisely, fix $\mubf \in L^2(\Omega, \Prob(D))$ and a sequence $(\mubf_n)_{n \in \N}$ which converges weakly to $\mubf$ in $L^2(\Omega, \Prob(D))$. If $\xi$ and $\eta$ are points of $\Omega$, take $(\varphi(\xi, \eta, \cdot), \psi(\xi, \eta, \cdot))$ a pair of Kantorovitch potential between $\mubf(\xi)$ and $\mubf(\eta)$. In other words, $\varphi(\xi, \eta, \cdot)$ and $\psi(\xi, \eta, \cdot)$ are continuous functions (in fact uniformly Lipschitz), such that $\varphi(\xi, \eta, x) + \psi(\xi, \eta, y) \leqslant |x-y|^2/2$ for any $x,y \in D$, and such that 
\begin{equation}
\label{equation_zz_aux_6}
\frac{W_2^2(\mubf(\xi), \mubf(\eta))}{2} = \int_D \varphi(\xi, \eta, x) \mubf(\xi, \ddr x) + \int_D \psi(\xi, \eta, y) \mubf(\eta, \ddr y). 
\end{equation}
One can do that in such a way that $\varphi : \Omega \times \Omega \to C(D)$ and $\psi : \Omega \times \Omega \to C(D)$ are measurable. Indeed, for fixed $\xi$ and $\eta$, $(\varphi(\xi, \eta, \cdot), \psi(\xi, \eta, \cdot)) \in C(D) \times C(D)$ is a maximizer a functional which is continuous on $C(D) \times C(D)$ and which depends on $\xi$ and $\eta$ in a measurable way: hence we can apply Proposition \ref{proposition_Aliprantis1819}. Then, using the double convexification trick (see \cite[Section 2.1]{Villani2003}) which is a measurable operation, we can assume that $(\varphi, \psi)$ are uniformly (w.r.t. $\xi$ and $\eta$) Lipschitz and bounded as elements of $C(D)$. By the Kantorovitch duality, for every $n \in \N$, 
\begin{equation}
\label{equation_zz_aux_5}
\Dir_\varepsilon(\mubf_n) \geqslant \frac{C_p}{\varepsilon^{p+2}} \iint_{\Omega \times \Omega} \1_{|\xi - \eta| \leqslant \varepsilon}  \left( \int_D \varphi(\xi, \eta, x) \mubf_n(\xi, \ddr x) + \int_D \psi(\xi, \eta, y) \mubf_n(\eta, \ddr y) \right) \ddr \xi \ddr \eta.
\end{equation} 
Now, apply Lusin's theorem to the mapping $\varphi : \Omega \times \Omega \to C(D)$ (for Lusin's theorem to other spaces than $\R$, see for instance \cite[Box 1.6]{OTAM}). For any $\delta > 0$, we can find a compact $X \subset \Omega \times \Omega$ such that $\Leb_\Omega \otimes \Leb_\Omega ([\Omega \times \Omega] \bsl X) \leqslant \delta$ and $\varphi : X \to C(D)$ is continuous on $X$. Now notice, as $|\varphi(\xi, \eta, x) - \varphi(\xi, \eta, y)| \leqslant C|x-y|$ uniformly in $\xi$ and $\eta$, that $\varphi : X \times D \to \R$ is a continuous function for the product topology on $X \times D \subset \Omega \times \Omega \times D$. This function can be extended in a function $\tilde{\varphi} \in C(\Omega \times \Omega \times D)$. To summarize, there exists a continuous function $\tilde{\varphi}$, which coincides with $\varphi$ on $X \times D$ (the important point is that there is coincidence on all $D$). Thus, denoting by $C$ a uniform bound of $\varphi$ and $\tilde{\varphi}$, one has that for every $\nubf \in L^2(\Omega, \Prob(D))$,  
\begin{equation}
\label{equation_zz_aux4}
\left| \iint_{\Omega \times \Omega} \1_{|\xi - \eta| \leqslant \varepsilon}  \left( \int_D \varphi(\xi, \eta, x) \nubf(\xi, \ddr x) \right) \ddr \xi \ddr \eta - \iint_{\Omega \times \Omega} \1_{|\xi - \eta| \leqslant \varepsilon}  \left( \int_D \tilde{\varphi}(\xi, \eta, x) \nubf(\xi, \ddr x) \right) \ddr \xi \ddr \eta \right| \leqslant C \delta. 
\end{equation} 
On the other hand, using Fubini's theorem one sees that
\begin{equation*}
\iint_{\Omega \times \Omega} \1_{|\xi - \eta| \leqslant \varepsilon}  \left( \int_D \tilde{\varphi}(\xi, \eta, x) \mubf_n(\xi, \ddr x) \right) \ddr \xi \ddr \eta = \iint_{\Omega \times D} \left( \int_{B(\xi, \varepsilon) \cap \Omega} \tilde{\varphi}(\xi, \eta, x) \ddr \eta \right) \mubf_n(\ddr \xi, \ddr x).
\end{equation*}
As $\tilde{\varphi}$ is continuous and bounded, it is not difficult to see that 
\begin{equation*}
(\xi, x) \in \Omega \times D \mapsto \int_{B(\xi, \varepsilon) \cap \Omega} \tilde{\varphi}(\xi, \eta, x) \ddr \eta \in \R
\end{equation*}
is continuous. Hence, using the weak convergence of $(\mubf_n)_{n \in \N}$, 
\begin{equation*}
\lim_{n \to + \infty} \iint_{\Omega \times \Omega} \1_{|\xi - \eta| \leqslant \varepsilon}  \left( \int_D \tilde{\varphi}(\xi, \eta, x) \mubf_n(\xi, \ddr x) \right) \ddr \xi \ddr \eta = \iint_{\Omega \times \Omega} \1_{|\xi - \eta| \leqslant \varepsilon}  \left( \int_D \tilde{\varphi}(\xi, \eta, x) \mubf(\xi, \ddr x) \right) \ddr \xi \ddr \eta. 
\end{equation*}
Using equation \eqref{equation_zz_aux4} with both $\mubf_n$ and $\mubf$ as $\nubf$, and using moreover the arbitrariness of $\delta$, we conclude that we can replace $\tilde{\varphi}$ by $\varphi$ in the equation above: 
\begin{equation*}
\lim_{n \to + \infty} \iint_{\Omega \times \Omega} \1_{|\xi - \eta| \leqslant \varepsilon}  \left( \int_D \varphi(\xi, \eta, x) \mubf_n(\xi, \ddr x) \right) \ddr \xi \ddr \eta = \iint_{\Omega \times \Omega} \1_{|\xi - \eta| \leqslant \varepsilon}  \left( \int_D \varphi(\xi, \eta, x) \mubf(\xi, \ddr x) \right) \ddr \xi \ddr \eta. 
\end{equation*}
Of course there is exactly the same statement with $\psi$. With the help of this information, combining \eqref{equation_zz_aux_5} and \eqref{equation_zz_aux_6}, we reach the conclusion that  
\begin{align*}
\liminf_{n \to + \infty} & \ \Dir_\varepsilon(\mubf_n) \\
& \geqslant \liminf_{n \to + \infty} \frac{C_p}{\varepsilon^{p+2}} \iint_{\Omega \times \Omega} \1_{|\xi - \eta| \leqslant \varepsilon}  \left( \int_D \varphi(\xi, \eta, x) \mubf_n(\xi, \ddr x) + \int_D \psi(\xi, \eta, y) \mubf_n(\eta, \ddr y) \right) \ddr \xi \ddr \eta \\
& = \frac{C_p}{\varepsilon^{p+2}} \iint_{\Omega \times \Omega} \1_{|\xi - \eta| \leqslant \varepsilon}  \left( \int_D \varphi(\xi, \eta, x) \mubf(\xi, \ddr x) + \int_D \psi(\xi, \eta, y) \mubf(\eta, \ddr y) \right) \ddr \xi \ddr \eta \\
& = \Dir_\varepsilon(\mubf). \qedhere
\end{align*}

\end{proof}

We are now ready to state and prove the main theorem of this subsection.

\begin{theo}
\label{theorem_equivalence_KS}
Let $\mubf \in L^2(\Omega, \Prob(D))$. Then $\Dir_\varepsilon(\mubf)$ converges to $\Dir(\mubf)$ as $\varepsilon \to 0$, and the sequence $(\Dir_{2^{-n} \varepsilon_0} (\mubf))_{n \in \N}$ is increasing for any $\varepsilon_0 > 0$. 

In addition for any $\varepsilon_0 > 0$, $\Dir_{2^{-n} \varepsilon_0}$ $\Gamma$-converges to $\Dir$ on the space $L^2(\Omega, \Prob(D))$ endowed with the weak topology as $n \to + \infty$. 
\end{theo}

\noindent In the case of a smooth mapping $\mubf$, the equivalence will directly derives from Proposition \ref{proposition_upper_gradient_smooth}. The difficulty of the proof is to study the behavior of $\Dir_\varepsilon$ w.r.t. approximations.

\begin{proof}

\emph{Monotonicity of $\Dir_\varepsilon$}. If $\mubf \in L^2(\Omega, \Prob(D))$, $\varepsilon > 0$ and $\lambda \in (0,1)$ then one has 
\begin{equation*}
\Dir_\varepsilon(\mubf) \leqslant \lambda \Dir_{\lambda \varepsilon}(\mubf) + (1 - \lambda) \Dir_{(1-\lambda) \varepsilon}(\mubf).
\end{equation*}
Indeed, this is a consequence of the triangle inequality and is valid for mappings valued in arbitrary metric spaces, see for instance \cite[Example 1) (i)]{Jost1994} or \cite[Equation (8.3.4)]{Jost2008} for a proof. In particular, by taking $\lambda = 1/2$, we see that the sequence $(\Dir_{2^{-n} \varepsilon_0} (\mubf))_{n \in \N}$ is increasing for any $\varepsilon_0 > 0$. Moreover, with well chosen $\lambda$, one sees that for a fixed $\mubf \in L^2(\Omega, \Prob(D))$ the function $\varepsilon \mapsto \varepsilon \Dir_\varepsilon(\mubf)$ is subadditive, which is enough to ensure the convergence of $\Dir_\varepsilon(\mubf)$ to some limit in $[0, + \infty]$ as $\varepsilon \to 0$.  

\medskip

\emph{The smooth case}. Let $\mubf \in H^1(\Omega, \Prob(D))$ be smooth in the sense of Definition \ref{definition_smooth}. Let $\vbf$ be its tangent velocity field, by Proposition \ref{proposition_smooth_velocity_def}, there holds $\vbf \in C(\Omega, L^2(D, \R^{pq}))$. We will show that the limit of $\Dir_\varepsilon(\mubf)$ is equal to $\Dir(\mubf)$. Indeed, one can write 
\begin{equation*}
\Dir_\varepsilon(\mubf) = \int_{\Omega} \mathrm{dir}_\varepsilon(\xi) \ddr \xi,
\end{equation*}  
where 
\begin{equation*}
\mathrm{dir}_\varepsilon(\xi) := C_p \int_{\Omega \cap B(\xi, \varepsilon)} \frac{W_2^2(\mubf(\xi), \mubf(\eta))}{2\varepsilon^{p+2}} \ddr \eta.
\end{equation*}
If $\xi \notin \dr \Omega$ (it happens for a.e. $\xi$), for $\varepsilon$ small enough, $B(\xi, \varepsilon) \subset \Omega$ and we can perform the following change of variables in spherical coordinates: denoting by $\Sph^{p-1}$ the unit sphere of $\R^p$ and $\sigma$ its surface measure, 
\begin{equation*}
\mathrm{dir}_\varepsilon(\xi) = \frac{C_p}{2} \int_{\Sph^{d-1}} \left( \int_0^1 \frac{W_2^2(\mubf(\xi), \mubf(\xi + r \varepsilon \theta))}{\varepsilon^2} r^{p-1} \ddr r \right) \sigma(\ddr \theta). 
\end{equation*}
Thanks to Proposition \ref{proposition_upper_gradient_smooth} we have the pointwise limit of the integrand, and we can pass to the limit as $\varepsilon \to 0$: recall that $\mubf$ is Lipschitz, which gives a uniform bound from above of the Wasserstein distances. Hence, for a.e. $\xi \in \Omega$, 
\begin{align*}
\lim_{\varepsilon \to 0} \mathrm{dir}_\varepsilon(\xi) & = \frac{C_p}{2} \int_{\Sph^{d-1}} \left[ \int_0^1 \left( \int_D |\vbf(\xi, x) \cdot (r \theta)|^2 \mubf(\xi, \ddr x) \right) r^{p-1}  \ddr r \right] \sigma(\ddr \theta) \\
& = \frac{C_p}{2} \int_D \left( \int_{B(0,1)} |\vbf(\xi, x) \cdot \eta|^2 \ddr \eta \right) \mubf(\xi, \ddr x) \\
& = \frac{1}{2} \int_D |\vbf(\xi, x)|^2 \mubf(\xi, \ddr x),
\end{align*} 
where the last inequality comes from the definition of $C_p$. To integrate this equality over $\Omega$, we still use the fact that $\mubf$ is Lipschitz to get the appropriate bounds, hence 
\begin{equation*}
\lim_{\varepsilon \to 0} \Dir_\varepsilon(\mubf)  
= \int_\Omega \left( \lim_{\varepsilon \to 0} \mathrm{dir}_\varepsilon(\xi) \right) \ddr \xi 
=  \int_\Omega \left(  \int_D \frac{1}{2} |\vbf(\xi, x)|^2 \mubf(\xi, \ddr x) \right) \ddr \xi 
= \Dir(\mubf, \vbf \mubf) = \Dir(\mubf). 
\end{equation*}

\medskip

\emph{General case: $\lim_\varepsilon \Dir_\varepsilon \leqslant \Dir$}. Let $\mubf \in H^1(\Omega, \Prob(D))$. As $\mubf$ is in $H^1$ in the sense of Reshetnyak, and using the main result of \cite{Reshetnyak2004}, we know that $l := \lim_\varepsilon \Dir_\varepsilon(\mubf)$ is finite. It implies, thanks to the theory of Korevaar and Schoen \cite[Theorem 1.10]{Korevaar1993}, that the so-called energy density is absolutely continuous w.r.t. $\Leb_\Omega$ which means $\lim_\varepsilon \Dir_\varepsilon(\mubf)$ does not decrease too much if we restrict $\mubf$ to a domain $\tilde{\Omega}$ slightly smaller than $\Omega$. More precisely, it implies that for any $\delta$ there exists $\tilde{\Omega}$ compactly embedded in $\Or$ such that, for some $\varepsilon_0$ small enough, 
\begin{equation*}
l - \delta \leqslant \Dir_{\varepsilon_0}( \mubf |_{\tilde{\Omega}}  ) \leqslant l. 
\end{equation*}  
Let $(\mubf_n)_{n \in \N}$ the sequence of elements of $L^2(\tilde{\Omega}, \Prob(D))$ given by Theorem \ref{theorem_approximation}. We choose $n$ large enough so that $\Dir(\mubf_n) \leqslant \Dir(\mubf |_{\tilde{\Omega}}) + \delta$ and $\Dir_{\varepsilon_0}(\mubf_n) \geqslant \Dir_{\varepsilon_0}(\mubf|_{\tilde{\Omega}}) - \delta$: it is possible because $\Dir_{\varepsilon_0}$ is lower semi-continuous w.r.t. weak convergence on $L^2(\tilde{\Omega}, \Prob(D))$. Hence, 
\begin{equation*}
l \leqslant \Dir_{\varepsilon_0}(\mubf|_{\tilde{\Omega}}) + \delta \leqslant \Dir_{\varepsilon_0}(\mubf_n) + 2 \delta \leqslant  \Dir(\mubf_n) + 2 \delta \leqslant \Dir( \mubf |_{\tilde{\Omega}}) + 3 \delta \leqslant \Dir(\mubf) + 3 \delta,
\end{equation*}  
where the third inequality comes from monotonicity and the smooth case treated above. As $\delta$ is arbitrary, we get that $l \leqslant \Dir(\mubf)$, which means
\begin{equation*}
\lim_{\varepsilon \to 0} \Dir_{\varepsilon}(\mubf) \leqslant \Dir(\mubf).
\end{equation*}
This equation still holds if $\mubf \notin H^1(\Omega, \Prob(D))$ as the r.h.s. is infinite. 

\medskip

\emph{General case: $\lim_\varepsilon \Dir_\varepsilon \geqslant \Dir$}. For this part, we need to control in a fine way the behavior of $\Dir_\varepsilon$ w.r.t. the approximation procedure of Theorem \ref{theorem_approximation}. Let $\mubf \in H^1(\Omega, \Prob(D))$ be given. Fix $\tilde{\Omega} \subset \Or$ compactly included and let $\tilde{\mubf}_n, \mubf_n$ the sequences used in the proof of Theorem \ref{theorem_approximation}. We recall that they are defined by 
\begin{equation*}
\tilde{\mubf}_n(\xi) := [\Phi^D_{1/n}] [\mubf(\xi)], \ \ \mubf_n (\xi) := \int_{\Omega} \chi_n(\xi - \eta) \tilde{\mubf}_n(\eta) \ddr \eta,
\end{equation*}
where $\chi_n : \R^p \to \R$ is a compactly supported convolution kernel and $\mubf_n$ is defined only over $\tilde{\Omega}$. Using the result for the smooth case, 
\begin{equation}
\label{equation_zz_aux_r_01}
\lim_{\varepsilon \to 0} \Dir_\varepsilon(\mubf_n|_{\tilde{\Omega}}) = \Dir(\mubf_n|_{\tilde{\Omega}}). 
\end{equation}
As the heat flow is a contraction in the Wasserstein space (Proposition \ref{proposition_heat_flow_properties}), we know that $\Dir_\varepsilon(\tilde{\mubf}_n) \leqslant \Dir_\varepsilon(\mubf)$. As $W_2^2$ is jointly convex w.r.t. to its two arguments, the function $\Dir_\varepsilon$ is convex for the affine structure on $L^2(\tilde{\Omega},\Prob(D))$. Hence, exactly by the same argument than in the proof of Theorem \ref{theorem_approximation},   
\begin{equation*}
\Dir_\varepsilon(\mubf_n) \leqslant \Dir_\varepsilon( \tilde{\mubf}_n) \leqslant \Dir_\varepsilon(\mubf),
\end{equation*}
and the important point is that the r.h.s. does not depend on $n$. Taking the limit $\varepsilon \to 0$ and using equation \eqref{equation_zz_aux_r_01}, we see that 
\begin{equation*}
\Dir(\mubf_n) = \lim_{\varepsilon \to 0} \Dir_\varepsilon(\mubf_n) \leqslant \lim_{\varepsilon \to 0} \Dir_\varepsilon(\mubf). 
\end{equation*} 
Now we can send $n \to + \infty$ and use Theorem \ref{theorem_approximation} to say that the l.h.s. converges to $\Dir(\mubf|_{\tilde{\Omega}})$. As $\tilde{\Omega}$ is now arbitrary, it yields the result
\begin{equation*}
\Dir(\mubf) \leqslant \lim_{\varepsilon \to 0} \Dir_\varepsilon(\mubf).
\end{equation*} 
In the case $\mubf \notin H^1(\Omega, \Prob(D))$, to justify that $\lim_{\varepsilon \to 0} \Dir_\varepsilon(\mubf)$, we can use for instance \cite[Proposition 4]{Chiron2007} which is valid for mappings valued in arbitrary metric spaces.

\medskip

\emph{The $\Gamma$-convergence}. The statement of $\Gamma$-convergence is now easy. To summarize, until now we have proved the monotonicity and that 
\begin{equation*}
\Dir(\mubf) = \lim_{\varepsilon \to 0} \Dir_\varepsilon(\mubf)
\end{equation*}
for every $\mubf \in L^2(\Omega, \Prob(D))$. It is an exercise that we leave to the reader to check that any sequence of functionals which are l.s.c. (which is the case for the $\Dir_\varepsilon$, see Proposition \ref{proposition_Dir_eps_lsc}) and which converges in a increasing way in fact $\Gamma$-converges.  
\end{proof}

\subsection{Boundary values}
\label{subsection_boundary_values}

It is well known that it is possible to make sense of the values of a $H^1$ real-valued function on hypersurfaces, in particular to give a meaning to the values of such a function on the boundary of a domain. As we want to define the Dirichlet problem, which consists in minimizing the Dirichlet energy with fixed values on the boundary $\dr \Omega$, we need to give a meaning to the boundary values of elements of $H^1(\Omega, \Prob(D))$. Korevaar and Schoen have already developed a trace theory in a fairly general context \cite[Section 1.12]{Korevaar1993}. However, in our specific situation and in view of proving the dual formulation of the Dirichlet problem, we will define the boundary values by showing how one can extend the continuity equation for test functions $\varphi \in C^1(\Omega \times D , \R^p)$ which are no longer compactly supported in $\Or$. Even if we do not prove it in this article, our definition of trace coincides with the one of \cite{Korevaar1993}: to be convinced one can look at Proposition \ref{proposition_extension_boundary_values} and compare it to \cite[Theorem 1.12.3]{Korevaar1993}. Recall that $\nO$ denotes the outward normal to $\dr \Omega$.   

\begin{theo}
\label{theorem_boundary_values}
Let $\mubf \in H^1(\Omega, \Prob(D))$. Then there exists a vector-valued measure $\BT_{\mubf} \in \M(\Omega \times D, \R^p)$ supported on $\nO \dr \Omega \times D$ (which means that $\BT_{\mubf}(\varphi) = 0$ if $\varphi \cdot \nO = 0$ on $\dr \Omega \times D$) such that for any $\varphi \in C^1(\Omega \times D, \R^p)$ and for any $\Em \in \M(\Omega \times D, \R^{pq})$ for which $(\mubf, \Em)$ satisfies the continuity equation and $\Dir(\mubf, \Em) < + \infty$,
\begin{equation}
\label{equation_continuity_equation_boundary}
\iint_{\Omega \times D} \nabla_\Omega \cdot \varphi \ddr \mubf + \iint_{\Omega \times D} \nabla_D \varphi \cdot \ddr \Em = \BT_{\mubf}(\varphi).
\end{equation}
Moreover if $\mubf$ is continuous then for any $\varphi \in C^1(\Omega \times D, \R^p)$,  
\begin{equation*}
\BT_{\mubf}(\varphi) = \int_{\dr \Omega} \left( \int_D \varphi(\xi, x) \cdot \nO(\xi) \mubf(\xi, \ddr x) \right) \sigma(\ddr \xi),
\end{equation*}
where $\sigma$ is the surface measure on $\dr \Omega$.
\end{theo}

\noindent $\BT_{\mubf}$ stands for ``Boundary Term'' of $\mubf$. It is not surprising that, if $\mubf$ is continuous, the value of $\BT_{\mubf}$ depends only on the values of $\mubf$ on the boundary.

\begin{proof}
Take $\mubf \in H^1(\Omega, \Prob(D))$ and $\Em = \vbf \mubf \in \M(\Omega \times D, \R^{pq})$ such that $(\mubf, \Em)$ satisfies the continuity equation and $\Dir(\mubf, \Em) < + \infty$. The l.h.s. of \eqref{equation_continuity_equation_boundary} defines a vector-valued distribution on $\Omega \times D$ acting on $\varphi$. We need to show that it is of order $0$ and that it does not depend on $\Em$. 

We define $f : \Omega \to \R^p$ by, for a.e. $\xi \in \Omega$,  
\begin{equation*}
f(\xi) := \int_D \varphi(\xi, x) \mubf(\xi, \ddr x). 
\end{equation*} 
Using the continuity equation with test functions of the form $\chi \varphi^\alpha$, for $\chi \in C^1_c(\Or, \R^p)$ and $\alpha \in \{ 1,2, \ldots,p \}$, one can see that $f \in H^1(\Or, \R^p)$ and 
\begin{equation*}
\dr_\alpha f^\beta(\xi) = \int_{D} \dr_\alpha \varphi^\beta(\xi, x) \mubf(\xi, \ddr x) + \int_{D} \nabla_D \varphi^\beta(\xi, x) \cdot \vbf^\alpha(\xi, x) \mubf(\xi, \ddr x).
\end{equation*} 
for all $\alpha, \beta \in \{ 1,2, \ldots, p \}$. In particular $f$ admits on $\dr \Omega$ a trace $\bar{f} : \dr \Omega \to \R^p$. We apply the divergence theorem: one can find in \cite[Section 4.3]{Evans1992} a statement when $\dr \Omega$ is only Lipschitz and $f$ has Sobolev regularity. In our case, given the expression of $\nabla f$, it reads 
\begin{equation}
\label{equation_zz_aux_21}
\iint_{\Omega \times D} \nabla_\Omega \cdot \varphi \ddr \mubf + \iint_{\Omega \times D} \nabla_D \varphi \cdot \ddr \Em = \int_{\dr \Omega} \bar{f}(\xi) \cdot \nO(\xi) \sigma(\ddr \xi)
\end{equation}
where $\nO$ is the outward normal to $\dr \Omega$ and $\sigma$ its the surface measure. In particular we see that the r.h.s. of \eqref{equation_continuity_equation_boundary} does not depend on $\Em$. Moreover, as $\| f \|_\infty \leqslant \| \varphi \|_\infty$, the same $L^\infty$ bounds holds for $\bar{f}$, thus 
\begin{equation*}
\left| \int_{\dr \Omega} \bar{f}(\xi) \cdot \nO(\xi) \sigma(\ddr \xi) \right| \leqslant \sigma(\dr \Omega) \| \varphi \|_\infty.
\end{equation*} 
It allows to conclude that the r.h.s. of \eqref{equation_continuity_equation_boundary} is a distribution of order $0$ acting on $\varphi$, hence it can be represented by a measure $\BT_{\mubf} \in \M(\Omega \times D, \R^p)$. From \eqref{equation_zz_aux_21} it is clear that $\BT_{\mubf}$ is supported on $\nO \dr \Omega \times D$. 

If we assume moreover that $\mubf$ is continuous, so is $f$. Indeed, for any $\xi, \eta \in \Omega$,
\begin{align*}
|f(\xi) - f(\eta)| & = \left|\int_D \varphi(\xi, x) \mubf(\xi, \ddr x) - \int_D \varphi(\eta, x) \mubf(\xi, \ddr x \right| \\
& \leqslant \int_D |\varphi(\xi, x) - \varphi(\eta, x)| \mubf(\xi, \ddr x) + \left| \int_D \varphi(\eta, x) \mubf(\xi, \ddr x) - \int_D \varphi(\eta,x) \mubf(\eta, \ddr x) \right| \\
& \leqslant  \| \nabla_\Omega \varphi \|_\infty |\xi - \eta| + \left| \int_D \varphi(\eta, x) \mubf(\xi, \ddr x) - \int_D \varphi(\eta,x) \mubf(\eta, \ddr x) \right|.
\end{align*} 
When $\xi \to \eta$, the first term obviously goes to $0$, and the second one too by definition of the weak convergence (by assumption $\mubf(\xi) \to \mubf(\eta)$ in the weak sense). Thus $\bar{f}$ coincides with $f$, which gives the announced result.
\end{proof}

\noindent If $\mubf \in H^1(\Omega, \Prob(D))$, using the disintegration theorem and testing against well chosen functions, one can show that there exists $\bar{\mubf} : \dr \Omega \to \Prob(D)$ defined $\sigma$-a.e. such that $\BT_{\mubf} = \nO \bar{\mubf} \otimes \sigma$. The mapping $\bar{\mubf}$ can be seen as a definition of the values of $\mubf$ on $\dr \Omega$. 
 
Now we can define what it means to share the same boundary values and prove that the set of $\mubf$ with fixed boundary values is closed. 

\begin{defi}
Let $\mubf$ and $\nubf$ two elements of $H^1(\Omega, \Prob(D))$. We say that $\mubf |_{\dr \Omega} = \nubf |_{\dr \Omega}$ if $\BT_{\mubf} = \BT_{\nubf}$.
\end{defi}

\begin{prop}
\label{proposition_boundary_fixed_closed}
Let $\mubf_b \in H^1(\Omega, \Prob(D))$ and $C \in \R$ be fixed. Then the set 
\begin{equation*}
\{ \mubf \in H^1(\Omega, \Prob(D)) \ : \ \mubf |_{\dr \Omega} = \mubf_b |_{\dr \Omega} \text{ and } \Dir(\mubf) \leqslant C \}
\end{equation*}
is closed for the weak topology on $L^2(\Omega, \Prob(D))$.
\end{prop}

\begin{proof}
The proof is straightforward. Indeed, take a sequence $(\mubf_n)_{n \in \N}$ in $ \in L^2(\Omega, \Prob(D))$ such that $\mubf_n |_{\dr \Omega} = \mubf_b |_{\dr \Omega}$ and $\Dir(\mubf_n) \leqslant C$ for any $n \in \N$, and assume it converges weakly to some $\mubf \in L^2(\Omega, \Prob(D))$. By lower semi-continuity of $\Dir$, we know that $\Dir(\mubf) \leqslant C$. For any $n \in \N$ choose $\Em_n \in \M(\Omega \times D, \R^{pq})$ tangent to $\mubf_n$, similarly take $\Em_b$ tangent to $\mubf_b$. The identity $\mubf_n |_{\dr \Omega} = \mubf_b |_{\dr \Omega}$ can be written: for every $\varphi \in C^1(\Omega \times D, \R^p)$ 
\begin{equation}
\label{equation_boundary_value_closed}
\iint_{\Omega \times D} \nabla \cdot \varphi \ddr \mubf_n + \iint_{\Omega \times D} \nabla_D \varphi \cdot \ddr \Em_n = \iint_{\Omega \times D} \nabla \cdot \varphi \ddr \mubf_b + \iint_{\Omega \times D} \nabla_D \varphi \cdot \ddr \Em_b.
\end{equation}
As seen in the proof of Proposition \ref{proposition_lsc_convex_Dir}, one can assume that, up to extraction, $(\Em_n)_{n \in \N}$ weakly converges to some $\Em$. It is easy to see that $(\mubf, \Em)$ satisfies the continuity equation and that $\Dir(\mubf, \Em) \leqslant C < + \infty$. Thus, we can pass to the limit in \eqref{equation_boundary_value_closed} and see that for any $\varphi \in C^1(\Omega \times D, \R^p)$, 
\begin{equation*}
\iint_{\Omega \times D} \nabla \cdot \varphi \ddr \mubf + \iint_{\Omega \times D} \nabla_D \varphi \cdot \ddr \Em = \iint_{\Omega \times D} \nabla \cdot \varphi \ddr \mubf_b + \iint_{\Omega \times D} \nabla_D \varphi \cdot \ddr \Em_b,
\end{equation*}  
which exactly means that $\mubf |_{\dr \Omega} = \mubf_b |_{\dr \Omega}$.
\end{proof}

\section{The Dirichlet problem and its dual}
\label{section_Dirichlet_problem}

\subsection{Statement of the problem}
\label{subsection_Dirichlet_problem}

With all the tools at our disposal, we are ready to state the Dirichlet problem. It simply consists in minimizing the Dirichlet energy under the constraint that the values at the boundary are fixed. 

\begin{defi}
Let $\mubf_b \in H^1(\Omega, \Prob(D))$. Then the Dirichlet problem with boundary values $\mubf_b$ is defined as 
\begin{equation*}
\min_{\mubf} \left\{ \Dir(\mubf) \ : \ \mubf \in H^1(\Omega, \Prob(D)) \text{ and } \mubf |_{\dr \Omega} = \mubf_b |_{\dr \Omega} \right\}.
\end{equation*}
A mapping $\mubf \in H^1(\Omega, \Prob(D))$ which realizes the minimum is called a solution of the Dirichlet problem. 
\end{defi}

\begin{defi}
Let $\mubf \in H^1(\Omega, \Prob(D))$. We say that $\mubf$ is harmonic if it is a solution of the Dirichlet problem with boundary values $\mubf$.
\end{defi}

With the work of the previous section, the existence of at least one solution is a straightforward application of the direct method of calculus of variations. 

\begin{theo}
\label{theorem_existence_solution_Dirichlet}
Let $\mubf_b \in H^1(\Omega, \Prob(D))$. Then there exists at least one solution of the Dirichlet problem with boundary values $\mubf_b$. 
\end{theo}

\begin{proof}
There exists at least one $\mubf$ with finite Dirichlet energy which satisfies the boundary conditions, namely $\mubf_b$. Thus, one can consider a minimizing sequence $(\mubf_n)_{n \in \N}$. By compactness of $L^2( \Omega, \Prob(D))$, we can assume, up to extraction, that this sequence converges weakly to some $\mubf \in L^2(\Omega, \Prob(D))$. By Proposition \ref{proposition_boundary_fixed_closed}, we know that $\mubf$ also satisfies $\mubf |_{\dr \Omega} = \mubf_b |_{\dr \Omega}$. The lower semi-continuity of $\Dir$ allows to conclude that $\mubf$ is a minimizer of $\Dir$.
\end{proof}

Let us spend a few words about the question of uniqueness. Of course, the proof above provides no information about it. By convexity of the Dirichlet energy (Proposition \ref{proposition_lsc_convex_Dir}), we know that the set of solutions of the Dirichlet problem is convex. Recall that if $\Omega = [0,1]$ is a segment of $\R$, then the Dirichlet problem reduces to the problem of finding a geodesic between the two endpoints $\mubf_b(0)$ and $\mubf_b(1)$. It is well known that a sufficient condition for uniqueness is to impose that either $\mubf_b(0)$ or $\mubf_b(1)$ are absolutely continuous w.r.t. $\Leb_D$, and there can be non uniqueness when it is not the case (see for instance \cite[Chapter 5]{OTAM}). Hence, it would natural, in order to investigate the question of uniqueness, to impose that for every $\xi \in \dr \Omega$, the measure $\mubf_b(\xi)$ is absolutely continuous w.r.t. $\Leb_D$ (or maybe just for $\xi$ belonging to a positive fraction of $\dr \Omega$).
We do not know if uniqueness holds under this hypothesis: a difference with the case where $\Omega$ is a segment is the fact that we do not know a static or Lagrangian formulation. In other words, we do not know the equivalent of transport plans, which in the case of a $1$-dimensional $\Omega$, allow to parametrize geodesics and to greatly simply the problem. However we are able to prove uniqueness in a non trivial case: the one of a family of elliptically contoured distributions treated in Subsection \ref{subsection_location_scatter}, see also the introduction where the strategy of the proof is discussed.  

\subsection{Lipschitz extension}
\label{subsection_Lipschitz_extension}
 
To give ourselves the boundary conditions, we need a mapping $\mubf_b$ defined on the whole $\Omega$, even though only its values near $\dr \Omega$ will play a role. Thus a natural question arises: if $\mubf_b$ is only defined on $\dr \Omega$, is it possible to extend it on $\Omega$? The next theorem shows that the answer is positive in the case where $\mubf_b$ is Lipschitz on $\dr \Omega$. Indeed, in this case we can build an extension which is Lipschitz on $\Omega$, thus in $H^1(\Omega, \Prob(D))$ thanks to Theorem \ref{theorem_equivalence_sobolev_metric}.   

\begin{theo}
\label{theorem_lipschitz_extension}
Let $\mubf_l : \dr \Omega \to \Prob(D)$ a Lipschitz mapping. Then there exists $\mubf : \Omega \to \Prob(D)$ Lipschitz such that $\mubf(\xi) = \mubf_l(\xi)$ for every $\xi \in \dr \Omega$. 
\end{theo}

\noindent For a continuous $\mubf$ the boundary term $\BT_{\mubf}$ depends only on the values of $\mubf$ on $\dr \Omega$ (Theorem \ref{theorem_boundary_values}), hence the boundary term of the Lipschitz extension of $\mubf_l : \dr \Omega \to \Prob(D)$ does not depend on the extension. In other words, the following problem is well defined: 

\begin{defi}
Let $\mubf_l : \dr \Omega \to \Prob(D)$ a Lipschitz mapping. Then the Dirichlet problem with boundary values $\mubf_l$ is defined as the Dirichlet problem with boundary values $\mubf_b$, where $\mubf_b$ is any Lipschitz extension of $\mubf_l$ on $\Omega$.
\end{defi}

Now, let us prove the Lipschitz extension theorem. It relies on the following Lemma, which allows to treat the case where $\Omega$ is a ball. 

\begin{lm}
\label{lemma_lipschitz_extension_balls}
Let $B(0,1)$ be the unit ball of $\R^p$ and $\Sph^{p-1} := \dr B(0,1)$ its boundary. Let $\mubf_l : \Sph^{d-1} \to \Prob(D)$ a Lipschitz mapping and take $x_0 \in D$. Define, for any $r \in [0,1]$ the map $T_r : D \to D$ by $T_r(x) = r x + (1-r) x_0$. Then the mapping $\mubf : B(0,1) \to \Prob(D)$ defined by 
\begin{equation*}
\mubf(r\xi) := T_r \# [\mubf(\xi)]
\end{equation*}
for any $r \in [0,1]$ and any $\xi \in \Sph^{d-1}$ is Lipschitz. 
\end{lm} 

\begin{proof}
If $\xi \in \Sph^{d-1}$ is fixed, then $r \in [0,1] \mapsto \mubf(r\xi)$ is the constant speed geodesic joining $\delta_{x_0}$ to $\mubf_l(\xi)$. Hence, we can write that $W_2(\mubf(r\xi), \mubf(s \xi)) \leqslant C |r-s|$, where $C$ depends only on the diameter of $\Prob(D)$. On the other hand, as $T_r$ is $r$-Lipschitz in $D$, then $\nu \mapsto T_r \# \nu$ is also $r$-Lipschitz in $\Prob(D)$. Hence, for any $\xi$ and $\eta$ in $\Sph^{d-1}$, one has $W_2(\mubf(r\xi), \mubf(r\eta)) \leqslant C r |\xi-\eta|$, where $C$ is the Lipschitz constant of $\mubf_l$. Putting the two estimates together, we deduce that for any $r,s \in [0,1]$ and any $\xi,\eta \in \Sph^{p-1}$, 
\begin{equation*}
W_2(\mubf(r\xi), \mubf(s \eta)) \leqslant C [|r-s| + \min(r,s) |\xi-\eta|],
\end{equation*}  
which is enough to conclude that $\mubf$ is Lipschitz.
\end{proof}

\noindent Notice that the Lipschitz constant of the extension is not controlled by the Lipschitz constant of $\mubf_l$: the distance between $\delta_{x_0}$ and the range of $\mubf_l$ also plays a role as $\mubf(0) = \delta_{x_0}$. Hence, we cannot use a decomposition with Withney cubes to extend mappings defined on arbitrary closed subsets $\Omega$, but only on the boundary of smooth sets: basically we need to use Lemma \ref{lemma_lipschitz_extension_balls} only a finite number of times. 

\begin{proof}[Proof of Theorem \ref{theorem_lipschitz_extension}]
We will use Lemma \ref{lemma_lipschitz_extension_balls} in the following form: if $\Omega$ is a domain which is in a bilipschitz bijection with a ball, then Theorem \ref{theorem_lipschitz_extension} holds for this domain. 

We reason by induction on $p \geqslant 1$ the dimension of $\Omega$. In dimension $1$, $\Omega = I$ is a segment. To extend a mapping defined only on the boundary of the segment $I$, we take the constant speed geodesic in $\Prob(D)$ between the values of $\mubf_l$ at the two endpoints of $I$. 

Now assume that the result holds for some $p-1 \geqslant 1$ and let $\Omega$ be a compact domain with Lipschitz boundary in $\R^p$. The goal is to cut $\Omega$ in a finite number of pieces on which Lemma \ref{lemma_lipschitz_extension_balls} apply. For each $\xi \in \Omega$ we choose $r_{\xi} >0$ such that $B(\xi, r_{\xi}) \cap \Omega$ is in a bilipschitz bijection with a ball. It is obvious that we can do that for $\xi \in \Or$, and for points on $\dr \Omega$ we use the fact that $\Omega$ is locally the epigraph of a Lipschitz function. By compactness, we find balls $B_1, B_2, \ldots, B_N$ covering $\Omega$ such that $B_n \cap \Omega$ is in a bilipschitz bijection with a ball for any $n \in \{ 1,2, \ldots, N \}$. We can of course assume that $B_n$ is not included in $B_m$ for any $n \neq m$. Then we define recursively $X_1 := B_1 \cap \Omega$ and $X_n = (B_n \cap \Omega) \bsl \mathring{X}_{n-1}$ for $n \in \{ 2, \ldots, N \}$. For any $n \in \{ 1,2, \ldots, N \}$, $X_n$ is still in a bilipschitz bijection with a ball. On $\bigcup_n \dr X_n$, which is made of $\dr \Omega$ and of pieces of spheres of $\R^p$, thus locally in bilipschitz bijection with Lipschitz domains of $\R^{p-1}$, we can use the induction assumption and extend $\mubf_l$. Then, we use Lemma \ref{lemma_lipschitz_extension_balls} to extend $\mubf$ on $\mathring{X}_n$ for each $n \in \{ 1,2, \ldots, N \}$. We have obtained a function $\mubf$ which is continuous and Lipschitz on each $X_n, n \in \{ 1,2, \ldots, N \}$: it is globally Lipschitz on $\Omega$.     
\end{proof} 

\subsection{The dual problem}

We will know show a rigorous proof of the absence of duality gap. The dual problem was already obtained, at least formally, in the introduction. 

\begin{theo}
\label{theorem_duality}
Let $\mubf_b \in H^1(\Omega, \Prob(D))$. Then one has 
\begin{align*}
\sup_\varphi \Bigg\{ \BT_{\mubf_b}(\varphi) \ : \ \varphi \in C^1(\Omega \times D, \R^p) \ & \text{ and } \ \nabla_\Omega \cdot \varphi + \frac{|\nabla_D \varphi|^2}{2} \leqslant 0 \text{ on } \Omega \times D \Bigg\} \\
& = \min_{\mubf} \left\{ \Dir(\mubf) \ : \ \mubf \in H^1(\Omega, \Prob(D)) \text{ and } \mubf |_{\dr \Omega} = \mubf_b |_{\dr \Omega} \right\}.
\end{align*}
\end{theo}

\begin{proof}
We rely on the Fenchel-Rockafellar duality theorem which can be found in \cite[Theorem 1.9]{Villani2003}. Let $X := C(\Omega \times D, \R^{1+pq})$ the space of continuous functions defined on the compact space $\Omega \times D$ and valued in $\R^{1+pq}$ endowed with the norm of uniform convergence. An element of $X$ will be written $(a,b)$, where $a \in C(\Omega \times D)$ and $b \in C(\Omega \times D, \R^{pq})$. The dual space $X^\star$ is, by the Riesz theorem, $\M(\Omega \times D, \R^{1+pq})$. Again an element of $X^\star$ will be written $(\mubf, \Em)$ where $\mubf \in \M(\Omega \times D)$ is a signed measure and $\Em \in \M(\Omega \times D, \R^{pq})$ is a vector-valued measure. We introduce the functionals $F : X \to \R$ and $G : X \to \R$ defined as, for any $(a,b) \in X$, 
\begin{align*}
F(a,b) & = \begin{cases}
0 & \text{if } \dst{a(\xi, x) + \frac{|b(\xi, x)|^2}{2}} \leqslant 0 \text{ for every } (\xi, x) \in \Omega \times D \\
+ \infty & \text{else},
\end{cases}
\\
G(a,b) & = \begin{cases}
- \BT_{\mubf_b}(\varphi) & \text{if } (a,b) = (\nabla_\Omega \cdot \varphi, \nabla_D \varphi) \text{ for some } \varphi \in C^1(\Omega \times D, \R^p) \\
+ \infty & \text{else}.
\end{cases}
\end{align*}
Notice that thanks to \eqref{equation_continuity_equation_boundary}, $G$ is well defined and does not depend on the choice of $\varphi$ such that $(a,b) = (\nabla_\Omega \cdot \varphi, \nabla_D \varphi)$. Notice also that at the point $(-1,0) \in X$, one has that $F$ is finite and continuous and that $G$ is finite (take $\varphi(\xi, x) := (- \xi^1, 0, 0, \ldots, 0)$, where $\xi^1$ is the first component of $\xi$). As moreover $F$ and $G$ are convex, one can apply Fenchel-Rockafellar duality which means 
\begin{align*}
- \min_{(\mubf, \Em) \in X^\star} [ F^\star(\mubf, \Em) + G^\star(-\mubf,& -\Em) ]  = \inf_X (F+G) \\
& = - \sup_\varphi \Bigg\{ \BT_{\mubf_b}(\varphi) \ : \ \varphi \in C^1(\Omega \times D, \R^p) \ \text{ and } \ \nabla_\Omega \cdot \varphi + \frac{|\nabla_D \varphi|^2}{2} \leqslant 0 \Bigg\},
\end{align*}
where the last inequality is just a rewriting of the definition of $F$ and $G$. Let us compute $F^\star(\mubf, \Em)$. By definition, 
\begin{equation*}
F^\star(\mubf, \Em) = \sup_{a,b} \left\{ \iint_{\Omega \times D} a \ddr \mubf + \iint_{\Omega \times D} b \cdot \ddr \Em \ : (a,b) \in C(\Omega \times D, \mathcal{K})  \right\},
\end{equation*} 
where $\mathcal{K}$ is defined in Definition \ref{definition_dirichlet_mu_E}. In particular, if $\mubf$ is not a positive measure, then choosing suitable negative $a$, one sees that $F^\star(\mubf, \Em) = + \infty$. Moreover, if $\mubf \in L^2(\Omega, \Prob(D))$ and $(\mubf, \Em)$ satisfies the continuity equation, then $F^\star(\mubf, \Em) = \Dir(\mubf, \Em)$: this is precisely Definition \ref{definition_dirichlet_mu_E}. On the other hand, we can compute $G^\star$: for any $(\mubf,\Em) \in X^\star$, 
\begin{equation*}
G^\star(-\mubf, -\Em) = \sup_{\varphi \in C^1(\Omega \times D, \R^p)} \left( \BT_{\mubf_b}(\varphi) - \iint_{\Omega \times D} \nabla_\Omega \cdot \varphi \ddr \mubf - \iint_{\Omega \times D} \nabla_D \varphi \cdot \ddr \Em \right).
\end{equation*} 
By linearity of the expression inside the $\sup$ w.r.t. $\varphi$, we see that $G^\star(-\mubf, -\Em) < + \infty$ if and only if $G^\star(-\mubf, -\Em) = 0$, which translates in 
\begin{equation*}
\BT_{\mubf_b}(\varphi) = \iint_{\Omega \times D} \nabla_\Omega \cdot \varphi \ddr \mubf + \iint_{\Omega \times D} \nabla_D \varphi \cdot \ddr \Em
\end{equation*}
for every $\varphi \in C^1(\Omega \times D, \R^p)$. Let $a \in C(\Omega)$ a continuous function. It can always be written $a = \nabla_\Omega \cdot \varphi$, where $\varphi \in C^1(\Omega, \R^p)$ (take $\varphi = \nabla f$ where $f$ solves $\Delta f = a$), thus using the fact that for such a $\varphi$, 
\begin{equation*}
\BT_{\mubf_b}(\varphi) = \iint_{\Omega \times D} \nabla_\Omega \cdot \varphi \ddr \mubf_b = \iint_{\Omega \times D}a \ddr \mubf_b = \int_\Omega a(\xi) \ddr \xi,
\end{equation*}
one sees that if $G^\star(-\mubf, -\Em) < + \infty$, then 
\begin{equation*}
\int_\Omega a(\xi) \ddr \xi = \iint_{\Omega \times D} a \ddr \mubf.
\end{equation*}
Provided that $\mubf$ is a positive measure (recall that it happens if $F^\star(\mubf, \Em) < + \infty$) and by arbitrariness of $a$, it implies that the disintegration of $\mubf$ w.r.t. $\Leb_\Omega$ is made of probability measures on $D$, in other words that $\mubf \in L^2(\Omega, \Prob(D))$. Once we have this information, testing with functions $\varphi$ which are compactly supported on $\Omega$, we see that if $G^\star(-\mubf, -\Em) < + \infty$ then $(\mubf, \Em)$ satisfies the continuity equation, and testing with arbitrary $\varphi$, we see that $\BT_{\mubf} = \BT_{\mubf_b}$. In the end, one concludes that 
\begin{equation*}
\min_{(\mubf, \Em) \in X^\star} [ F^\star(\mubf, \Em) + G^\star(-\mubf, -\Em) ] =  \min_{\mubf} \left\{ \Dir(\mubf) \ : \ \mubf \in H^1(\Omega, \Prob(D)) \text{ and } \mubf |_{\dr \Omega} = \mubf_b |_{\dr \Omega} \right\}. \qedhere
\end{equation*}  
\end{proof}

A natural question which arises is the existence of an optimal $\varphi \in C^1(\Omega \times D, \R^p)$ (or in a space of less regular functions). If $\Omega$ is a segment, the constraint on $\varphi$ translates into the Hamilton-Jacobi equation 
\begin{equation*}
\partial_t \varphi + \frac{|\nabla \varphi|^2}{2} \leqslant 0,
\end{equation*}
whose explicit solutions are known. In particular, one can parametrize the function $\varphi$ by its value at the initial time, the unknown becomes a scalar function defined on $D$. In the compact case, by the double convexification trick, one can get compactness in a maximizing sequence. In our case, the constraint reads in full coordinate 
\begin{equation*}
\sum_{\alpha = 1}^p \dr_\alpha \varphi^\alpha + \frac{1}{2} \sum_{\alpha = 1}^p \sum_{i=1}^q |\dr_i \varphi^\alpha|^2 \leqslant 0.
\end{equation*} 
Now we do not know if one can parametrize a $\varphi$ which satisfies the constraint by its values on the boundary of $\dr \Omega$ (and even if it were the case, on which part of the boundary?). Moreover, notice that the function $\varphi$ is vector-valued, though the constraint involves only one scalar equation: to get compactness out of it seems more complicated. We have not investigated deeply the question of the existence of an optimal $\varphi$, but we believe that it can be substantially more complicated than in the case where $\Omega$ is a segment of $\R$.

\section{Failure of the superposition principle}
\label{section_superposition_principle}

\subsection{The superposition principle}

In this section, we want to explain why a powerful tool to study curves valued in the Wasserstein space (i.e. the case where $\Omega$ is a segment of $\R$), namely the superposition principle, fails in higher dimensions. To say it briefly, there is no Lagrangian point of view for mappings into the Wasserstein space, one has to work only with the Eulerian one. Notice that the question of the existence of a superposition principle was already formulated by Brenier \cite[Problem 3.1]{Brenier2003}, but left unanswered. As we want to prove a negative result, we will not only provide a counterexample to the superposition principle, but also try to explain the obstruction and why this principle fails for all but few exceptional cases. Let us first recall the superposition principle for absolutely continuous curves. 

The set $\Omega$ will be replaced by the unit segment $I = [0,1]$. As stated in Proposition \ref{proposition_Benamou_Brenier_1D}, the set $H^1(I, \Prob(D))$ coincides with the set of absolutely continuous curves. We denote by $\Gamma =C(I, D)$ the set of continuous curves valued in $D$ endowed with the norm of uniform convergence, it is a polish space. If $f \in \Gamma$, then $\dot{f}$ denotes the derivative w.r.t. time of $f$ provided that it exists. For any $t \in I$, $e_t : \Gamma \to D$ is the evaluation operator, which means $e_t(f) = f(t)$ for any $f \in \Gamma$. The following result can be found in \cite[Section 8.2]{Ambrosio2008}. 

\begin{theo}
\label{theorem_superposition_1D}
Let $\mubf \in H^1(I, \Prob(D))$. Then there exists a probability measure $Q \in \Prob(\Gamma)$ such that 
\begin{itemize}
\item[(i)] for any $t \in I$, $e_t \# Q = \mubf(t)$ ; 
\item[(ii)] the following equality holds: 
\begin{equation*}
\Dir(\mubf) = \int_{\Gamma} \left( \int_I \frac{1}{2} |\dot{f}(t)|^2 \ddr t \right) Q( \ddr \gamma ).
\end{equation*}
\end{itemize}
\end{theo}  

The measure $Q$ can be seen as a multimarginal transport plan coupling all the different instants, whose $2$-marginals are almost optimal transport plans if they are taken between two very close instants. In other words, for any $t$ and $s$ in $I$, $(e_s, e_t) \# Q$ is a transport plan between $\mubf(s)$ and $\mubf(t)$ (by (i)), and it is almost an optimal transport plan if $s$ is very close to $t$ by (ii). 

Another way to see it is the following: if $f \in \Gamma$, then we can also see it as an element $\mubf_f$ of $H^1(I, \Prob(D))$. Indeed, just set $\mubf(t) = \delta_{f(t)}$ for any $t \in I$, and one can define $\Em_f \in \M(I \times D, \R^q)$ by, for any $b \in C(I \times D, \R^q)$, 
\begin{equation*}
\iint_{I \times D} b \cdot \ddr \Em_f := \int_I b(t, f(t)) \cdot \dot{f}(t) \ddr t.
\end{equation*} 
With this choice, one can check that 
\begin{equation*}
\Dir(\mubf_f) = \Dir(\mubf_f, \Em_f) = \int_I \frac{1}{2} |\dot{f}(t)|^2 \ddr t.
\end{equation*}
Then, Theorem \ref{theorem_superposition_1D} is saying that there exists $Q \in \Prob(\Gamma)$ such that $\mubf$ is the mean w.r.t. $Q$ of the $\mubf_f$ (this is (i)), and such the $\Em$ which is tangent to $\mubf$ is the mean w.r.t. $Q$ of the $\Em_f$. Indeed, by linearity of the continuity equation the mean of the $\Em_f$ is an admissible momentum. Using Jensen's inequality,
\begin{equation*}
\Dir(\mubf) = \Dir \left( \int_{\Gamma} \mubf_f Q(\ddr f) \right) \leqslant \Dir \left( \int_{\Gamma} \mubf_f Q(\ddr f), \int_{\Gamma} \Em_f Q(\ddr f) \right) \leqslant \int_{\Gamma} \Dir(\mubf_f, \Em_f) Q(\ddr f)
\end{equation*}
and the r.h.s. is equal to the l.h.s. by (ii). Hence, all inequalities are equalities, which tells us that $\int_{\Gamma} \Em_f Q(\ddr f)$ is the tangent momentum to $\mubf$.

%is saying that there exists $Q \in \Prob(\Gamma)$ such that $\mubf$ is the mean w.r.t. $Q$ of the $\mubf_f$'s (this is (i)), and such the $\Em$ which is tangent to $\mubf$ is the mean w.r.t. $Q$ of the $\Em_f$ (indeed, thanks to Jensen's inequality, (ii) translates in the fact that $\Dir(\mubf) = \Dir(\mubf, \int_\Gamma \Em_f Q(\ddr f) )$ ). 

\medskip  

Let us try to see what a superposition principle would look like if the dimension of $\Omega$ is larger than $1$. We denote by $\F$ the space $L^2(\Omega, D)$ which is a polish space. As it was already done in \cite{Brenier2003}, if $f \in H^1(\Omega, D)$, then we can see it as an element $\mubf_f$ of $H^1(\Omega, \Prob(D))$ by setting $\mubf_f(\xi) := \delta_{f(\xi)}$. In other words, a classical function can be seen as a mapping valued in the Wasserstein space by identifying $f(\xi) \in D$ with $\delta_{f(\xi)} \in \Prob(D)$. More precisely, we define $\mubf_f \in L^2(\Omega, \Prob(D))$ and $\Em_f \in \M(\Omega \times D, \R^{pq})$ by, for any $a \in C(\Omega \times D)$ and $b \in C(\Omega \times D, \R^{pq})$,
\begin{align*}
\iint_{\Omega \times D} a \ddr \mubf_f &:= \int_\Omega a(\xi, f(\xi)) \ddr \xi, \\
\iint_{\Omega \times D} b \cdot \ddr \Em_f &:= \int_\Omega b(\xi, f(\xi)) \cdot \nabla f(\xi) \ddr \xi.
\end{align*}

\begin{prop}
\label{proposition_injection_function_generalized}
If $f \in H^1(\Omega, D)$, and if $\mubf_f$ and $\Em_f$ are defined as above, then $\Em_f$ is tangent to $\mubf_f$ and
\begin{equation*}
\Dir(\mubf_f) = \Dir (\mubf_f, \Em_f) = \int_\Omega \frac{1}{2} |\nabla f(\xi)|^2 \ddr \xi.
\end{equation*}
\end{prop}

\begin{proof}
To check the first part, take $\varphi \in C^1_c(\Or \times D, \R^p)$. Defining $\tilde{\varphi} \in H^1(\Omega, \R^p)$ by $\tilde{\varphi}(\xi) = \varphi(\xi, f(\xi))$, we have that $\tilde{\varphi}$ is compactly supported in $\Or$ and 
\begin{equation*}
\nabla \cdot \tilde{\varphi} = (\nabla_\Omega \cdot \varphi)(\xi, f(\xi)) + (\nabla_D \varphi)(\xi, f(\xi)) \cdot \nabla f(\xi).
\end{equation*}
Integrating this identity w.r.t. $\Omega$, as the l.h.s. vanishes by compactness of the support of $\tilde{\varphi}$, we see that we can conclude that $(\mubf_f, \Em_f)$ satisfies the continuity equation. 

Notice that $\Em_f$ has a density $\vbf_f \in L^2_{\mubf_f}(\Omega \times D, \R^{pq})$ w.r.t. $\mubf$ given by $\vbf_f(\xi, x) = \nabla f(\xi)$. In particular, for a fixed $\xi$, $\vbf_f(\xi, \cdot)$ is constant hence the gradient of a function. Using Proposition \ref{proposition_characterization_v_tangent}, one sees that it is enough to conclude that $\Em_f$ is tangent. Moreover, as $\vbf_f$ does not depend on $x$, 
\begin{equation*}
\Dir(\mubf_f) = \Dir (\mubf_f, \Em_f) = \iint_{\Omega \times D} \frac{1}{2} |\vbf_f(\xi)|^2 \mubf(\ddr \xi, \ddr x) = \int_\Omega \frac{1}{2} |\vbf_f(\xi)|^2 \ddr \xi = \int_\Omega \frac{1}{2} |\nabla f(\xi)|^2 \ddr \xi. \qedhere
\end{equation*} 
\end{proof}

We mention that \cite[Theorem 3.1]{Brenier2003} states that if $f : \Omega \to D$ is a (classical) harmonic map, then $\mubf_f$ is also an harmonic mapping. To prove such a fact, Brenier showed how one can build a solution of the dual problem (with boundary values $\mubf_f$) from the function $f$. A more recent analysis of such result, in the case where $D$ is replaced by a Riemannian manifold, can be found in \cite{Lu2017}.    

By analogy, the superposition principle would read as follows: If $\mubf \in H^1(\Omega, \Prob(D))$ and $\Em \in \M(\Omega \times D, \R^{pq})$ is tangent to $\mubf$, does there exist $Q \in \Prob(\F)$ such that $\mubf$ is the mean of $\mubf_f$ w.r.t. $Q$ and $\Em$ is the mean of $\Em_f$ w.r.t. $Q$? Thanks to Jensen's inequality and the uniqueness of the tangent momentum, the second condition can in fact be rewritten as 
\begin{equation*}
\Dir(\mubf) = \Dir(\mubf, \Em) = \int_{\F} \Dir(\mubf_f,\Em_f) Q(\ddr f) = \int_{\F} \left( \int_\Omega \frac{1}{2} |\nabla f(\xi)|^2 \ddr \xi \right) Q(\ddr f).
\end{equation*}
These considerations can be summarized by the following definition, which is the same as \cite[Problem 3.1]{Brenier2003}. For $f \in \F$ we define its ``classical'' Dirichlet energy $\Dir_c(f)$ by
\begin{equation*}
\Dir_c(f) = \begin{cases}
\dst{\int_\Omega \frac{1}{2} |\nabla f(\xi)|^2 \ddr \xi} & \text{if } f \in H^1(\Omega, D), \\
+ \infty & \text{else}.
\end{cases}
\end{equation*}  

\begin{defi}
\label{definition_superposition_principle}
Let $\mubf \in H^1(\Omega, \Prob(D))$. We say that $\mubf$ admits a superposition principle if there exists $Q \in \Prob(\F)$ such that 
\begin{itemize}
\item[(i)] for any $a \in C(\Omega \times D)$;
\begin{equation*}
\iint_{\Omega \times D} a \ddr \mubf = \int_{\F} \left( \int_\Omega a(\xi, f(\xi)) \ddr \xi \right) Q(\ddr f),
\end{equation*}
\item[(ii)] the following identity holds: 
\begin{equation*}
\int_{\F} \Dir_c(f) Q(\ddr f) \leqslant \Dir(\mubf).
\end{equation*}
\end{itemize}
\end{defi}

\noindent In particular, with our definition, if $Q$ represents $\mubf \in H^1(\Omega, D)$, then for $Q$-a.e. function $f$ one has $\Dir_c(f) < + \infty$ hence $f$ belongs to $H^1(\Omega, D)$. Let us underline that (i) is heuristically the same as (i) of Theorem \ref{theorem_superposition_1D}, but in a form integrated over $\Omega$ because the evaluation operator does not make sense in higher dimensions: the elements of $\F$ are not necessarily continuous. In Definition \ref{definition_superposition_principle}, if (i) and (ii) holds, then the inequality in (ii) is in fact an equality because the reverse inequality always holds. Indeed, if $\mubf$ satisfies the superposition principle, we can say that $\mubf = \int_{\F} \mubf_f Q(\ddr f)$. By convexity of the Dirichlet energy (Proposition \ref{proposition_lsc_convex_Dir}), we can apply Jensen's inequality, thus 
\begin{equation*}
\Dir(\mubf) \leqslant \int_{\F} \Dir(\mubf_f) Q(\ddr f)  = \int_{\F} \Dir_c(f) Q(\ddr f).
\end{equation*}

\subsection{Counterexample}

We will first provide a counterexample which we will try to make as generic as possible. In what follows, we take $\Omega := \U$ to be the unit disk of $\R^2$ and $\Sph^1 = \dr \U$ its boundary. We also take $D = \U$. We view $\U$ as a subset of the complex plane $\mathbb{C}$: multiplication on $\U$ means complex multiplication. 

Let $\mubf_s : \Sph^1 : \to \Prob(\U)$ be the (complex) square root: it is the mapping defined by, for $\xi \in \Sph^1$, 
\begin{equation*}
\mubf_s(\xi) := \frac{1}{2} \sum_{z^2 = \xi}  \delta_{z}  = \frac{1}{2} (\delta_{\sqrt{\xi}} + \delta_{-\sqrt{\xi}}), 
\end{equation*}
where $\sqrt{\xi}$ is a (complex) square root of $\xi$. The function $\mubf_s$ is clearly Lipschitz (with Lipschitz constant equals to $2$). In fact, if $\xi = e^{i t}$ with $t \in \R$, one can write 
\begin{equation*}
\mubf_s(e^{it}) = \frac{1}{2} \left( \delta_{\exp(it/2)} + \delta_{\exp(it/2 + i \pi)} \right).
\end{equation*} 
The function $t \mapsto \mubf_s(e^{it})$ is $2 \pi$-periodic, but it cannot be written as a superposition of continuous $2\pi$-periodic functions, only $4\pi$-periodic ones. Hence, the superpositon principle with continuous functions fails for this mapping. This example is well known in the theory of $Q$-functions \cite{DeLellis2011}, we took it from there. To our purpose, we will need the fact that the superposition principle with $H^{1/2}$ functions fails for the mapping $\mubf_s$: roughly speaking, it holds because $H^{1/2}$ functions, in dimension $1$, cannot have jumps. 

\begin{lm}
\label{lemma_H1/2_no_jumps}
There is no function $f \in H^{1/2}(\Sph^1, \U)$ such that $f(\xi)^2 = \xi$ for a.e. $\xi \in \Sph^1$.
\end{lm}

\noindent As this lemma is not directly related to harmonic mappings, we postpone its proof to the end of this article in Section \ref{section_appendix_H1/2_square_root}. With the help of this lemma, we can prove that no mapping $\mubf \in H^1(\U, \Prob( \U ))$ such that $\mubf |_{\dr \U} = \mubf_s$ can have a superposition principle: indeed, if it were the case, then we could restrict the superposition to $\dr \U$, and we would have a superposition principle for $\mubf_s$ with functions in $H^{1/2}$ which is a contradiction. To make this argument rigorous is a bit technical given the definition we chose for the boundary values of mappings in $H^1(\U, \Prob(\U))$: $\mubf$ is not necessarily continuous.

\begin{prop}
\label{proposition_failure_superposition_singular}
Let $\mubf \in H^1(\U, \Prob(\U))$ such that $\mubf |_{\dr \U} = \mubf_s$. Then $\mubf$ cannot admit a superposition principle. 
\end{prop}  

\begin{proof}
We will of course reason by contradiction. We assume that there exists $Q \in \Prob(\F)$ which satisfies the points (i) and (ii) of Definition \ref{definition_superposition_principle} (in fact only point (i) will be sufficient). Let $\Em = \vbf \mubf$ tangent to $\mubf$. Take $\delta > 0$ and $\varepsilon > 0$. We choose $\chi_\varepsilon \in C^1([0,1])$ an increasing function supported on $[1- \varepsilon, 1]$, such that $\chi_\varepsilon(1) = 1$. Define $a_\varepsilon \in C^1(\U, \R^2)$ and $b_\delta \in C^1(\U \times \U)$ by, for any $\xi, x \in \U$, 
\begin{align*}
a_\varepsilon(\xi) &= \frac{\xi}{|\xi|}  \chi_\varepsilon(|\xi|), \\ 
b_\delta(\xi, x) &= \frac{|\xi - x^2|^2}{\delta^2}.
\end{align*}
In words, $a_\varepsilon$ is a vector-valued function, parallel to lines issued from the origin, and whose norm is increasing on the annulus of radii $1 - \varepsilon$ and $1$ from $0$ to $1$. Define $A_\varepsilon = \{ \xi \in \U \ : \ 1 - \varepsilon \leqslant |\xi| \leqslant 1 \}$ the annulus outside which $a_\varepsilon$ vanishes. A simple computation gives
\begin{equation*}
\left| \nabla \cdot a_\varepsilon (\xi) - \chi'_\varepsilon(|\xi|) \right| \leqslant C \1_{A_\varepsilon}(\xi),
\end{equation*}  
where $C$ does not depend on $\varepsilon$. On the other hand, $b_\delta$ is a smooth scalar function, which vanishes if $x^2 = \xi$, which is larger than $1$ if $|x^2- \xi| \geqslant \delta$ and whose derivative is bounded by $C \delta^{-2}$. As a test function for the continuity equation, we take $\varphi(\xi, x) = a_\varepsilon(\xi) b_\delta(\xi, x)$. With this choice, for every $\xi \in \Sph^1$, one has 
\begin{equation*}
\int_{\U} \varphi(\xi, x) \mubf_s(\xi, \ddr x) = \frac{1}{2} \sum_{x^2 = \xi} \varphi(\xi, x) = 0.
\end{equation*}   
Thus, $\BT_{\mubf_s}(\varphi) = 0$ and the continuity equation tested against $\varphi$ reads
\begin{multline*}
\Bigg| \iint_{\U \times \U}  \chi'_\varepsilon(|\xi|)  b_\delta(\xi, x) \mubf(\ddr \xi, \ddr x) \\
 + \iint_{\U \times \U} [ a_\varepsilon(\xi) \cdot \nabla_\Omega b_\delta(\xi, x) + (a_\varepsilon(\xi) \otimes \nabla_D b_\delta(\xi, x)) \cdot \vbf(\xi, x)] \mubf(\ddr \xi, \ddr x) \Bigg| \leqslant C \varepsilon.
\end{multline*} 
Indeed, in the r.h.s, the reminder $\nabla \cdot a_\varepsilon - \chi'_\varepsilon(|\xi|)$ of order $1$ has been integrated over $A_\varepsilon$ whose area scales like $\varepsilon$. For the first integral, we use the assumption that $\mubf$ satisfies the superposition principle. For the second one, we bound $\nabla b_\delta$ by $C \delta^{-2}$, notice that $a_\varepsilon$ vanishes outside $A_\varepsilon$ and use Cauchy-Schwarz: 
\begin{align*}
\int_{\F} \left( \int_{\U} \chi'_\varepsilon(|\xi|) b_\delta(\xi, f(\xi)) \ddr \xi  \right) Q(\ddr f) & =  \iint_{\U \times \U}  \chi'_\varepsilon(|\xi|) b_\delta(\xi, x) \mubf(\ddr \xi, \ddr x) \\ 
& \leqslant \frac{C}{\delta^2} \iint_{A_\varepsilon \times \U} (1+|\vbf(\xi, x)|) \mubf(\ddr \xi, \ddr x) + C \varepsilon \\
& \leqslant  \frac{C}{\delta^2} \sqrt{\iint_{\U \times \U} (1 +|\vbf(\xi, x)|^2) \mubf(\ddr \xi, \ddr x)} \sqrt{\iint_{A_\varepsilon \times \U} \mubf(\ddr \xi, \ddr x)} + C \varepsilon  \\
& \leqslant \frac{C}{\delta^2} \sqrt{1 + 2 \Dir(\mubf)} \sqrt{\varepsilon} + C \varepsilon \leqslant C \frac{\sqrt{\varepsilon}}{\delta^2},
\end{align*} 
where $C$ denotes a generic constant which changes from one line to another and the inequality may hold only for small $\varepsilon$ and $\delta$. Let us call $\F_{\delta, \varepsilon} \subset \F$ the set of $f \in \F$ such that 
\begin{equation*}
\int_{\U} \chi'_\varepsilon(|\xi|) |f(\xi)^2 - \xi |^2 \ddr \xi \geqslant \delta^2.
\end{equation*}
By Markov's inequality, one can say that
\begin{multline*}
Q(\F_{\delta, \varepsilon}) = Q \left( \left\{ f \in \F \ : \ \int_{\U} \chi'_\varepsilon(|\xi|) b_\delta(\xi, f(\xi)) \ddr \xi \geqslant 1 \right\} \right) \\ 
\leqslant \int_{\F} \left( \int_{\U} \chi'_\varepsilon(|\xi|) b_\delta(\xi, f(\xi)) \ddr \xi  \right) Q(\ddr f) \leqslant C \frac{\sqrt{\varepsilon}}{\delta^2}.
\end{multline*}
Now take the sequence $\varepsilon_n := 2^{-n}$. By the previous estimate, one sees that 
\begin{equation*}
\sum_{n = 1}^{+ \infty} Q ( \F_{\delta, \varepsilon_n}) < + \infty.
\end{equation*}
By the Borel-Cantelli lemma, one has that $ Q ( \limsup_n \F_{\delta, \varepsilon_n}) = 0$ which means that for $Q$-a.e. $f \in \F$, there exists $n_0$ (which may depend on $f$) such that 
\begin{equation*}
\int_{\U} \chi'_{\varepsilon_n}(|\xi|) |f(\xi)^2 - \xi|^2 \ddr \xi \leqslant \delta^2
\end{equation*} 
for all $n \geqslant n_0$. Recall also that $Q$-a.e. $f$ belongs to $H^1(\Omega, D)$. For such an $f$, sending $n$ to $+ \infty$ and by definition of the trace of $f$, 
\begin{equation*}
\int_{\Sph^1} |\bar{f}(\xi)^2 - \xi|^2  \sigma(\ddr \xi) \leqslant \delta^2,
\end{equation*}
where in this formula $\bar{f}$ stands for the trace of $f$ on $\Sph^1$ and $\sigma$ the surface measure on $\dr \U$. Then using this estimate for smaller and smaller $\delta$ along a countable sequence, we conclude that $Q$-a.e. function $f$ satisfies $\bar{f}(\xi)^2 = \xi$ a.e. on $\Sph^1$. But on the other hand the trace of $Q$-a.e. function $f$ belongs to $H^{1/2}(\Sph^1, \U)$, which is a clear contradiction with Lemma \ref{lemma_H1/2_no_jumps}.   
\end{proof}

From this Proposition, we deduce that there exists an harmonic and a Lipschitz mapping $\mubf \in H^1(\Omega, \Prob(D))$ for which the superposition principle fails: just take respectively a solution of the Dirichlet problem with boundary values $\mubf_s$, or a Lipschitz extension of $\mubf_s$.

Though, these examples can seem too particular and rely too much on some singular boundary conditions. To produce stronger examples, we will use the fact that, roughly speaking, the set of $\mubf$ admitting a superposition principle is stable by approximation. Thus, by contraposition, any neighborhood of a $\mubf$ which does not admit a superposition principle will contain other measures not admitting a superposition principle. 

\begin{prop}
\label{proposition_superposition_approximation}
Let $(\mubf_n)_{n \in \N}$ a sequence of elements of $H^1(\Omega, \Prob(D))$ such that, for every $n \in \N$, $\mubf_n$ admits a superposition principle. We assume that $(\mubf_n)_{n \in \N}$ converges weakly to $\mubf \in H^1(\Omega, \Prob(D))$ and that $\lim_n \Dir(\mubf_n) = \Dir(\mubf)$. Then $\mubf$ admits a superposition principle. 
\end{prop}

\begin{proof}
For any $n \in \N$, let $Q_n \in \Prob(\F)$ such that (i) and (ii) of Definition \ref{definition_superposition_principle} are satisfied. By Rellich's theorem (recall that $D$ is compact), the functional $\Dir_c : \F \to \R$ has compact sublevel sets in the $L^2(\Omega, D)$-topology. As 
\begin{equation*}
\sup_{n \in \N} \int_{\F} \Dir_c(f) Q_n(\ddr f) = \sup_{n \in \N} \Dir(\mubf_n) < + \infty, 
\end{equation*}
we can say \cite[Remark 5.1.5]{Ambrosio2008} that $(Q_n)_{n \in \N}$ is tight, hence up to extraction it weakly converges in $\Prob(\F)$ to some $Q \in \Prob(\F)$. We will show that $Q$ represents $\mubf$. 

Let us take $a \in C(\Omega \times D)$ and define $A : \F \to \R$ by, for any $f \in \F$, 
\begin{equation*}
A(f) := \int_\Omega a(\xi,f(\xi)) \ddr \xi.
\end{equation*}
The function $A$ is continuous for the $L^2$ topology. Thus, starting from 
\begin{equation*}
\int_{\F} A(f) Q_n(\ddr f) = \iint_{\Omega \times D} a \ddr \mubf_n,
\end{equation*}
which is valid by Definition \ref{definition_superposition_principle}, we can pass both terms to the limit (recall that $\mubf_n$ weakly converges to $\mubf$) and see that $(\mubf, Q)$ satisfies (i) of Definition \ref{definition_superposition_principle}.

Moreover, as $\Dir_c$ is l.s.c. (for the $L^2(\Omega, D)$ topology), we can say that 
\begin{equation*}
\int_{\F} \Dir_c(f) Q(\ddr f) \leqslant \liminf_{n \to + \infty} \int_{\F} \Dir_c(f) Q_n(\ddr f) = \liminf_{n \to + \infty} \Dir(\mubf_n) = \Dir(\mubf), 
\end{equation*} 
which gives point (ii) of Definition \ref{definition_superposition_principle} and concludes the proof. 
\end{proof}

With this proposition, one can use for instance the heat flow to regularize mappings and produce ``smoother'' counterexamples. For instance, let $\mubf \in H^1(\U, \Prob(\U))$ which does not satisfy the superposition principle. Set $\mubf_n(\xi) := \Phi^{\U}_{1/n} \mubf(\xi)$: for a fixed $\xi \in \U$, we regularize $\mubf(\xi)$ with the help of the heat flow acting on $\Prob(\U)$. One can check easily that $\mubf_n$ converges weakly in $L^2(\U,\Prob(\U))$ to $\mubf$. As $\Phi^{\U}_{1/n}$ is a contraction in the Wasserstein space (Proposition \ref{proposition_heat_flow_properties}), $\Dir(\mubf_n) \leqslant \Dir(\mubf)$ and by lower semi-continuity of $\Dir$ we deduce that $\lim_n \Dir(\mubf_n) = \Dir(\mubf)$. According to Proposition \ref{proposition_superposition_approximation}, we deduce that $\mubf_n$ does not satisfy the superposition principle for $n$ large enough. On the other hand, the for any $\xi$ and any $n$ the measure $\mubf_n(\xi)$ is smooth: it admits a density bounded from below and from above.

\subsection{Local obstruction to the superposition principle}

The counterexample provided above shows a \emph{global} obstruction. Indeed, the mapping $\mubf_s$ can be thought locally in $\Omega$ as a superposition of classical functions, but there is a contradiction if we try to make this superposition global. On the other hand, there is also (at least formally) \emph{local obstructions} to the superposition principle. To describe them we will stay sloppy about the regularity issues and concentrate on heuristic explanations. 

Indeed, if $\mubf$ admits a superposition principle given by $Q \in \Prob(\F)$, and if $\vbf$ is the velocity field tangent to $\mubf$, then for $Q$-a.e. $f$, one has $\nabla f(\xi) = \vbf(\xi, f(\xi))$. To prove this fact, notice that the tangent momentum $\Em = \vbf \mubf$ is equal to $\int_\F \Em_f Q(\ddr f)$ (see the discussion preceding Definition \ref{definition_superposition_principle}), i.e. for any $b \in C(\Omega \times D, \R^{pq})$,
\begin{equation*}
\iint_{\Omega \times D} b \cdot \ddr \Em := \int_{\F} \left( \int_{\Omega} b(\xi, f(\xi)) \cdot \nabla f(\xi) \ddr \xi \right) Q(\ddr f).  
\end{equation*}
Thus, one can say that 
\begin{align*}
\Dir(\mubf) = \iint_{\Omega \times D} \frac{1}{2} |\vbf|^2 \ddr \mubf 
 = \iint_{\Omega \times D} \frac{1}{2} \vbf \cdot \ddr \Em 
& = \int_{\F} \left( \int_\Omega \frac{1}{2} \vbf(\xi, f(\xi)) \cdot \nabla f(\xi) \ddr \xi \right) Q(\ddr f)  \\
& \leqslant \int_{\F} \left( \int_\Omega \frac{1}{4} \left[ |\vbf(\xi, f(\xi))|^2 + |\nabla f(\xi)|^2 \right] \ddr \xi \right) Q(\ddr f) \\
& = \frac{1}{4} \iint_{\Omega \times D} | \vbf |^2 \ddr \mubf + \frac{1}{2} \int_{\F} \left( \int_\Omega \frac{1}{2} |\nabla f(\xi)|^2 \ddr \xi \right) Q(\ddr f) \\
& = \Dir(\mubf).  
\end{align*}
In particular, the inequality is an equality: one sees that for $Q$-a.e. $f \in \F$, one has $\nabla f(\xi) = \vbf(\xi, f(\xi))$ for a.e. $\xi \in \Omega$.

The analogue if $\Omega$ is a segment is the fact that (using notations from Theorem \ref{theorem_superposition_1D}) for $Q$-a.e. $f$, $\dot{f}(t) = \vbf(t, f(t))$: the measure $Q$ is supported on the flow of the vector field $\vbf$ (see \cite[Theorem 8.2.1]{Ambrosio2008}). In dimension larger than $1$, the constraint $\nabla f = \vbf(\cdot, f)$ is much stronger. In particular, it implies that along every curve $\gamma : I \to \Omega$, the function $f \circ \gamma$ follows the flow of $\vbf \cdot \dot{\gamma}$. However, there are many different curves going from one point to another: if we want all the results to be coherent, some commutation properties of the flow of $\vbf$ along different directions are needed, which turns out to be a very strong constraint. Indeed, coordinatewise, the constraint reads for every $\alpha \in \{ 1,2, \ldots, p \}$ and $i \in \{ 1,2, \ldots, q \}$,
\begin{equation*}
\dr_\alpha f^i(\xi) = \vbf^{\alpha i}(\xi, f(\xi)).
\end{equation*}
If we differentiate w.r.t. $\beta$, we find that 
\begin{equation*}
\dr_{\beta \alpha} f^i (\xi) = \dr_\beta \vbf^{\alpha i}(\xi, f(\xi)) + \sum_{j=1}^q \dr_\beta f^j(\xi) \dr_j \vbf^{\alpha i} (\xi, f(\xi)) = \left( \dr_\beta \vbf^{\alpha i} +  \sum_{j=1}^q \vbf^{\beta j} \dr_j \vbf^{\alpha i}  \right)(\xi, f(\xi)). 
\end{equation*}
The l.h.s is clearly symmetric if we exchange the role of $\alpha$ and $\beta$, so must be the r.h.s. It implies that for all $\alpha, \beta \in \{ 1,2, \ldots, p \}$, 
\begin{equation*}
\dr_\alpha \vbf^{\beta i} + \sum_{j=1}^q \vbf^{\alpha j} \dr_j \vbf^{\beta i}  = \dr_\beta \vbf^{\alpha i} + \sum_{j=1}^q \vbf^{\beta j} \dr_j \vbf^{\alpha i},
\end{equation*}
at least on the support of $\mubf$ in $\Omega \times D$. In other words, we see that $\vbf$ must satisfy a differential constraint for the superposition principle to hold, and there is no reason why this constraint would be satisfied for a generic $\mubf \in H^1(\Omega, \Prob(D))$, even for a harmonic mapping.  

An other way to understand the \emph{local} failure of the superposition principle is the following. We will be sloppy and use the evaluation operators $e_\xi : \F \to D$ defined by $e_\xi(f) := f(\xi)$ (these operators are in principle not defined as elements of $\F$ are not continuous). If $\mubf$ admits a superposition principle, it would mean that for $\xi$ and $\eta$ very close, $(e_\xi, e_\eta) \# Q \in \Prob(D \times D)$ is a transport plan between $\mubf(\xi)$ and $\mubf(\eta)$ (because of point (i)) which is almost optimal (between of point (ii)). It also works with three measures: if $\xi, \eta$ and $\theta$ are three points of $\Omega$ very close to each other (for instance located at the vertices of an equilateral triangle), then $(e_\xi, e_\eta, e_\theta) \# Q \in \Prob(D \times D \times D)$ is a coupling between $\mubf(\xi), \mubf(\eta)$ and $\mubf(\theta)$ whose $2$-marginals are almost optimal transport plans. However, it is known that, if $\mu_1, \mu_2$ and $\mu_3 \in \Prob(D)$, then in general there exists no coupling between the three whose $2$-marginals are optimal transport plans.   

\section{A Ishihara type property}
\label{section_Ishihara}

As explained in the introduction, we want to show in this section that $F \circ \mubf$ is subharmonic (which means $\Delta(F \circ \mubf) \geqslant 0$) as soon as $\mubf \in H^1(\Omega, \Prob(D))$ is harmonic and $F : \Prob(D) \to \R$ is convex along generalized geodesics. As far as the regularity of $F$ is concerned the simplest would be to assume that $F$ is continuous on $\Prob(D)$. Nevertheless, this assumption is very strong and excludes natural functionals (like the internal energies). In the case where $F$ is only l.s.c., we will need additional assumptions: it is the object of the following definition.

\begin{defi}
We say that $F : \Prob(D) \to \R \cup \{ + \infty \}$ is regular if it is l.s.c. on $\Prob(D)$, if
\begin{equation*}
\mubf \in L^2(\Omega, \Prob(D)) \mapsto \int_\Omega F(\mubf(\xi)) \ddr \xi
\end{equation*}
is l.s.c. for the weak convergence on $L^2(\Omega, \Prob(D))$, and if $F$ is bounded on the bounded sets of $L^\infty(D) \cap \Prob(D)$. 
\end{defi}

\noindent Lower semi-continuity of $F$ is a reasonable assumption. To impose that $F$ is bounded on bounded sets of $L^\infty(D) \cap \Prob(D)$ is not a strong constraint as $D$ is compact, we will need it to ensure that, by regularizing probability measures with the heat flow, we get measures for which $F$ is finite. 

Lower semi-continuity of $\mathfrak{F} : \mubf \mapsto \int_\Omega (F \circ \mubf)$ is less usual: by a standard argument left to the reader, it implies that $F$ is convex for the affine structure on $\Prob(D)$. However, we do not know in the general case if the fact that $F$ is convex and l.s.c. on $\Prob(D)$ is enough to ensure lower semi-continuity of $\mathfrak{F}$. Indeed, to apply abstract functional analysis arguments, we would like to work in the space $\M(\Omega \times D)$ endowed with the total variation norm: it is the dual of the Banach space $(C(\Omega \times D), \| \cdot \|_\infty)$. If $F$ is convex and l.s.c. on $\Prob(D)$, it can be shown easily that $\mathfrak{F}$ is convex and l.s.c. on $\M(\Omega \times D)$ endowed with the total variation norm. However, it only implies that $\mathfrak{F}$ is l.s.c. for the topology on $\M(\Omega \times D)$ defined by duality w.r.t. the dual of $\M(\Omega \times D)$, the latter being strictly larger than $C(\Omega \times D)$. 

However, for the usual functionals on $\Prob(D)$ we can do an \emph{ad hoc} analysis and we have the following results.   

\begin{prop}
\label{proposition_regular_functional}
Let $V \in L^1(D)$ a l.s.c. function. Then the functional
\begin{equation*}
F : \mu \in \Prob(D) \mapsto \int_D  V \ddr \mu
\end{equation*} 
is regular. 

Let $f : [0, + \infty) \to \R$ a proper and convex function such that $\dst{\lim_{t \to + \infty}} f(t)/t = + \infty$. Then the functional defined by
\begin{equation*}
F : \mu \in \Prob(D) \mapsto \begin{cases}
\dst{\int_D} f(\mu(x)) \ddr x & \text{if } \mu \text{ is absolutely continuous w.r.t. } \Leb_D \\
+ \infty & \text{else},
\end{cases}
\end{equation*} 
is regular. 
\end{prop}
  
\begin{proof}
As $V$ is l.s.c. on the compact $D$, it is bounded from below. As $V$ is in $L^1(\Omega)$, the function $F$ is clearly bounded on bounded sets of $L^\infty(D) \cap \Prob(D)$. Then, we can use \cite[Proposition 7.1]{OTAM}, seeing either $V$ as a l.s.c. function on $D$, or as a l.s.c. on $\Omega \times D$ (constant w.r.t. its first variable) to get that both $F$ and $\int_\Omega (F \circ \cdot)$ are l.s.c. 

For the internal energy, to get lower semi-continuity of $F$ we rely on \cite[Proposition 7.7]{OTAM}. To get the lower semi-continuity of $\int_\Omega (F \circ \cdot)$, we can see that 
\begin{equation*}
\int_\Omega F(\mubf(\xi)) \ddr \xi = \begin{cases}
\dst{\iint_{\Omega \times D} f(\mubf(\xi, x)) \ddr \xi \ddr x} & \text{if } \mubf \text{ is absolutely continuous w.r.t. } \Leb_\Omega \otimes \Leb_D \\
 + \infty & \text{else},
\end{cases}
\end{equation*}  
thus \cite[Proposition 7.7]{OTAM} still applies. As $f$ is bounded on bounded sets of $[0, + \infty)$, we see that $F$ is bounded on bounded sets of $L^\infty(D) \cap \Prob(D)$. 
\end{proof}  

\noindent However, the interaction energy is not regular: it lacks convexity w.r.t. the affine structure on $\Prob(D)$ \cite[Chapter 7]{OTAM}. For instance, take $\Omega = D = [0,1]$ and define $F : \Prob(D) \to \R$ by 
\begin{equation*}
F(\mu) := \iint_{D \times D} |x-y|^2 \mu(\ddr x) \mu(\ddr y).
\end{equation*}  
This functional is continuous on $\Prob(D)$ and bounded on bounded subsets of $L^\infty(D) \cap \Prob(D)$. However, if we define $\mubf_n(\xi) := \delta_{x_n(\xi)}$ with $x_n(\xi) = 1/2 + 1/2 \cos(n \xi)$, one can see that $F(\mubf_n(\xi)) = 0$ for all $\xi \in \Omega$ and $n \in \N$, but $(\mubf_n)_{n \in \N}$ converges weakly on $\Prob(\Omega \times D)$ to $\mubf := \Leb_\Omega \otimes \Leb_D$, for which the value $\int_\Omega (F \circ \mubf)$ is strictly positive. On the other hand, as soon as the interaction potential is continuous, the interaction energy is continuous on $\Prob(D)$.  

Finally, let us recall that a function $f : \Omega \to \R$ is said subharmonic on $\Or$ in the sense of distributions if $\Delta f \geqslant 0$ as a distribution in $\Or$. 

\begin{theo}
\label{theorem_ishihara}
Let $F : \Prob(D) \to \R \cup \{+ \infty \}$ a functional which is convex along generalized geodesics. Assume either that $F$ is continuous (and everywhere finite) on $\Prob(D)$ or that $F$ is regular. Let $\mubf_l : \dr \Omega \to \Prob(D)$ a Lipschitz mapping such that $\sup_{\dr \Omega} (F \circ \mubf_l) < + \infty$. 

Then there exists at least one solution $\mubf \in H^1(\Omega, \Prob(D))$ of the Dirichlet problem with boundary conditions $\mubf_l$ such that $(F \circ \mubf) : \Omega \to \R$ is subharmonic in $\Or$ in the sense of distributions and 
\begin{equation}
\label{equation_maximum_principle}
\esssup_{\Omega} (F \circ \mubf) \leqslant \sup_{\dr \Omega} (F \circ \mubf_l).
\end{equation}
Moreover, if $F$ is regular then $\mubf$ can be chosen in such a way that 
\begin{equation}
\label{equation_zz_aux_11}
\int_\Omega F (\mubf(\xi)) \ddr \xi \leqslant \int_\Omega F (\nubf(\xi)) \ddr \xi. 
\end{equation}
if $\nubf$ is any other solution of the Dirichlet problem with boundary values $\mubf_l$.
\end{theo} 

\noindent Let us make some comments. The first one is that \eqref{equation_maximum_principle} is nothing else than the maximum principle. It is not implied by the subharmonicity of $(F \circ \mubf)$ as the latter holds only in $\Or$ and we do not know if $(F \circ \mubf)$ is continuous. The second one is that  \eqref{equation_zz_aux_11} characterizes $\mubf$ if $F$ is strictly convex. More generally, the subharmonicity of $F \circ \mubf$ would hold for $\mubf$ solution of the Dirichlet problem minimizing
\begin{equation*}
\int_\Omega a(\xi) F (\mubf(\xi)) \ddr \xi,
\end{equation*}
where $a \in C(\Omega)$ is a continuous and strictly positive function (it comes from a slight modification of the proof which is left to the reader). The last comment is that this result is somehow disappointing because we cannot guarantee the subharmonicity to hold for all solutions. The main issue is that we reason by approximation, thus the solution $\mubf$ is constructed as the limit of some approximate mappings, the existence of the limit is coming from compactness. But as we have no uniqueness result for the Dirichlet problem, we can only identify the limit through \eqref{equation_zz_aux_11} (which is a byproduct of the approximation process) but we cannot say much more. 

The rest of this section is devoted to the proof of Theorem \ref{theorem_ishihara}. In Subsection \ref{subsection_Ishihara_preliminary} we prove some preliminary results. The most difficult and interesting case is the one where $F$ is not assumed to be continuous but only regular: it is the object of Subsections \ref{subsection_Ishihara_approximate_problems} and \ref{subsection_Ishihara_limit}. To conclude, in Subsection \ref{subsection_Ishihara_continuous}, we briefly comment about the simplifications of the proof in the case of a continuous $F$.    

\subsection{Preliminary results}
\label{subsection_Ishihara_preliminary}

We prove first some technical results which would have overburdened the previous sections. The first one deals with Rellich compactness theorem, as we will want some strong convergence of our solutions of the approximate problems. 

\begin{prop}
\label{proposition_Rellich}
Let $(\mubf_n)_{n \in \N}$ a sequence in $H^1(\Omega, \Prob(D))$ such that $\sup_n \Dir(\mubf_n) < + \infty$. Then, up to extraction, the sequence $(\mubf_n)_{n \in \N}$ converges strongly in $L^2(\Omega, \Prob(D))$ to some $\mubf \in H^1(\Omega, \Prob(D))$.   
\end{prop}

\begin{proof}
This is nothing else than the Rellich compactness theorem, but for mappings valued in metric spaces. Remark that $\Prob(D)$ has a finite diameter, thus in this result we only need a control on the Dirichlet energy of $\mubf_n$. We can find this result for instance in \cite[Theorem 1.13]{Korevaar1993} or in \cite[Theorem 5.4.3]{Ambrosio2003}. Any way, this result is also a consequence of the next proposition.   
\end{proof}

\noindent In fact, we will need a stronger result, as we want so show compactness if we only have a control of the approximate Dirichlet energies. 

\begin{prop}
\label{proposition_Rellich_approximate}
Let $(\mubf_\varepsilon)_{\varepsilon > 0}$ a family in $L^2(\Omega, \Prob(D))$ such that $\liminf_{\varepsilon} \Dir_\varepsilon(\mubf_\varepsilon) < + \infty$. Then there exists a sequence $(\varepsilon_n)_{n \in \N}$ which goes to $0$ such that $(\mubf_{\varepsilon_n})_{n \in \N}$ converges strongly in $L^2(\Omega, \Prob(D))$ to some $\mubf \in H^1(\Omega, \Prob(D))$.   
\end{prop}

\noindent There is a well known criterion for compactness in $L^2(\Omega)$: the Riesz-Fréchet-Kolmogorov theorem. It requires a uniform control of the $L^2$-norm of the difference between a function and its translated. Here, we have only a control of the distance between a function and its translation in average (thanks to $\Dir_\varepsilon$), and our mappings take values in $\Prob(D)$ rather than $\R$. Nevertheless, the strategy of the proof of the Riesz-Fréchet-Kolmogorov theorem is rather straightforward to adapt. Recall that $K^\Omega$ denotes the heat kernel on $\Omega$.

\begin{proof}
There exists a sequence $(\varepsilon_m)_{m \in \N}$, converging to $0$, such that $\sup_m \Dir_{\varepsilon_m}(\mubf_{\varepsilon_m}) < + \infty$. 

As in the proof of Theorem \ref{theorem_approximation}, let $\chi$ be a smooth function, radial, compactly supported in $B(0,1)$ and we set $\chi_t(\xi) = t^{-p} \chi(\xi/t)$. We will regularize $\mubf_{\varepsilon_m}$ only w.r.t. the source space $\Omega$. More specifically, for any $\tilde{\Omega}$ compactly supported in $\Or$ and $t$ small enough, we define $\tilde{\mubf}_{m,t} \in L^2(\tilde{\Omega}, \Prob(D))$ by
\begin{equation}
\tilde{\mubf}_{m,t}(\xi) := \int_{\Omega} \chi_t(\xi - \eta) \mubf_{\varepsilon_m}(\eta) \ddr \eta
\end{equation}
for any $\xi \in \tilde{\Omega}$. We first estimate $d_{L^2}(\tilde{\mubf}_{m,t}, \mubf_{\varepsilon_m}|_{\tilde{\Omega}})$. Using Jensen's inequality and the definition of $\Dir_{t}$,  
\begin{align*}
d_{L^2}(\tilde{\mubf}_{m,t}, \mubf_{\varepsilon_m}|_{\tilde{\Omega}}) & = 
\int_{\tilde{\Omega}} W_2^2 \left( \int_{B(0,t)} \chi_t(\eta) \mubf_{\varepsilon_m}(\xi - \eta) \ddr \eta, \mubf(\xi) \right) \ddr \xi \\
& \leqslant \int_{\tilde{\Omega}} \int_{B(0,t)} \chi_t(\eta) W_2^2 \left(  \mubf_{\varepsilon_m}(\xi - \eta) , \mubf(\xi) \right) \ddr \eta \ddr \xi \\
& \leqslant \frac{ 2 t^{p+2} \| \chi_t \|_\infty}{C_p} \Dir_{t}(\mubf_{\varepsilon_m}) = C t^2 \Dir_{t}(\mubf_{\varepsilon_m}). 
\end{align*}
Now, because of the monotonicity of $\Dir_t$ (Theorem \ref{theorem_equivalence_KS}) remember that $\Dir_{t}(\mubf_{\varepsilon_m}) \leqslant \Dir_{\varepsilon_m}(\mubf_{\varepsilon_m})$ if $m$ is large enough (and $t$ should in fact be of the form $2^N \varepsilon_m$ but it does not really matter). In consequence, for any $\delta > 0$, there exists $t > 0$ (small) and $m_0 \in \N$, such that for any $m \geqslant m_0$, 
\begin{equation*}
d_{L^2}(\tilde{\mubf}_{m,t}, \mubf_{\varepsilon_m}|_{\tilde{\Omega}}) \leqslant \delta. 
\end{equation*}    

On the other hand, for a fixed $t > 0$, we want to show compactness of the family $\left(\tilde{\mubf}_{m,t} \right)$ in $L^2(\tilde{\Omega}, \Prob(D))$. We will show that this family is uniformly equi-Hölder as mappings defined on $\tilde{\Omega}$ and valued in $(\Prob(D), W_2)$: it implies compactness in $C(\tilde{\Omega}, \Prob(D))$ from which we easily deduce compactness in $L^2(\tilde{\Omega}, \Prob(D))$. Here $\tilde{\Omega}$ is a compact subset of $\Omega$ lying at a distance larger than $t$ from $\dr \Omega$. We prefer to work on the $1$-Wasserstein distance whose definition is recalled in Section \ref{section_preliminairies}.  Take $\varphi \in C(D)$ a $1$-Lipschitz function, up to translation by a constant we can assume that $\| \varphi \|_\infty \leqslant C$ with $C$ independent of $\varphi$. Then for any $\xi, \eta \in \tilde{\Omega}$, 
\begin{align*}
\int_D \varphi(x) \tilde{\mubf}_{m,t}(\xi, \ddr x) - \int_D \varphi(x) \tilde{\mubf}_{m,t}(\eta, \ddr x)  & =  \iint_{ \tilde{\Omega}  \times D} \varphi(x) \left( \chi_t(\xi - \theta) - \chi_t(\eta - \theta) \right) \mubf(\theta, \ddr x) \ddr \theta \\
& \leqslant |\xi - \eta| \frac{1}{t^{p+1}} \| \chi' \|_\infty \| \varphi \|_\infty. 
\end{align*}
As the bound is independent on $\varphi$, we deduce that $W_1(\tilde{\mubf}_{m,t}(\xi), \tilde{\mubf}_{m,t}(\eta)) \leqslant C t^{-(p+1)} |\xi - \eta|$ for all $\xi$ and $\eta$ in $\tilde{\Omega}$. Using $W_2 \leqslant C \sqrt{W_1}$ \cite[Equation (5.1)]{OTAM}, we see that, for a fixed $t$, the family $(\tilde{\mubf}_{m,t})_{m \in \N}$, defined on $\tilde{\Omega},$ is uniformly equi-continuous (more precisely $1/2$-Hölder continuous).

Now we put the pieces together. For each $n \geqslant 1$, take $\tilde{\Omega}_n \subset \Or$ compactly supported in $\Or$ such that $\Leb_\Omega(\Omega \bsl \tilde{\Omega}_n) \leqslant 1/n$. Choose also $t_n$ small enough such that $d_{L^2}(\tilde{\mubf}_{m,t_n}, \left. \mubf_{\varepsilon_m} \right|_{\tilde{\Omega}_n}) \leqslant 1/n$ holds for $m$ large enough and the distance between $\tilde{\Omega}_n$ and $\dr \Omega$ is smaller than $t_n$. Then, using Ascoli-Arzelà theorem, up to a subsequence, we know that $(\tilde{\mubf}_{m,t_n})_{m \in \N}$ converges strongly in $L^2(\tilde{\Omega}_n, \Prob(D))$, in particular it is a Cauchy sequence. Up to a diagonal extraction in $(\varepsilon_m)_{m \in \N}$ (we do not relabel the sequence), we can assume that $(\left. \tilde{\mubf}_{m,t_n} \right|_{\tilde{\Omega}_n})_{m \in \N}$ is a Cauchy sequence for all $n \in \N$. Notice, as $\Prob(D)$ has a finite diameter, that $|d_{L^2}(\mubf, \nubf) -d_{L^2}(\left. \mubf \right|_{\tilde{\Omega}_n}, \left. \nubf \right|_{\tilde{\Omega}_n})| \leqslant C/n$ for all $\mubf, \nubf \in L^2(\Omega, \Prob(D))$. Hence, for any $n \in \N$, one has for $m$ and $m'$ large enough, 
\begin{align*}
d_{L^2}(\mubf_{\varepsilon_m}, \mubf_{\varepsilon_{m'}}) & 
\leqslant d_{L^2}(\left. \mubf_{\varepsilon_m} \right|_{\tilde{\Omega}_n}, \tilde{\mubf}_{m,t_n} ) + d_{L^2}(\tilde{\mubf}_{m,t_n}, \tilde{\mubf}_{m',t_n}) + d_{L^2}(\left. \mubf_{\varepsilon_{m'}} \right|_{\tilde{\Omega}_n}, \tilde{\mubf}_{m',t_n} ) + \frac{2C}{n} \\
& \leqslant \frac{2+2C}{n} + d_{L^2}(\tilde{\mubf}_{m,t_n}, \tilde{\mubf}_{m',t_n}), 
\end{align*}
and $d_{L^2}(\tilde{\mubf}_{m,t_n}, \tilde{\mubf}_{m',t_n})$ can be made arbitrary small for $m$ and $m'$ large enough. In other words, $(\mubf_{\varepsilon_m})_{m \in \N}$ is a Cauchy sequence in $L^2(\Omega, \Prob(D))$, thus it converges strongly. 
\end{proof}  

We will also need a result about the boundary conditions. Indeed, as the minimizers of $\Dir_\varepsilon$ will only live in $L^2(\Omega, \Prob(D))$, we cannot define and impose boundary values. To bypass this difficulty, we extend slightly our domain into a larger domain $\Omega_e \supset \Omega$ and impose the values of the mappings everywhere on $\Omega_e \bsl \Or$. 

\begin{prop}
\label{proposition_extension_boundary_values}
Let $\mubf_l : \dr \Omega \to \Prob(D)$ a Lipschitz mapping. There exists a compact $\Omega_e$ such that $\Omega \subset \ring{\Omega}_e$, and a Lipschitz mapping $\mubf_e \in L^2(\Omega_e \bsl \Or, \Prob(D)) $ such that $\mubf_e = \mubf_l$ on $\dr \Omega$ and
\begin{equation}
\label{equation_same_values_me_ml}
\{ \mubf_e(\xi) \ : \ \xi \in  \Omega_e \bsl \Or \} = \{ \mubf_l(\xi) \ : \ \xi \in  \dr \Omega \}.
\end{equation} 
Moreover, a mapping $\mubf \in H^1(\Omega, \Prob(D))$ satisfies $\mubf |_{\dr \Omega} = \mubf_l$ if and only if the mapping $\tilde{\mubf}$ defined on $\Omega_e$ by 
\begin{equation*}
\tilde{\mubf}(\xi) = \begin{cases}
\mubf(\xi) & \text{if } \xi \in \Or \\
\mubf_e(\xi) & \text{if } \xi \in \Omega_e \bsl \Or,
\end{cases}
\end{equation*}
belongs to $H^1(\Omega_e, \Prob(D))$.
\end{prop}

\begin{proof}
As $\Omega$ has a Lipschitz boundary, one can say \cite[Section 1.12]{Korevaar1993} that there exists a compact $\Omega_e$ such that $\Omega \subset \ring{\Omega}_e$, and $\Psi : [0,1] \times \dr \Omega  \to \Omega_e \bsl \Or$ a bilipschitz mapping such that $\Psi(0, \cdot)$ is the identity on $\dr \Omega$. Roughly speaking (for instance if $\dr \Omega$ is $C^1$), $\Psi(t,\xi) = \xi + t \nO(\xi)$ where $\nO$ is the outward normal to $\dr \Omega$. Then, one can define 
\begin{equation*}
\mubf_e(\Psi(t,\xi)) := \mubf_l(\xi) 
\end{equation*} 
for every $t \in [0,1]$ and $\xi \in \dr \Omega$: we extend $\mubf_l$ by keeping it constant along the normal to $\dr \Omega$. Because $\Psi$ is bilipschitz and $\mubf_l$ is Lipschitz, it is clear that $\mubf_e$ is a Lipschitz mapping. Moreover, by construction, \eqref{equation_same_values_me_ml} obviously holds.   

Let us prove the second point. Take $\Em \in \M(\Omega \times D, \R^{pq})$ and $\Em_e \in \M((\Omega_e \bsl \Or) \times D, \R^{pq})$ the momenta tangent to respectively $\mubf$ and $\mubf_e$. The tangent momentum of $\tilde{\mubf}$, if it were to exist, must coincide with $\Em$ on $\Omega \times D$ and with $\Em_e$ on $(\Omega_e \bsl \Or) \times D$ because of Corollary \ref{corollary_localization}. Hence, if must be $\tilde{\Em} \in \M(\Omega_e \times D, \R^{pq})$ defined by 
\begin{equation*}
\iint_{\Omega_e \times D} b \cdot \ddr \tilde{\Em} = \iint_{\Omega \times D} b \cdot \ddr \Em + \iint_{(\Omega_e \bsl \Or) \times D} b \cdot \ddr \Em_e.
\end{equation*}     
As we already have $\Dir(\tilde{\mubf}, \tilde{\Em}) < + \infty$, we see that $\tilde{\mu} \in H^1(\Omega_e, \Prob(D))$ if and only if $(\tilde{\mubf}, \tilde{\Em})$ satisfies the continuity equation. If $\varphi \in C^1_c(\Omega_e, \R^p)$, 
\begin{align*}
\iint_{\Omega_e \times D}& \nabla_\Omega \cdot \varphi  \ddr \tilde{\mubf} + \iint_{\Omega_e \times D} \nabla_D \varphi \cdot \ddr \tilde{\Em} \\
& = \iint_{\Omega \times D} \nabla_\Omega \cdot \varphi  \ddr \mubf + \iint_{\Omega \times D} \nabla_D \varphi \cdot \ddr \Em + \iint_{(\Omega_e \bsl \Or) \times D} \nabla_\Omega \cdot \varphi  \ddr \mubf_e + \iint_{(\Omega_e \bsl \Or) \times D} \nabla_D \varphi \cdot \ddr \Em_e \\
& = \BT_{\mubf}(\varphi) + \BT_{\mubf_e}(\varphi). 
\end{align*}
By Whitney's theorem, the restriction of functions in $C^1_c(\ring{\Omega}_e, \R^p)$ to $\Omega$ coincide with $C^1(\Omega, \R^p)$, thus we see that $\tilde{\mubf} \in H^1(\Omega_e, \Prob(D))$ if and only if $\BT_{\mubf} = - \BT_{\mubf_e}$. Considering the fact that the outward normal to $\Omega_e \bsl \Or$ is $- \nO$, and that $\mubf_e$ is continuous with values on $\dr \Omega$ given by $\mubf_l$, the proposition is proved. 
\end{proof}

\subsection{The approximate problems and their optimality conditions}
\label{subsection_Ishihara_approximate_problems}

In all this subsection, we assume that $F$ is regular. As explained before, we use $\Dir_\varepsilon$ to approximate $\Dir$, as the optimality conditions of $\Dir_\varepsilon$ imply that for each $\xi \in \Omega$, $\mubf(\xi)$ is a barycenter of all $\mubf(\eta)$ for $\eta$ in the ball of center $\xi$ and radius $\varepsilon$. 

Let us introduce some notations that we will keep during the rest of the proof. We denote by $\Omega_e \supset \Omega$ and $\mubf_e \in H^1(\Omega_e \bsl \Or, \Prob(D))$ the objects given by Proposition \ref{proposition_extension_boundary_values}. Take $\varepsilon_0 > 0$ such that $B(\xi, \varepsilon_0) \subset \Omega_e$ for all $\xi \in \dr \Omega$. We denote by 
\begin{equation*}
L^2_e(\Omega_e, \Prob(D)) := \{ \mubf \in L^2(\Omega_e, \Prob(D)) \ : \ \mubf |_{\Omega_e \bsl \Or} = \mubf_e \}
\end{equation*}  
the set of $L^2$ mappings which coincide with $\mubf_e$ on $\Omega_e \bsl \Or$. This set $L^2_e(\Omega_e, \Prob(D))$ is clearly closed for the weak convergence on $L^2(\Omega_e, \Prob(D))$, in particular it is compact for the weak convergence. We also define $H^1_e(\Omega_e, \Prob(D)) := H^1(\Omega_e, \Prob(D)) \cap L^2_e(\Omega_e, \Prob(D))$. In the rest of the proof, we extend the definitions of $\Dir_\varepsilon$ and $\Dir$ on $L^2_e(\Omega_e, \Prob(D))$. More precisely, if $\mubf \in L^2_e(\Omega_e, \Prob(D))$, 
\begin{equation*}
\Dir_\varepsilon(\mubf) := C_p \iint_{\Omega_e \times \Omega_e} \frac{W_2^2(\mubf(\xi), \mubf(\eta))}{2 \varepsilon^{p+2}} \1_{|\xi - \eta| \leqslant \varepsilon} \ddr \xi \ddr \eta,
\end{equation*}
and 
\begin{align*}
\Dir(& \mubf) \\
&:= \inf_{\Em} \left\{ \Dir(\mubf, \Em) \ : \ \Em \in \M(\Omega_e \times D, \R^{pq}) \text{ and } (\mubf,\Em) \text{ satisfies the continuity equation on } \Omega_e \times D \right\}.
\end{align*}
(we integrate over $\Omega_e$ and not only on $\Omega$). We also use the notation
\begin{equation*}
M := \sup_{\dr \Omega} (F \circ \mubf_l),
\end{equation*}
by assumption $M$ is finite. Remark that by construction, if $\mubf \in L^2_e(\Omega_e, \Prob(D))$, then for all $\xi \in \Omega_e \bsl \Or$ one has $F(\mubf(\xi)) \leqslant M$.

As $F$ is l.s.c. on the compact set $\Prob(D)$, it is bounded from below. Hence, we can translate it by a constant and assume that $F \geqslant 0$ on $\Prob(D)$. 
 
Let $\varepsilon > 0$ and $\lambda > 0$ be fixed. The approximate problem is defined as
\begin{equation}
\label{equation_approximate_pb}
\min_{\mubf} \left\{ \Dir_\varepsilon(\mubf) + \lambda \int_{\Omega_e} F(\mubf(\xi)) \ddr \xi \ : \ \mubf \in L^2_e(\Omega_e, \Prob(D)) \right\}.
\end{equation}     
To add the term $\lambda \int_{\Omega_e} F \circ \mubf$ has two purposes: on the one hand, it ensures that $F \circ \mubf$ will be regular enough (namely in $L^1(\Omega_e)$) to extract information from the optimality conditions; on the other hand by taking the limit $\varepsilon \to 0$ and then $\lambda \to 0$, we will be able to say that $F \circ \mubf_{\varepsilon, \lambda}$ (where $\mubf_{\varepsilon, \lambda}$ is a minimizer of the  approximate problem) converges pointewisely, and it is necessary to pass to the limit the (approximate) subharmonicity that we will get from the optimality conditions of the approximate problem. 

The following result is easy with all the tools developed above. 

\begin{prop}
\label{proposition_existence_solution_approximate}
For any $\varepsilon > 0$ and $\lambda > 0$, there exists a solution to the approximate problem \eqref{equation_approximate_pb}. 
\end{prop}

\begin{proof}
Let $\nu \in \Prob(D)$ any measure such that $F(\nu) < + \infty$ (it exists as $F$ is regular). If we define $\mubf \in L^2_e(\Omega_e, \Prob(D))$ by $\mubf |_{\Or} := \nu$ and $\mubf |_{\Omega_e \bsl \Or} := \mubf_e$, one can see that $\int_{\Omega_e} F(\mubf(\xi)) \ddr \xi < + \infty$, moreover as $\Prob(D)$ has a finite diameter $\Dir_\varepsilon(\mubf) < + \infty$. Hence, the minimization problem is non empty. In consequence, we are minimizing over the set $L^2_e(\Omega_e, \Prob(D))$, which is compact for the weak convergence, a functional which is l.s.c. (see Proposition \ref{proposition_Dir_eps_lsc} and the regularity assumption on $F$): we can use the direct method of calculus of variations.   
\end{proof}

Starting from now, for any $\varepsilon > 0$ and $\lambda > 0$, we denote by $\mubf_{\varepsilon, \lambda}$ a solution of the approximate problem \eqref{equation_approximate_pb}. 

\begin{prop}
\label{proposition_optimality_conditions_approximate}
Let $0 < \varepsilon \leqslant \varepsilon_0$ and $\lambda > 0$ be fixed. Then for a.e. $\xi \in \Omega$, $\mubf_{\varepsilon, \lambda}(\xi)$ is a minimizer over $\Prob(D)$ of 
\begin{equation*}
\nu \mapsto \frac{C_p}{\varepsilon^{p+2}} \int_{B(\xi, \varepsilon)} W_2^2(\nu, \mubf_{\varepsilon, \lambda}(\eta)) \ddr \eta + \lambda F(\nu). 
\end{equation*}
\end{prop}

\begin{proof}
We reason by contradiction. If the property does not hold, there exists $c > 0$ and a set $X \subset \Or$ of strictly positive measure such that for all $\xi \in X$, 
\begin{equation}
\label{equation_zz_aux_14}
\frac{C_p}{\varepsilon^{p+2}} \int_{B(\xi, \varepsilon)} W_2^2(\mubf_{\varepsilon, \lambda}(\xi), \mubf_{\varepsilon, \lambda}(\eta)) \ddr \eta + \lambda F(\mubf_{\varepsilon, \lambda}(\xi)) \geqslant c +  \min_{\nu \in \Prob(D)} \left( \frac{C_p}{\varepsilon^{p+2}} \int_{B(\xi, \varepsilon)} W_2^2(\nu, \mubf_{\varepsilon, \lambda}(\eta)) \ddr \eta + \lambda F(\nu) \right).
\end{equation}
Now, consider $\delta > 0$ small and $Y \subset X$ such that $\Leb_\Omega(Y) = \delta$. On every point of $\xi \in Y$, we want to select a minimizer $\nu$ (which depends on $\xi$) of the r.h.s. of \eqref{equation_zz_aux_14}, and we want to dot it in a measurable way. Notice that 
\begin{equation*}
\nu \mapsto \frac{C_p}{\varepsilon^{p+2}} \int_{B(\xi, \varepsilon)} W_2^2(\nu, \mubf_{\varepsilon, \lambda}(\eta)) \ddr \eta + \lambda F(\nu)
\end{equation*}  
is the sum of a functional continuous w.r.t. $\nu$ and measurable w.r.t. $\xi$, and the functional $\lambda F$ which is l.s.c. w.r.t. $\nu$ but which does not depend on $\xi$. The fact that $F$ is only l.s.c. prevents us from using directly Proposition \ref{proposition_Aliprantis1819}, though by some \emph{ad hoc} measurable selection result which is stated and proved in the appendix (Proposition \ref{proposition_measurable_selection_argmin}), one can still choose $\nu(\xi)$ a minimizer in such a way that it is measurable in $\xi$. In other words, we construct $\tilde{\mubf} \in L^2_e(\Omega_e, \Prob(D))$ such that $\tilde{\mubf} = \mubf_{\varepsilon, \lambda}$ on $\Omega_e \bsl Y$ and
\begin{equation*}
\frac{C_p}{\varepsilon^{p+2}} \int_{B(\xi, \varepsilon)} W_2^2(\mubf_{\varepsilon, \lambda}(\xi), \mubf_{\varepsilon, \lambda}(\eta)) \ddr \eta + \lambda F(\mubf_{\varepsilon, \lambda}(\xi)) \geqslant c +   \left( \frac{C_p}{\varepsilon^{p+2}} \int_{B(\xi, \varepsilon)} W_2^2(\tilde{\mubf}(\xi), \mubf_{\varepsilon, \lambda}(\eta)) \ddr \eta + \lambda F(\tilde{\mubf}(\xi)) \right)
\end{equation*} 
for all $\xi \in Y$. Now we evaluate: 
\begin{align*}
\Bigg( & \Dir_\varepsilon (\tilde{\mubf})  + \lambda \int_{\Omega_e} F(\tilde{\mubf}(\xi)) \ddr \xi \Bigg) - \left( \Dir_\varepsilon(\mubf_{\varepsilon, \lambda}) + \lambda \int_{\Omega_e} F(\mubf_{\varepsilon, \lambda}(\xi)) \ddr \xi \right) \\
& = \frac{C_p}{2 \varepsilon^{p+2}} \iint_{\Omega_e \times \Omega_e}  \left[ W_2^2(\mubf_{\varepsilon, \lambda}(\xi), \mubf_{\varepsilon, \lambda}(\eta)) - W_2^2(\tilde{\mubf}(\xi), \tilde{\mubf}(\eta)) \right] \1_{|\xi - \eta| \leqslant \varepsilon}   \ddr \xi \ddr \eta + \lambda \int_Y [F(\tilde{\mubf}(\xi)) - F(\mubf_{\varepsilon, \lambda}(\xi))] \ddr \xi 
\end{align*}
The integral over $\Omega_e \times \Omega_e$ can be split over four parts: the one over $(\Omega_e \bsl Y) \times (\Omega_e \bsl Y)$, which vanishes because $\mubf_{\varepsilon, \lambda} = \tilde{\mubf}$ on this set; the one over $Y \times Y$, which can be bounded by $C \delta^2$, where $C$ depends on the diameter of $\Prob(D)$ and on $\varepsilon$; and the ones over $(\Omega_e \bsl Y) \times Y$ and $Y \times (\Omega_e \bsl Y)$ which are equal by symmetry. Moreover, one has 
\begin{align*}
\frac{C_p}{2 \varepsilon^{p+2}} \iint_{Y \times (\Omega_e \bsl Y)}  & \left[ W_2^2(\mubf_{\varepsilon, \lambda}(\xi), \mubf_{\varepsilon, \lambda}(\eta)) - W_2^2(\tilde{\mubf}(\xi), \tilde{\mubf}(\eta)) \right] \1_{|\xi - \eta| \leqslant \varepsilon}   \ddr \xi \ddr \eta \\
& = \frac{C_p}{2 \varepsilon^{p+2}} \iint_{Y \times (\Omega_e \bsl Y)}   \left[ W_2^2(\mubf_{\varepsilon, \lambda}(\xi), \mubf_{\varepsilon, \lambda}(\eta)) - W_2^2(\tilde{\mubf}(\xi), \mubf_{\varepsilon, \lambda}(\eta)) \right] \1_{|\xi - \eta| \leqslant \varepsilon}  \ddr \xi \ddr \eta \\
& \leqslant C \delta^2 + \frac{C_p}{2 \varepsilon^{p+2}} \iint_{Y \times \Omega_e}  \left[ W_2^2(\mubf_{\varepsilon, \lambda}(\xi), \mubf_{\varepsilon, \lambda}(\eta)) - W_2^2(\tilde{\mubf}(\xi), \mubf_{\varepsilon, \lambda}(\eta)) \right]  \1_{|\xi - \eta| \leqslant \varepsilon}  \ddr \xi \ddr \eta \\
& = C \delta^2 + \frac{C_p}{2 \varepsilon^{p+2}} \int_Y \left( \int_{B(\xi, \varepsilon)} \left[ W_2^2(\mubf_{\varepsilon, \lambda}(\xi), \mubf_{\varepsilon, \lambda}(\eta)) - W_2^2(\tilde{\mubf}(\xi), \mubf_{\varepsilon, \lambda}(\eta)) \right] \ddr \eta \right) \ddr \xi, 
\end{align*}
where the inequality comes from the fact that we have add the piece $Y \times Y$ which is of size $\delta^2$ and one which we integrate a function which is bounded. Notice that we have used that $B(\xi, \varepsilon) \subset \Omega_e$ for $\xi \in \Omega$ as $\varepsilon < \varepsilon_0$. The part on $(\Omega_e \bsl Y) \times Y$ gives exactly the same amount, thus
\begin{align*}
&\Bigg(  \Dir_\varepsilon (\tilde{\mubf})  + \lambda \int_{\Omega_e} F(\tilde{\mubf}(\xi)) \ddr \xi \Bigg) - \left( \Dir_\varepsilon(\mubf_{\varepsilon, \lambda}) + \lambda \int_{\Omega_e} F(\mubf_{\varepsilon, \lambda}(\xi)) \ddr \xi \right) \\
& \leqslant C \delta^2 + \int_Y \left( \frac{C_p}{\varepsilon^{p+2}} \left[ \int_{B(\xi, \varepsilon)} [ W_2^2(\mubf_{\varepsilon, \lambda}(\xi), \mubf_{\varepsilon, \lambda}(\eta)) - W_2^2(\tilde{\mubf}(\xi), \mubf_{\varepsilon, \lambda}(\eta))] \ddr \eta  \right] + \lambda \left[ F(\tilde{\mubf}(\xi)) - F(\mubf_{\varepsilon, \lambda}(\xi)) \right] \right) \ddr \xi \\
& \leqslant C \delta^2 - c \delta,
\end{align*}  
where the last inequality comes precisely form the way we chose $\tilde{\mubf}$ on $Y$ and of $\Leb_\Omega(Y) = \delta$. Hence, taking $\delta$ small enough, the r.h.s. is strictly negative, which is a contradiction with the optimality of $\mubf_{\varepsilon, \lambda}$. 
\end{proof}

Remark that if $\lambda = 0$, our proof still works, and it precisely shows that $\mubf_{\varepsilon, 0}(\xi)$ is a barycenter of the $\mubf_{\varepsilon, 0}(\eta)$ for $\eta$ running over the ball of center $\xi$ and radius $\varepsilon$, a fact which was already stated by Jost \cite{Jost1994}. The crucial result which allows us to get subharmonicity is the following. If $\lambda = 0$, it is exactly Jensen's inequality for functionals convex along generalized geodesics, however here the situation is slightly different as $\mubf_{\varepsilon, \lambda}(\xi)$ is not really a barycenter. Notice that $F \circ \mubf_{\varepsilon, \lambda}$ is integrable on $\Omega_e$.  

\begin{prop}
\label{proposition_subharmonicity_approximate}
Let $0 <\varepsilon \leqslant \varepsilon_0$ and $\lambda > 0$ be fixed. Then, for a.e. $\xi \in \Omega$, 
\begin{equation*}
\int_{B(\xi, \varepsilon)} [F(\mubf_{\varepsilon, \lambda}(\eta)) - F(\mubf_{\varepsilon, \lambda}(\xi))] \ddr \eta \geqslant 0.  
\end{equation*}
\end{prop} 

\begin{proof}
Let us take a point $\xi \in \Omega$ for which the conclusion of Proposition \ref{proposition_optimality_conditions_approximate} holds and such that $F(\mubf_{\varepsilon,\lambda}(\xi)) < + \infty$: it is the case for a.e. points of $\Omega$. As a competitor, we use $S^F_t [\mubf_{\varepsilon,\lambda}(\xi)]$ for small $t > 0$, which means that we let $\mubf_{\varepsilon,\lambda}(\xi)$ follow the gradient flow of $F$, see Theorem \ref{theorem_EVI}. By Proposition \ref{proposition_optimality_conditions_approximate},
\begin{align*}
\frac{C_p}{\varepsilon^{p+2}} \int_{B(\xi, \varepsilon)} & W_2^2(\mubf_{\varepsilon,\lambda}(\xi), \mubf_{\varepsilon, \lambda}(\eta)) \ddr \eta + \lambda F(\mubf_{\varepsilon,\lambda}(\xi)) \\
& \leqslant \frac{C_p}{\varepsilon^{p+2}} \int_{B(\xi, \varepsilon)} W_2^2(S^F_t [\mubf_{\varepsilon,\lambda}(\xi)], \mubf_{\varepsilon, \lambda}(\eta)) \ddr \eta + \lambda F(S^F_t [\mubf_{\varepsilon,\lambda}(\xi)]).
\end{align*}
By the very definition of gradient flows, $F(S^F_t [\mubf_{\varepsilon,\lambda}(\xi)]) \leqslant F(\mubf_{\varepsilon,\lambda}(\xi))$. Thus, rearranging the terms and dividing by $2t > 0$, 
\begin{equation*}
\int_{B(\xi, \varepsilon)} \frac{W_2^2(S^F_t [\mubf_{\varepsilon,\lambda}(\xi)], \mubf_{\varepsilon, \lambda}(\eta)) - W_2^2(\mubf_{\varepsilon,\lambda}(\xi), \mubf_{\varepsilon, \lambda}(\eta))}{2 t} \ddr \eta \geqslant 0. 
\end{equation*}
For a.e. $\eta \in B(\xi, \varepsilon)$, one has that $F(\mubf_{\varepsilon, \lambda}(\eta)) < + \infty$. Hence, using Theorem \ref{theorem_EVI}, we see that for a.e. $\eta \in B(\xi, \varepsilon)$, the quantity 
\begin{equation*}
\frac{W_2^2(S^F_t [\mubf_{\varepsilon,\lambda}(\xi)], \mubf_{\varepsilon, \lambda}(\eta)) - W_2^2(\mubf_{\varepsilon,\lambda}(\xi), \mubf_{\varepsilon, \lambda}(\eta))}{2 t}
\end{equation*}
has a $\limsup$ bounded by $F(\mubf_{\varepsilon, \lambda}(\eta)) - F(\mubf_{\varepsilon, \lambda}(\xi))$ and is uniformly bounded in $t$ by $F(\mubf_{\varepsilon, \lambda}(\eta))$ (by Theorem \ref{theorem_EVI} and positivity of $F$), the latter being integrable on $B(\xi, \varepsilon)$. Hence, by Fatou's lemma, we can pass to the limit $t \to 0$ and conclude that
\begin{equation*}
\int_{B(\xi, \varepsilon)} [F(\mubf_{\varepsilon, \lambda}(\eta)) - F(\mubf_{\varepsilon, \lambda}(\xi))] \ddr \eta \geqslant 0. \qedhere
\end{equation*} 
\end{proof}

Let us conclude this subsection by proving a maximum principle, but for mappings which are $\varepsilon$-subharmonic. Recall that $M$ is the supremum of $F \circ \mubf$ on $\Omega_e \bsl \Or$ for any $\mubf \in L^2_e(\Omega_e, \Prob(D))$. 

\begin{prop}
\label{proposition_maximum_principle_approximate}
Let $ 0 < \varepsilon \leqslant \varepsilon_0$ and $\lambda > 0$ be fixed. Then, for a.e. $\xi \in \Omega_e$, one has $F(\mubf_{\varepsilon, \lambda}(\xi)) \leqslant M$. 
\end{prop} 

\begin{proof}
Let $\delta > 0$ be fixed and consider $f_\delta : \Omega_e \to \R \cup \{ + \infty \}$ defined by $f_\delta(\xi) = F(\mubf_{\varepsilon, \lambda}(\xi)) + \delta |\xi - \xi_0|^2$, where $\xi_0$ is any point of $\Omega$. By strict convexity of the square function and thanks to Proposition \ref{proposition_subharmonicity_approximate}, for a.e. $\xi \in \Omega$, 
\begin{equation*}
\int_{B(\xi, \varepsilon)} [f_\delta(\eta) - f_\delta(\xi)]  \ddr \eta > 0. 
\end{equation*}
In particular, the essential supremum of $f_\delta$ cannot be reached on $\Or$, it must be reached on $\Omega_e \bsl \Or$. On $\Omega_e \bsl \Or$ we control the values of $F \circ \mubf_{\varepsilon, \lambda}$ by $M$, in consequence $\mathrm{ess}\,\mathrm{sup}_{\Omega_e} f_\delta \leqslant M + C \delta$, where $C$ depends on the diameter of $\Omega$. Sending $\delta$ to $0$ (along a sequence), we get the result.  
\end{proof}

\subsection{Limit to the Dirichlet problem}
\label{subsection_Ishihara_limit}

In all this subsection, we still assume that $F$ is regular.

The goal is now to pass to the limit and to show that $\mubf_{\varepsilon, \lambda}$ converges to $\mubf$ a solution of the Dirichlet problem such that $F \circ \mubf$ is subharmonic. Recall that $\Dir_\varepsilon$ $\Gamma$-converges to $\Dir$ when $\varepsilon \to 0$, see Theorem \ref{theorem_equivalence_KS}. To get subharmonicity, we will need strong convergence, it implies to take first the limit $\varepsilon \to 0$ and then $\lambda \to 0$. But on the other hand, we need a uniform bound on the minimal values of the approximate problems to pass to the limit. To get them implies that we need to produce at least one mapping $\mubf$ in $H^1_e(\Omega_e, \Prob(D))$ such that $\int_{\Omega_e} (F \circ \mubf) < + \infty$. To do this, we cannot rely on the Lipschitz extension (Theorem \ref{theorem_lipschitz_extension}): there is no way to guarantee that $\int_\Omega (F \circ \mubf) < + \infty$. To get this uniform bound, we will take first the limit $\lambda \to 0$ and then $\varepsilon \to 0$ (relying only on weak convergence). It will produce a solution $\tilde{\mubf} \in H^1_e(\Omega_e, \Prob(D))$ of the Dirichlet problem with $\int_{\Omega_e} (F \circ \tilde{\mubf}) < + \infty$ but we cannot guarantee subharmonicity of $F \circ \tilde{\mubf}$. However it brings uniform bounds and enables us to take the limit $\varepsilon \to 0$, $\lambda \to 0$ and get a solution $\bar{\mubf}$ of the Dirichlet problem for which $F \circ \bar{\mubf}$ is subharmonic. 

We take two sequences $(\varepsilon_n)_{n \in \N}$, $(\lambda_m)_{m \in \N}$ that both converge to $0$ while being strictly positive. More precisely we take $\varepsilon_n := \varepsilon_0 2^{-n}$ for any $n \in \N$, thus we always have $\varepsilon_n \leqslant \varepsilon_0$ and $\Dir_{\varepsilon_n}$ converges in an increasing way and $\Gamma$-converges to $\Dir$. We will not relabel the sequences when extracting subsequences. Moreover, to avoid heavy notations, we will drop the indexes $n$ and $m$; and $\lim_{n \to + \infty}$, $\lim_{m \to + \infty}$ will be denoted respectively by $\lim_{\varepsilon \to 0}$ and $\lim_{\lambda \to 0}$. 

\begin{prop}
\label{proposition_tildemu_existence_property}
Up to extraction, there exists $\tilde{\mubf} \in H^1_e(\Omega_e, \Prob(D))$ such that 
\begin{equation*}
\tilde{\mubf} := \lim_{\varepsilon \to 0} \left( \lim_{\lambda \to 0} \mubf_{\varepsilon, \lambda} \right),
\end{equation*}
where the limits are taken weakly in $L^2_e(\Omega_e, \Prob(D))$. Moreover, $\tilde{\mubf}$ is a minimizer of $\Dir$ in the space $H^1_e(\Omega_e, \Prob(D))$ and 
\begin{equation}
\label{equation_zz_aux7}
\int_{\Omega_e} F(\tilde{\mubf}(\xi)) \ddr \xi < + \infty. 
\end{equation}
\end{prop}

\begin{proof}
The existence of $\tilde{\mubf} \in L^2_e(\Omega_e, \Prob(D))$ is trivial: recall that $L^2_e(\Omega_e, \Prob(D))$ is compact for the weak convergence. Moreover, using Proposition \ref{proposition_maximum_principle_approximate}, we have that for $\varepsilon \leqslant \varepsilon_0$ and $\lambda > 0$, 
\begin{equation*}
\int_{\Omega_e} F(\mubf_{\varepsilon, \lambda}(\xi)) \ddr \xi \leqslant M |\Omega_e|.
\end{equation*}
By the regularity assumption on $F$, we can pass this inequality to the weak limit and get \eqref{equation_zz_aux7}.  

The minimizing property of $\tilde{\mubf}$ is more involved. Assume by contradiction that there exists $\nubf \in H^1_e(\Omega, \Prob(D))$ such that $\Dir(\nubf) < \Dir(\tilde{\mubf})$. By the $\Gamma$-convergence of $\Dir_\varepsilon$ to $\Dir$ and the positivity of $F$, one has 
\begin{equation*}
\Dir(\nubf) < \Dir(\tilde{\mubf}) \leqslant \liminf_{\varepsilon \to 0} \left( \liminf_{\lambda \to 0} \left( \Dir_{\varepsilon}(\mubf_{\varepsilon, \lambda}) + \lambda \int_{\Omega_e} F(\mubf_{\varepsilon, \lambda}(\xi)) \ddr \xi  \right) \right).
\end{equation*}
In particular, we can choose $\varepsilon > 0$ small enough such that (by monotonicity of $\Dir_\varepsilon$) 
\begin{equation*}
\Dir_\varepsilon(\nubf) \leqslant \Dir(\nubf) < \liminf_{\lambda \to 0} \left( \Dir_{\varepsilon}(\mubf_{\varepsilon, \lambda}) + \lambda \int_{\Omega_e} F(\mubf_{\varepsilon, \lambda}(\xi)) \ddr \xi  \right).
\end{equation*}
We regularize $\nubf$ in the following way: for $t> 0$, we denote by $\nubf_t := (\1_{\Or} \Phi^D_t) \nubf$ the element of $L^2_e(\Omega_e, \Prob(D))$ for which the heat flow on $D$ has been followed only in $\Or$: in other words, for any $t > 0$, 
\begin{equation*}
\nubf_t(\xi) := \begin{cases}
(\Phi^D_t) [\nubf(\xi)] & \text{if } \xi \in \Or, \\
\nubf(\xi) = \mubf_e(\xi) & \text{if } \xi \in \Omega_e \bsl \Or.
\end{cases}
\end{equation*}  
Clearly, $\nubf_t \in L^2_e(\Omega_e, \Prob(D))$. Moreover, as $W_2(\nubf_t(\xi), \nubf(\xi)) \leqslant \omega^D(t)$ with $\omega^D(t) \to 0$ as $t \to 0$ (see Proposition \ref{proposition_heat_fow_uniform_0}), we see that $\nubf_t$ converges strongly in $L^2_e(\Omega_e, \Prob(D))$ to $\nubf$. In particular, thanks to the continuity of $\Dir_\varepsilon$, there exists $t$ small enough such that 
\begin{equation*}
\Dir_\varepsilon(\nubf_t) < \liminf_{\lambda \to 0} \left( \Dir_{\varepsilon}(\mubf_{\varepsilon, \lambda}) + \lambda \int_{\Omega_e} F(\mubf_{\varepsilon, \lambda}(\xi)) \ddr \xi  \right). 
\end{equation*} 
Because of the standard $L^\infty-L^1$ estimate for the heat flow (see (ii) of Proposition \ref{proposition_heat_flow_properties}), one has that $\{ \nubf_t(\xi) \ : \ \xi \in \Or \}$ is included in a bounded set of $L^\infty(D) \cap \Prob(D)$. In particular, $F \circ \nubf_t$ is bounded on $\Or$. As it is also bounded on $\Omega_e \bsl \Or$ by $M$, we see that $\int_{\Omega_e} F \circ \nubf_t < + \infty$. Hence, for some $\lambda$ small enough, 
\begin{equation*}
\Dir_\varepsilon(\nubf_t) + \lambda \int_{\Omega_e} F(\nubf_t(\xi)) \ddr \xi  < \Dir_{\varepsilon}(\mubf_{\varepsilon, \lambda}) + \lambda \int_{\Omega_e} F(\mubf_{\varepsilon, \lambda}(\xi)) \ddr \xi,
\end{equation*}
which is a contradiction with the optimality of $\mubf_{\varepsilon, \lambda}$. 
\end{proof}

\begin{prop}
\label{proposition_existence_barmubf_ishihara}
Up to extraction, there exists $\bar{\mubf} \in H^1_e(\Omega_e, \Prob(D))$ such that
\begin{equation*}
\bar{\mubf} := \lim_{\lambda \to 0} \left( \lim_{\varepsilon \to 0} \mubf_{\varepsilon, \lambda} \right),
\end{equation*}
where the limits are taken strongly in $L^2_e(\Omega_e, \Prob(D))$. Moreover, $\bar{\mubf}$ is a minimizer of $\Dir$ in the space $H^1_e(\Omega_e, \Prob(D))$ and for any other minimizer $\nubf$ of $\Dir$ in $H^1_e(\Omega_e, \Prob(D))$, 
\begin{equation}
\label{equation_zz_aux_10}
\int_{\Omega_e} F(\bar{\mubf}(\xi)) \ddr \xi \leqslant \liminf_{\lambda \to 0} \left( \liminf_{\varepsilon \to 0} \left( \int_{\Omega_e} F(\mubf_{\varepsilon, \lambda}(\xi)) \ddr \xi \right) \right) \leqslant \int_{\Omega_e} F(\nubf(\xi)) \ddr \xi.  
\end{equation}
\end{prop}

\begin{proof}
Using $\tilde{\mubf}$ as a competitor in the approximate problem, given the monotonicity of $\Dir_\varepsilon$, one has that 
\begin{equation*}
\Dir_\varepsilon(\mubf_{\varepsilon, \lambda}) + \lambda \int_{\Omega_e} F(\mubf_{\varepsilon, \lambda}(\xi)) \ddr \xi \leqslant \Dir(\tilde{\mubf}) + \lambda \int_{\Omega_e} F(\tilde{\mubf}(\xi)) \ddr \xi \leqslant C,
\end{equation*}
where the constant $C$ is uniform in $\varepsilon > 0$ and $0 <\lambda \leqslant 1$. In particular, using the Rellich-like theorem (Proposition \ref{proposition_Rellich_approximate}), we see that, up to extraction, $\mubf_{\varepsilon, \lambda}$ converges strongly in $L^2_e(\Omega_e, \Prob(D))$ to some $\bar{\mubf}_\lambda$ when $\varepsilon \to 0$. Moreover, by $\Gamma$-convergence of $\Dir_\varepsilon$ and the regularity of $F$, 
\begin{equation}
\label{equation_zz_aux8}
\Dir(\bar{\mubf}_\lambda) + \lambda \int_{\Omega_e} F(\bar{\mubf}_\lambda(\xi)) \ddr \xi \leqslant \liminf_{\varepsilon \to 0} \left( \Dir_\varepsilon(\mubf_{\varepsilon, \lambda}) + \lambda \int_{\Omega_e} F(\mubf_{\varepsilon, \lambda}(\xi)) \ddr \xi \right) \leqslant C. 
\end{equation} 
Hence, we have a uniform bound on $\Dir(\bar{\mubf}_\lambda)$, and we can apply Rellich theorem (Proposition \ref{proposition_Rellich}) to see that $\bar{\mubf}_\lambda$ converges strongly in $L^2(\Omega_e, \Prob(D))$ to some $\bar{\mubf} \in H^1_e(\Omega_e, \Prob(D))$ when $\lambda \to 0$. Moreover, using the lower semi-continuity of $\Dir$ and positivity of $F$, 
\begin{equation}
\label{equation_zz_aux9}
\Dir(\bar{\mubf}) \leqslant \liminf_{\lambda \to 0} \left( \Dir(\bar{\mubf}_\lambda) + \lambda \int_{\Omega_e} F(\bar{\mubf}_\lambda(\xi)) \ddr \xi \right).
\end{equation}  

Let us assume by contradiction that $\bar{\mubf}$ is not a minimizer of $\Dir$. Thanks to Proposition \ref{proposition_tildemu_existence_property}, it boils down to assume that $\Dir(\tilde{\mubf}) < \Dir(\bar{\mubf})$. In particular, as $F \circ \tilde{\mubf}$ is integrable on $\Omega_e$ and with the help of \eqref{equation_zz_aux9}, it means that there exists $\lambda$ small enough such that 
\begin{equation*}
\Dir(\tilde{\mubf}) + \lambda \int_{\Omega_e} F(\tilde{\mubf}(\xi)) \ddr \xi < \Dir(\bar{\mubf}_\lambda) + \lambda \int_{\Omega_e} F(\bar{\mubf}_\lambda(\xi)) \ddr \xi. 
\end{equation*} 
Using the fact that $\Dir_\varepsilon(\tilde{\mubf}) \to \Dir (\tilde{\mubf})$ to handle the l.h.s. and \eqref{equation_zz_aux8} to deal with the r.h.s., we see that for $\varepsilon > 0$ small enough, 
\begin{equation*}
\Dir_\varepsilon(\tilde{\mubf}) + \lambda \int_{\Omega_e} F(\tilde{\mubf}(\xi)) \ddr \xi < \Dir_{\varepsilon}(\mubf_{\varepsilon, \lambda}) + \lambda \int_{\Omega_e} F(\mubf_{\varepsilon, \lambda}(\xi)) \ddr \xi,
\end{equation*} 
which is a contradiction with the optimality of $\mubf_{\varepsilon, \lambda}$. Hence, $\bar{\mubf}$ is a minimizer of $\Dir$ over $H^1_e(\Omega_e, \Prob(D))$.

Remark that in \eqref{equation_zz_aux_10} the first inequality is a consequence of the fact that $F$ is regular. Assume by contradiction that there exists $\nubf \in H^1_e(\Omega_e, \Prob(D))$ a minimizer of $\Dir$ such that \eqref{equation_zz_aux_10} does not hold. In particular as $\Dir(\bar{\mubf}) = \Dir(\nubf)$, and by $\Gamma$-convergence of $\Dir_\varepsilon$ and lower semi-continuity of $\Dir$,
\begin{equation*}
\Dir(\nubf) = \Dir(\bar{\mubf}) \leqslant \liminf_{\lambda \to 0} \left( \liminf_{\varepsilon \to 0} \left( \Dir_\varepsilon(\mubf_{\varepsilon, \lambda}) \right) \right),
\end{equation*}
thus one can write that for some $\lambda$ small enough, 
\begin{equation*}
\Dir(\nubf) + \lambda \int_{\Omega_e} F(\nubf(\xi)) \ddr \xi < \liminf_{\varepsilon \to 0} \left( \Dir_{\varepsilon}(\mubf_{\varepsilon, \lambda}) + \lambda \int_{\Omega_e} F(\mubf_{\varepsilon, \lambda}(\xi)) \ddr \xi \right):
\end{equation*}
it leads to the same contradiction as before by taking $\varepsilon > 0$ small enough.
\end{proof}

Now, the key result to get subharmonicity of $F \circ \bar{\mubf}$ is that we can pass at the pointwise limit the quantity $F \circ \mubf_{\varepsilon, \lambda}$. 

\begin{prop}
\label{proposition_simple_cv_Fmu}
For a.e. $\xi \in \Omega$, 
\begin{equation*}
F(\bar{\mubf}(\xi)) = \lim_{\lambda \to 0} \left( \lim_{\varepsilon \to 0} \left( F(\mubf_{\varepsilon, \lambda}(\xi)) \right) \right). 
\end{equation*}
\end{prop}

\begin{proof}
As the convergence of $\mubf_{\varepsilon, \lambda}$ to $\bar{\mubf}$ holds strongly in $L^2_e(\Omega, \Prob(D))$, we can, up to extraction, assume that it holds a.e. In other words, for a.e. $\xi \in \Omega$, 
\begin{equation*}
\bar{\mubf}(\xi) =  \lim_{\lambda \to 0} \left( \lim_{\varepsilon \to 0} \left( \mubf_{\varepsilon, \lambda}(\xi) \right) \right)
\end{equation*}
in $\Prob(D)$. By lower semi-continuity of $F$ on $\Prob(D)$, the inequality
\begin{equation*}
F(\bar{\mubf}(\xi)) \leqslant \liminf_{\lambda \to 0} \left( \liminf_{\varepsilon \to 0} \left( F(\mubf_{\varepsilon, \lambda}(\xi)) \right) \right)
\end{equation*} 
holds for a.e. $\xi \in \Omega$. On the other hand, use \eqref{equation_zz_aux_10} with $\nubf = \bar{\mubf}$: up to extraction one has 
\begin{equation*}
\int_{\Omega_e} F(\bar{\mubf}(\xi)) \ddr \xi = \lim_{\lambda \to 0} \left( \lim_{\varepsilon \to 0} \left( \int_{\Omega_e} F(\mubf_{\varepsilon, \lambda}(\xi)) \ddr \xi \right) \right).  
\end{equation*}  
By combining the two equations above (recall that all the functions $F \circ \mubf_{\varepsilon, \lambda}$ and $F \circ \bar{\mubf}$ are positive and bounded above by $M$ thanks to Proposition \ref{proposition_maximum_principle_approximate}), we reach the desired conclusion (this is just an adaptation of the proof of Scheffé's lemma). 
\end{proof}

\begin{prop}
\label{proposition_conclusion_Ishihara}
The function $F \circ \bar{\mubf}$ is subharmonic on $\Or$. Moreover, 
\begin{equation*}
\esssup_{\Omega} (F \circ \bar{\mubf}) \leqslant M. 
\end{equation*}
\end{prop}

\begin{proof}
The fact that the essential supremum of $F \circ \bar{\mubf}$ is bounded by $M$ is a simple combination of Propositions \ref{proposition_maximum_principle_approximate} and \ref{proposition_simple_cv_Fmu}. For the subharmonicity, take $\psi \in C^\infty_c(\Or)$ a smooth and positive function compactly supported in $\Or$. For $0 < \varepsilon \leqslant \varepsilon_0$ small enough, one has, thanks to Proposition \ref{proposition_subharmonicity_approximate}, 
\begin{equation*}
\int_{\Omega_e} \psi(\xi) \left( \frac{1}{\varepsilon^{d+2}} \int_{B(\xi, \varepsilon)} [F(\mubf_{\varepsilon, \lambda}(\eta)) - F(\mubf_{\varepsilon, \lambda}(\xi))] \ddr \eta \right) \ddr \xi \geqslant 0. 
\end{equation*}
Performing a discrete integration by parts (which is possible if $\varepsilon$ is smaller than the distance between $\dr \Omega$ and the support of $\psi$), one sees that 
\begin{equation*}
\int_{\Omega_e} F(\mubf_{\varepsilon, \lambda}(\xi)) \left( \frac{1}{\varepsilon^{d+2}} \int_{B(\xi, \varepsilon)} [\psi(\eta) - \psi(\xi)] \ddr \eta \right) \ddr \xi \geqslant 0. 
\end{equation*}
Now send $\varepsilon \to 0$ and then $\lambda \to 0$. By smoothness of $\psi$, the quantity $\varepsilon^{-(d+2)} \int_{B(\xi, \varepsilon)} [\psi(\eta) - \psi(\xi)] \ddr \eta$ converges to $\Delta \psi(\xi)$ (up to a multiplicative constant). On the other hand, $F(\mubf_{\varepsilon, \lambda}(\xi))$ converges pointwisely to $F(\bar{\mubf})$ (see Proposition \ref{proposition_simple_cv_Fmu}) while being bounded by $M$. By Lebesgue dominated convergence theorem, 
\begin{equation*}
\int_{\Omega_e} F(\mubf(\xi)) \Delta \psi(\xi) \ddr \xi \geqslant 0,
\end{equation*} 
which exactly means that $F \circ \mubf$ is subharmonic in the sense of distributions as $\psi$ is an arbitrary smooth and positive function. 
\end{proof}

Now we can conclude: 

\begin{proof}[Proof of Theorem \ref{theorem_ishihara} if $F$ is regular]
We take $\mubf$ the restriction of $\bar{\mubf}$ to $\Omega$. Thanks to Proposition \ref{proposition_extension_boundary_values}, the fact that $\bar{\mubf}$ is a minimizer of $\Dir$ among $H^1_e(\Omega_e, \Prob(D))$ is translated into the fact that $\mubf$ is a solution of the Dirichlet problem with boundary values $\mubf_l$. The subharmonicity and the upper bound of $F \circ \bar{\mubf}$ are preserved by restriction. To get the minimality of $\int_\Omega (F \circ \bar{\mubf})$, we just use \eqref{equation_zz_aux_10}.   
\end{proof}

\subsection{Simplifications in the continuous case}
\label{subsection_Ishihara_continuous}

In this subsection, we assume that $F$ is continuous. In particular, as $\Prob(D)$ is compact, it implies that $F$ is bounded. The proof is simpler because we do not need to add the term $\lambda \int F \circ \mubf$ in the approximate problem. Indeed, strong convergence in $L^2(\Omega, \Prob(D))$ of a sequence $\mubf_n$ to $\mubf$ implies, up to extraction, the convergence a.e. of $(F \circ \mubf_n)$ to $(F \circ \mubf)$. 

We define $\Omega_e, \mubf_e$ and the functional spaces $L^2_e(\Omega_e, \Prob(D)), H^1_e(\Omega_e, \Prob(D))$ as in the beginning of Subsection \ref{subsection_Ishihara_approximate_problems}. 

\begin{proof}[Proof of Theorem \ref{theorem_ishihara} if $F$ is continuous]
For any $\varepsilon > 0$, we take $\mubf_\varepsilon \in L^2_e(\Omega_e, \Prob(D))$ a minimizer of $\Dir_\varepsilon$ over $L^2_e(\Omega_e, \Prob(D))$. 

We can still apply Proposition \ref{proposition_optimality_conditions_approximate} and conclude that for a.e. $\xi \in \Omega$, $\mubf_\varepsilon(\xi)$ is a barycenter of the $\mubf_\varepsilon(\eta)$ for $\eta \in B(\xi, \varepsilon)$. The proof of Jensen's inequality (Proposition \ref{proposition_subharmonicity_approximate}) works in the same way as $F$ is bounded on $\Prob(D)$. Hence, the maximum principle given by Proposition \ref{proposition_maximum_principle_approximate} is still true as it is only implied by Proposition \ref{proposition_subharmonicity_approximate}. 

To pass to the limit $\varepsilon \to 0$, we use the fact that (along an appropriate sequence) $\Dir_\varepsilon$ $\Gamma$-converges to $\Dir$. Hence, up to extraction, $\mubf_\varepsilon$ converges to $\bar{\mubf}$ which is a minimizer of $\Dir$ over $L^2_e(\Omega_e, \Prob(D))$. Thanks to Proposition \ref{proposition_Rellich_approximate}, the convergence takes place strongly in $L^2_e(\Omega_e, \Prob(D))$ and a.e. By continuity of $F$, we deduce that the conclusion of Proposition \ref{proposition_simple_cv_Fmu} still holds: $F \circ \mubf_\varepsilon$ converges a.e. to $F \circ \bar{\mubf}$ as $\varepsilon \to 0$. Thus the proof of Proposition \ref{proposition_conclusion_Ishihara} works exactly in the same way and it is enough to take for $\mubf$ the restriction of $\bar{\mubf}$ to $\Omega$.   
\end{proof}

\section{Examples}
\label{section_examples}

To conclude, we give examples of situations where the computation of harmonic mappings can be done explicitly. The first one is rather simple: when $D$ is a segment of $\R$ the space $\Prob(D)$ has a structure of Hilbert space, hence the study is considerably simpler and all the machinery developed above is too heavy. The second one is trickier: we restrict ourselves to a family of elliptically contoured distributions, which is a geodesically convex subset of finite dimension. Thus we end up with mappings valued in a finite-dimensional Riemannian manifold, on which we can write explicit Euler-Lagrange equations.  

Before studying these examples, let us say that the case where the measures on $\dr \Omega$ are Dirac masses has already been treated. Indeed, we have already mentioned \cite[Theorem 3.1]{Brenier2003} which states that if $f : \Omega \to D$ is harmonic, then the measure $\mubf_f$ (defined by $\mubf_f(\xi) = \delta_{f(\xi)}$ for all $\xi \in \Omega$) is a measured-valued harmonic mapping. This result has been extended in \cite[Theorem 3.3]{Lu2017} to the case where space $D$ is a simply connected manifold with negative curvature.   

\subsection{One dimensional target}

In this subsection, we assume that $D = I = [0,1]$ is the unit interval. The important point is that the space $\Prob(I)$ has a very simple structure: the right object to characterize an element $\mu \in \Prob(D)$ is its inverse distribution function $F^{[-1]}_\mu : [0,1] \to [0,1]$ defined by 
\begin{equation*}
F^{[-1]}_\mu(t) := \inf \{ x \in [0,1] \ : \ \mu([0, x]) \geqslant t \}. 
\end{equation*} 
It is well known that $F^{[-1]}_\mu$ is increasing, right continuous, and that there is a bijection between the set of increasing and right continuous mappings $[0,1] \to [0,1]$ and $\Prob(I)$. Moreover, for any $\mu, \nu \in \Prob(I)$, one has (see for instance \cite[Proposition 2.17]{OTAM}) 
\begin{equation}
\label{equation_wasserstein_distance_1D}
W_2^2(\mu, \nu) = \int_0^1 |F^{[-1]}_\mu(t) - F^{[-1]}_\nu(t)|^2 \ddr t. 
\end{equation} 
Introduce the Hilbert space $\Hil := L^2([0,1])$ with its usual norm (denoted by $|\cdot|_\Hil$) and the subspace $\Hili$ of increasing functions: if $f \in \Hil$, then we say that $f \in \Hili$ if $f(t) \in [0,1]$ for a.e. $t \in [0,1]$ and if for any $0 \leqslant t_1 < t_2 \leqslant t_3 < t_4 \leqslant 1$, one has 
\begin{equation*}
\frac{1}{t_2 - t_1} \int_{t_1}^{t_2} f(t) \ddr t \leqslant \frac{1}{t_4 - t_3} \int_{t_3}^{t_4} f(t) \ddr t
\end{equation*}
Notice that $\Hili$ is clearly a convex and closed subset of $\Hil$. Any $f \in \Hili$ has a unique increasing and right continuous representative. Indeed, take the representative given by the Lebesgue differentiation theorem: except on a subset $N$ which is negligible, it is increasing. Then, on $N$ and on any point of discontinuity, choose the right limit. Uniqueness is easy as any increasing and right continuous representative is continuous except at a countable number of points. This discussion can be summarized in the following proposition. 

\begin{prop}
If we define $\Psi(\mu) := F^{[-1]}_\mu$, then $\Psi$ is a one-to-one isometry between $\Prob(I)$ and $\Hili$. 
\end{prop}   

Now we need to make the bridge between the Dirichlet energy in the space $H^1(\Omega, \Prob(I))$ and the one in  $H^1(\Omega, \Hil)$. In fact, it was already proved by Korevaar and Schoen \cite{Korevaar1993} that their definition of Dirichlet energy coincides with the usual one if the target space is $\R$. By Pythagore's theorem, the equivalence still holds if the target space is a separable Hilbert space, as one can work on the coordinates in an orthogonal basis. As our definition of Dirichlet energy coincides with the one of Korevaar and Schoen, see Theorem \ref{theorem_equivalence_KS}, we can conclude that  
\begin{equation}
\label{equation_zz_aux_30}
\Dir(\mubf) := \int_{\Omega} |\nabla (\Psi \circ \mubf)(\xi)|^2_{\Hil} \ddr \xi. 
\end{equation}
for any $\mubf \in H^1(\Omega, \Prob(I))$. Thus, we can say the following: 

\begin{theo}
Let $\mubf_l : \dr \Omega \to \Prob(I)$ a given Lipschitz mapping. Then there exists a unique $\mubf \in H^1(\Omega, \Prob(I))$ solution of the Dirichlet problem with boundary values $\mubf_l$. Moreover, $(\Psi \circ \mubf)$ is the solution of the minimization problem 
\begin{equation}
\label{equation_zz_aux_12}
\min_f \left\{ \int_\Omega |\nabla f(\xi)|_{\Hil}^2 \ddr \xi \ : \ f \in H^1(\Omega, \Hil) \ \text{ and } \ f|_{\dr \Omega} = \Psi \circ \mubf_l \right\}. 
\end{equation}  
\end{theo}

\begin{proof}
Everything relies on \eqref{equation_zz_aux_30}. With the help of Proposition \ref{proposition_extension_boundary_values}, one can be convinced that imposing $\BT_{\mubf} = \BT_{\mubf_b}$ is the same as saying that the the trace of $(\Psi \circ \mubf)$ is $(\Psi \circ \mubf_l)$. Then, one takes $f$ to be the unique harmonic extension of $(\Psi \circ \mubf_l)$ in $H^1(\Omega, \Hil)$: it is the minimizer of \eqref{equation_zz_aux_12}. By the maximum principle, as $(\Psi \circ \mubf_l) \in \Hili$ on $\dr \Omega$, it is clear that $f \in H^1(\Omega, \Hili)$. Thus, we can simply set $\mubf := \Psi^{-1} \circ f$.  
\end{proof}

\subsection{Family of elliptically contoured distributions}
\label{subsection_location_scatter}

We finish by studying the case where the boundary values belong to a family of elliptically contoured distributions: they are parametrized by their covariance matrix. It can be seen as a generalization of the case where the measures are Gaussian. In this subsection, we would like to show that at least one solution of the harmonic problem is valued in the family of elliptically contoured distributions if it is the case for the boundary values, and to give a full solution (existence, uniqueness, regularity and Euler-Lagrange equation) under the additional assumption that the covariance matrices of the boundary values are non singular.

We will deal with centered measures (i.e. measures with zero mean) because the contribution of the mean to the Dirichlet energy can be handled independently. More precisely if $\mu \in \Prob(D)$ we denote by $m(\mu) := \int_D x \mu(\ddr x) \in D$ its mean and $\mu_0$ the centered measured defined as the push forward of $\mu$ by $(x \mapsto x - m(\mu))$. It is well known \cite[Problem 1]{Villani2003} that if $\mu, \nu$ are two probability measures then 
\begin{equation*}
W_2^2(\mu, \nu) = W_2^2(\mu_0, \nu_0) + |m(\mu) - m(\nu)|^2. 
\end{equation*}
If $\mubf \in L^2(\Omega, \Prob(D))$, we use the formula above on $\Dir_\varepsilon(\mubf)$:
\begin{equation*}
\Dir_\varepsilon(\mubf) = \Dir_\varepsilon(\mubf_0) + C_p \iint_{\Omega \times \Omega} \frac{ |m(\mubf(\xi)) - m(\mubf(\eta))|^2}{2 \varepsilon^{p+2}} \1_{|\xi - \eta| \leqslant \varepsilon} \ddr \xi \ddr \eta.
\end{equation*}
Then, sending $\varepsilon$ to $0$ and using \cite[Theorem 8.3.1]{Jost2008} to handle the part involving the Dirichlet energy of the means, one sees that  
\begin{equation*}
\Dir(\mubf) = \Dir(\mubf_0) + \frac{1}{2} \int_\Omega |\nabla [m(\mubf)](\xi)|^2 \ddr \xi. 
\end{equation*}
The term involving $m(\mubf)$ is easy to minimize (because $m(\mubf)$ is a vector-valued function, it boils down to take the harmonic extension) and it can be done independently from the term involving $\Dir(\mubf_0)$. In other words, it is not restrictive to work only with centered measures. 

Let us go back to the family of elliptically contoured distributions. As we have assumed that $D$ is compact, we cannot work with non compactly supported measures, in particular with Gaussian measures. For the rest of the subsection, we fix $\rho \in L^1(\R^q)$ a positive function compactly supported such that $\rho \Leb_D$ is a probability measure with zero mean and the identity matrix as a covariance matrix. Recall that the covariance matrix $\cov(\mu)$ of a centered measure $\mu \in \Prob(\R^q)$ is defined as: for any $i,j \in \{1,2, \ldots, q \}$,
\begin{equation*}
\cov(\mu)_{ij} := \int_{\R^q} x_i x_j \mu(\ddr x). 
\end{equation*}
For technical reasons, we also assume that $\rho$ is radial and that the Boltzmann entropy of $\rho \Leb_D$ (see \eqref{equation_boltzamnn_entropy} below) is finite. Let us denote by $S_q(\R)$ the set of symmetric $q \times q$ matrices and $S^{+}_q(\R) \subset S_q(\R)$ the set of symmetric and semi-definite positive $q \times q$ matrices. The space $S_q(\R)$ is equiped with its canonical scalar product $\langle \cdot, \cdot \rangle$ defined by $\langle A, B\rangle = \Tr(AB)$. The unique symmetric square root of a matrix $A \in S^+_q(\R)$ is denoted by $A^{1/2}$. Instead of parametrizing measures by their covariance matrix we will do it by the square root of their covariance matrix, i.e. by their standard deviation: it is more natural for homogeneity reasons and the formulas are slightly simpler. 

\begin{defi}
For any $A \in S^+_q(\R)$ we denote by $\rho_{A} \in \Prob(\R^q)$ the image measure of $\rho \Leb_D$ by the map $x \in \R^q \mapsto Ax \in \R^q$. 

The set of all $\rho_A$ for $A \in S^+_q(\R)$ is denoted by $\Probec(\R^q)$ and is called a family of elliptically contoured distributions (with reference measure $\rho \Leb_D$). 
\end{defi}

\noindent Thanks to the normalization of $\rho$, the measure $\rho_A$ has zero mean and covariance matrix $A^2$. Notice that if $A$ is invertible then 
\begin{equation*}
\rho_A(\ddr x) := \frac{1}{\det (A)} \rho \left( A^{-1}x \right) \ddr x.   
\end{equation*}
We would recover the Gaussian case by taking $\rho(x) = (2 \pi)^{-q/2} \exp(- |x|^2/2)$, but this function is not compactly supported. 

The crucial tool to establish that a harmonic extension of a mapping valued in a family of elliptically contoured distributions stays in the same family is the existence of a retraction on the set $\Probec(\R^q)$. Let us call $\Prob_2(\R^q)$ the set of probability measures on $\R^q$ with finite second moment.  

\begin{defi}
Let $R : \Prob_2(\R^q) \to \Probec(\R^q)$ the application defined by $R(\mu) := \rho_A$, where $A := \cov(\mu)^{1/2}$ is the symmetric square root of the covariance matrix of $\mu$.
\end{defi} 

\begin{prop}
The application $R : \Prob_2(\R^q) \to \Probec(D)$ leaves $\Probec(\R^q)$ unchanged and is a contraction (i.e. is $1$-Lipschitz) provided that $\Prob_2(\R^q)$ and $\Probec(\R^q)$ are endowed with the quadratic Wasserstein distance $W_2$. 
\end{prop}

\begin{proof}
The first part is obvious by the way we normalize $\rho$. The second part is a reformulation of Theorem 2.1 and Theorem 2.4 of \cite{Gelbrich1990}. 
\end{proof}

Let us prove state and prove here an easy technical lemma which will be crucial in the sequel. 

\begin{lm}
\label{lemma_composition_Lipschitz}
Let $\mubf_l : \dr \Omega \to \Prob(D)$ a Lipschitz function and $\mubf \in H^1(\Omega, \Prob(D))$ such that $\mubf|_{\dr \Omega} = \mubf_l$. Let $T : \Prob(D) \to \Prob(D)$ a $1$-Lipschitz mapping. Then $T \circ \mubf \in H^1(\Omega, \Prob(D))$ with $(T \circ \mubf)|_{\dr \Omega} = (T \circ \mubf_l)$ and 
\begin{equation*}
\Dir(T \circ \mubf) \leqslant \Dir(\mubf).
\end{equation*}
\end{lm}

\begin{proof}
As $T$ is a contraction and from the definition of $\Dir_\varepsilon$ it is obvious that 
\begin{equation*}
\Dir_\varepsilon(T \circ \mubf) \leqslant \Dir_\varepsilon(\mubf)
\end{equation*}
holds for any $\varepsilon > 0$. Then it is sufficient to send $\varepsilon$ to $0$. To get the assertion involving the boundary conditions, one can use for instance Proposition \ref{proposition_extension_boundary_values}.
\end{proof}

As we work in the compactly supported case, we add some assumption that $D$ is large enough in order for the boundary of $D$ to be invisible. More precisely, the following lemma will help us to handle the finiteness of $D$.  

\begin{lm}
\label{lemma_solution_Dirichlet_extension}
Let $\tilde{D} \subset D$ be a convex compact subset of $D$. Let $\mubf_l : \dr \Omega \to \Prob(\tilde{D})$ be a Lipschitz mapping. If $\mubf \in H^1(\Omega, \Prob(\tilde{D}))$ is a solution of the Dirichlet problem with boundary values $\mubf_l$, then, seen as an element of $H^1(\Omega, \Prob(D))$ (extending $\mubf$ by $0$ on $D \bsl \tilde{D}$), $\mubf$ is also a solution of the Dirichlet problem with boundary values $\mubf_l$ (with $\mubf_l$ seen as a mapping valued in $\Prob(D)$). 
\end{lm}

\begin{proof}
It relies on a simple observation. Let $P_{\tilde{D}} : D \to \tilde{D}$ be the Euclidean projection on $\tilde{D}$. One has that $\nu \mapsto P_{\tilde{D}} \# \nu$ is a $1$-Lipschitz function from $(\Prob(D), W_2)$ to $(\Prob(\tilde{D}), W_2)$ which leaves the boundary values $\mubf_l$ unchanged. Thus we can apply Lemma \ref{lemma_composition_Lipschitz} to see that $P_{\tilde{D}}$ maps any competitor from $H^1(\Omega, \Prob(D))$ into a competitor in $H^1(\Omega, \Prob(\tilde{D}))$. 
\end{proof}

We will say that $\tilde{D} \subset D$ is compatible with $\rho$ if it is a compact convex subset of $D$ and for any $\mu \in \Prob(\tilde{D})$, one has $R(\mu) \in \Prob(D)$. It holds if $D$ is large enough compared to $\tilde{D}$ and the diameter of the support of $\rho$. In the sequel, we will use the notations $\Probec(\tilde{D}) := \Prob(\tilde{D}) \cap \Probec(\R^q)$ and $\Probec(D) := \Prob(D) \cap \Probec(\R^q)$  

\begin{theo}
\label{theorem_gaussian_existence}
Take $\tilde{D} \subset D$ compatible with $\rho$. Let $\mubf_l : \dr \Omega \to \Probec(\tilde{D})$ a Lipschitz mapping valued in the family of elliptically contoured distributions. Then there exists $\mubf \in H^1(\Omega, \Prob(D))$ a solution of the Dirichlet problem with boundary values $\mubf_l$ such that $\mubf(\xi) \in \Probec(D)$ for a.e. $\xi \in \Omega$. 
\end{theo}

\noindent The assumption that $\tilde{D}$ is compatible with $D$ can be translated in the fact that the supports of the $\mubf_l(\xi)$ for $\xi \in \dr \Omega$ are small compared to $D$.

\begin{proof}
Let $\tilde{\mubf}$ be a solution of the Dirichlet problem with boundary values $\mubf_l$, it exists thanks to Theorem \ref{theorem_lipschitz_extension} and Theorem \ref{theorem_existence_solution_Dirichlet}. According to Lemma \ref{lemma_solution_Dirichlet_extension}, we can choose $\tilde{\mubf}$ such that $\tilde{\mubf} \in \Prob(\tilde{D})$ a.e. As $R$ is a contraction which leaves the boundary values unchanged, it is clear thanks to Lemma \ref{lemma_composition_Lipschitz} that $\mubf := R \circ \tilde{\mubf}$ is a solution of the Dirichlet problem with boundary values $\mubf_l$. By construction, $\mubf$ is valued in $\Probec(\R^q)$ and also in $\Prob(D)$ as $\tilde{D}$ is compatible with $\rho$.
\end{proof}

\noindent We believe that, conducting a careful analysis, one can prove that all solutions of the Dirichlet problem with boundary values $\mubf_l$ are valued in $\Probec(D)$.

\bigskip

Now, we want to go further and give a more explicit description of the solution valued in the family of elliptically contoured distributions. To this extent, we rely on the fact that the manifold $S^+_q(\R)$, when endowed with the distance induced by $W_2$ through the application $A \mapsto \rho_A$, has a structure of Riemannian manifold, at least when restricted to the set of non singular matrix. The computation of Wasserstein distance between gaussians distributions has been discovered independently many times (see for instance \cite{Dowson1982, Gelbrich1990}), while the resulting geometry was first investigated by Takatsu \cite{Takatsu2011}. The restriction of the Wasserstein distance to the set of gaussian measures is sometimes called the Bures metric.    

More precisely, if $A$ and $B$ are in $S^+_q(\R)$ it is known (see for instance \cite{Gelbrich1990}) that 
\begin{equation*}
W_2^2(\rho_A, \rho_B) = \Tr \left( A^2 + B^2 - 2 (AB^2A)^{1/2} \right).
\end{equation*} 
Notice that if $A$ and $B$ commute then $W_2^2(\rho_A, \rho_B) = \Tr((A-B)^2)$, which justifies that the right choice is to parametrize elements of the family of elliptically contoured distributions by the square root of their covariance matrix. Denote by $S^{++}_q(\R)$ the set of $q \times q$ symmetric definite positive matrices. If $A \in S^{+}_q(\R)$, we can define the linear map $L_A : S_q(\R) \to S_q(\R)$ by $L_A := A \otimes \Id + \Id \otimes A $. More explicitly for any $H \in S_q(\R)$ 
\begin{equation*}
L_A(H) = AH + HA.
\end{equation*}
The map $L_A$ is symmetric, and is moreover positive definite as soon as $A$ is positive definite (in this case in particular it is invertible). If $A$ is diagonal, then $L_A$ is also diagonal in the canonical basis for matrices. In particular, if $A$ and $B$ commute, then $L_A$ and $L_B$ also commute. If $A \in S^{++}_q(\R)$ and $B \in S_q(\R)$, a lengthy but straightforward computation leads to 
\begin{equation}
\label{equation_infinitesimal_Wasserstein_gaussian}
\lim_{t \to 0} \frac{W_2^2(\rho_A, \rho_{A+tB})}{t^2} = \langle B, \g_A (B) \rangle
\end{equation}  
where $\g_A : S_q(\R) \to S_q(\R)$ is a linear map defined as 
\begin{equation*}
\g_A := \frac{1}{2} (L_A)^2 (L_{A^2})^{-1} . 
\end{equation*}
More explicitly, if $A$ is a diagonal matrix with eigenvalues $\lambda_1, \lambda_2, \ldots, \lambda_q$ and $B = (B_{ij})_{1 \leqslant i,j \leqslant q}$ then 
\begin{equation*}
\langle B, \g_A (B) \rangle = \frac{1}{2} \sum_{1 \leqslant i,j \leqslant q} \frac{(\lambda_i + \lambda_j)^2}{\lambda_i^2 + \lambda_j^2} B^2_{ij}.
\end{equation*}
Notice that $\g_A$ always defines a scalar product on the space $S_q(\R)$. As a consequence, we can define the Riemannian manifold $(S_q^{++}(\R), \g)$: at each point $A \in S^{++}_q(\R)$ the tangent space, which is isomorphic to $S_q(\R)$, is endowed with the scalar product $\g_A$. If we do that, as we already know that $\Probec(\R^q)$ is a geodesic space and thanks to \eqref{equation_infinitesimal_Wasserstein_gaussian}, we see that the Riemannian distance $d_\g$ induced by $\g$ satisfies $d_\g(A,B) = W_2(\rho_A, \rho_B)$ for any $A,B \in S^{++}_q(\R)$. From this identity we can derive the following consequence. Take $\Abf \in H^1(\Omega, (S^{++}_q(\R), \g))$ a matrix-valued function and define $\rho_\Abf \in L^2(\Omega, \Prob(D))$ by $\rho_\Abf(\xi) = \rho_{\Abf(\xi)}$ for a.e. $\xi \in \Omega$. Then $\rho_\Abf \in H^1(\Omega, \Prob(D))$ and    
\begin{equation}
\label{equation_Dirichlet_Riemannian}
\Dir(\rho_\Abf) = \int_\Omega \frac{1}{2} \sum_{\alpha = 1}^p \langle \dr_\alpha \Abf(\xi), \g_{\Abf(\xi)} (\dr_\alpha \Abf(\xi)) \rangle \ddr \xi. 
\end{equation} 
To justify this identity, one can use for instance the formulation with $\Dir_\varepsilon$ (Theorem \ref{theorem_equivalence_KS}), replace the Wasserstein distance $W_2$ by the Riemannian distance $d_\g$, and use the already known equivalence between $\Dir$ and the limit of $\Dir_\varepsilon$ when $\varepsilon \to 0$ for mappings valued in a Riemannian manifold \cite[Theorem 8.3.1]{Jost2008}.

Notice that the metric tensor $\g_A$ diverges as $A$ becomes singular. Thus, it is natural to assume that the boundary values have non singular covariance matrices. With this assumption we are able to provide the full solution of the Dirichlet problem. 

\begin{theo}
\label{theorem_dirichlet_problem_location_scatter}
Take $\tilde{D} \subset D$ compatible with $\rho$. Let $\mubf_l : \dr \Omega \to \Probec(\tilde{D})$ a Lipschitz mapping such that $\det \left( \cov(\mubf_l(\xi)) \right) > 0$ for all $\xi \in \dr \Omega$ and define $\Abf_l(\xi) = \cov(\mubf_l(\xi))^{1/2}$ for all $\xi \in \dr \Omega$.

Then there exists a unique solution $\bar{\mubf} \in H^1(\Omega, \Prob(D))$ of the Dirichlet problem with boundary values $\mubf_l$ and $\bar{\mubf}(\xi) \in \Probec(D)$ for a.e. $\xi \in \Omega$. Moreover, if $\bar{\Abf} \in H^1(\Omega, (S^{++}_q(\R), \g))$ is defined by $\bar{\Abf}(\xi) := \cov(\bar{\mubf}(\xi))^{1/2}$ for a.e. $\xi \in \Omega$, then the following holds:
\begin{itemize}
\item[(i)] $\essinf_{\xi \in \Omega} \det(\bar{\Abf}(\xi)) > 0$;
\item[(ii)] $\bar{\Abf}$ is a minimizer of 
\begin{equation*}
\int_\Omega \frac{1}{2} \sum_{\alpha = 1}^p \langle \dr_\alpha \Bbf(\xi), \g_{\Bbf(\xi)} (\dr_\alpha \Bbf(\xi)) \rangle \ddr \xi.   
\end{equation*} 
among all $\Bbf \in H^1(\Omega, (S^{++}_q(\R), \g))$ which have boundary values $\Abf_l$;
\item[(iii)] $\bar{\Abf}$ is a weak solution of
\begin{equation}
\label{equation_euler_lagrange_matrix}
\sum_{\alpha = 1}^p \dr_\alpha \left( L_{\bar{\Abf}} L_{\bar{\Abf}^2}^{-1} (\dr_\alpha \bar{\Abf}) \right) + \sum_{\alpha = 1}^p \left( L_{\bar{\Abf}} L_{\bar{\Abf}^2}^{-1} ( \dr_\alpha \bar{\Abf}) \right)^2 = 0. 
\end{equation}
\item[(iv)] The mapping $\bar{\Abf}$ is smooth (namely $C^\infty$) in the interior of $\Omega$, and regularity up to the boundary holds provided $\Abf_l$ and $\dr \Omega$ are smooth enough.  
\end{itemize} 
\end{theo}

\noindent Notice that we are able to prove uniqueness among \emph{all} mappings valued in the Wasserstein space and not only those valued in the family of elliptically contoured distributions: it is one of the only case where we can prove that uniqueness holds for the Dirichlet problem. Remark also that \eqref{equation_euler_lagrange_matrix} is nothing else than the Euler-Lagrange equation associated to the problem of calculus of variations (ii). The last point is the application of the standard theory of elliptic regularity for harmonic mappings valued in Riemannian manifolds, in particular we refer the reader to \cite{Schoen1983} for the precise assumptions required for the regularity to hold up to the boundary. The only thing we will need to show is the absence of non constant \emph{minimizing tangent maps}, which we will prove thanks to an argument based on the maximum principle.

%We will not discuss the question of the regularity of $\bar{\Abf}$. As $\bar{\Abf}$ is minimizing a Dirichlet energy its regularity is expected to be better than $H^1(\Omega,(S^{++}_q(\R), \g))$. On the other hand, as $(S^{++}_q(\R), \g)$ is positively curved we cannot say automatically that $\bar{A}$ cannot blow up somewhere (except if $\Omega$ is of dimension $2$ where $\bar{\Abf}$ will be smooth, see \cite[Section 4.2]{Helein2008}). A more detailed analysis could the topic for an other study. 

\bigskip 

The rest of this subsection (and, incidentally, this article) is dedicated to the proof of Theorem \ref{theorem_dirichlet_problem_location_scatter} which is obtained by putting together Propositions \ref{proposition_existence_solution_non_singular}, \ref{proposition_EL_location_scatter}, \ref{proposition_uniqueness_location_scatter} and \ref{proposition_gaussian_no_MTM}. More precisely, the first step is to show the existence of one solution $\bar{\mubf}$ of the Dirichlet problem taking values in the family of elliptically contoured distributions for which the covariance matrices stay non singular inside $\Omega$ (Proposition \ref{proposition_existence_solution_non_singular}). Then, using the explicit expression \eqref{equation_Dirichlet_Riemannian}, it is fairly easy to show that (ii) and (iii) are satisfied (Proposition \ref{proposition_EL_location_scatter}). The hardest part is the question of uniqueness. As explained in the introduction, we will first show that any solution $\mubf$ of the Dirichlet problem with boundary values $\mubf_l$ must have $\bar{\vbf}$ as tangent velocity field, where $\bar{\vbf}$ is the tangent velocity field of $\bar{\mubf}$. Then, as $\bar{\vbf}$ will happen to be smooth enough (linear, hence Lipschitz w.r.t. variables in $D$), we will use the results about uniqueness of the ($1$-dimensional) continuity equation for smooth velocity field (Proposition \ref{proposition_uniqueness_location_scatter}). For the last point of the theorem, as $\bar{\Abf}$ is a Dirichlet minimizing mapping valued in a compact subset of the Riemannian manifold $(S_q^{++}(\R), \g)$ (thanks to point (i)), we can apply the classical theory: see \cite[Theorem IV]{Schoen1982} for the interior regularity and \cite{Schoen1983} for the boundary regularity. The only point to show is the absence of non constant minimizing tangent maps, which a consequence of Proposition \ref{proposition_gaussian_no_MTM} proved below.   

Let us begin by showing that for at least one solution of the Dirichlet problem the covariance matrices stay non singular inside $\Omega$. As a tool to measure regularity of elliptically contoured distributions, we will use the Boltzmann entropy. We define $H : \Prob(D) \to [0, + \infty]$ by 
\begin{equation}
\label{equation_boltzamnn_entropy}
H(\mu) := \begin{cases}
\dst{\int_D \mu(x) \ln(\mu(x))  \ddr x} & \text{if } \mu \text{ is absolutely continuous w.r.t. } \Leb_D, \\
+ \infty & \text{else}.
\end{cases}
\end{equation} 
It is known that $H$ is convex along generalized geodesics \cite[Theorem 9.4.10]{Ambrosio2008} and it is regular according to Proposition \ref{proposition_regular_functional}. Moreover, an explicit computation leads to $H(\rho_A) = - \ln ( \det A) + H(\rho \Leb_D)$ (with the convention $\ln(0) = - \infty$). Also, using the fact that Gaussian measures are the ones which minimize $H$ for a covariance matrix, we get that for any $\mu \in \Prob(D)$, 
\begin{equation}
\label{equation_entropy_covariance}
H(\mu) \geqslant - \frac{1}{2} \ln \left( \det \left( \cov(\mu) \right) \right) + C,
\end{equation}
where the constant $C$ is the entropy of a standard normal distribution.

\begin{prop}
\label{proposition_existence_solution_non_singular}
Take $\tilde{D} \subset D$ compatible with $\rho$. Let $\mubf_l : \dr \Omega \to \Probec(\tilde{D})$ a Lipschitz mapping such that $\det \left( \cov(\mubf_l(\xi)) \right) > 0$ for all $\xi \in \dr \Omega$. Then there exists $\bar{\mubf} \in H^1(\Omega, \Prob(D))$ a solution of the Dirichlet problem with boundary values $\mubf_l$ such that $\bar{\mubf}(\xi) \in \Probec(D)$ for a.e. $\xi \in \Omega$ and such that 
\begin{equation*}
\essinf_{\xi \in \Omega} \left[ \det \left( \cov(\bar{\mubf}(\xi)) \right) \right] > 0.
\end{equation*} 
\end{prop} 

\begin{proof}
Notice, thanks to the explicit formula for $H$ on $\Probec(\R^q)$ and as $\mubf_l$ is continuous, that $\sup_{\dr \Omega} (H \circ \mubf_l) < + \infty$. Take $\mubf \in H^1(\Omega, \Prob(\tilde{D}))$ the solution of the Dirichlet problem with boundary values $\mubf_l$ given by Theorem \ref{theorem_ishihara} (with $F = H$). Set $\bar{\mubf} := R \circ \mubf$. By the same argument as in Theorem \ref{theorem_gaussian_existence}, $\bar{\mubf} \in H^1(\Omega, \Probec(D))$ is a solution of the Dirichlet problem with boundary values $\mubf_l$. Using first the estimate \eqref{equation_entropy_covariance} and then the maximum principle \eqref{equation_maximum_principle},
\begin{align*}
\esssup_{\xi \in \Omega} \left[ - \ln \left( \det \left(  \cov(\bar{\mubf}(\xi)) \right) \right) \right] & = \esssup_{\xi \in \Omega} \left[ - \ln \left( \det \left(  \cov(\mubf(\xi))\right) \right) \right] \\
& \leqslant -2 C + 2 \esssup_{\xi \in \Omega} H(\mubf(\xi)) \\
& \leqslant -2 C + 2 \sup_{\xi \in \dr \Omega} H(\mubf_l(\xi)) < + \infty. \qedhere 
\end{align*}
\end{proof}

\noindent Until the end of the subsection, $\bar{\mubf} \in H^1(\Omega, \Probec(D))$ will denote the object defined in Proposition \ref{proposition_existence_solution_non_singular} and for a.e. $\xi \in \Omega$, one defines $\bar{\Abf}(\xi) = \cov(\bar{\mubf}(\xi))^{1/2}$. Notice that (i) of Theorem \ref{theorem_dirichlet_problem_location_scatter} is proved. Now let us derive the equation satisfied by $\bar{\Abf}$. 

\begin{prop}
\label{proposition_EL_location_scatter}
The mapping $\bar{\Abf} \in H^1(\Omega, (S^{++}_q(\R), \g))$ is a weakly harmonic map, more precisely a minimizer of 
\begin{equation*}
\Bbf \in H^1(\Omega, (S^{++}_q(\R), \g)) \mapsto \int_\Omega \frac{1}{2} \sum_{\alpha = 1}^p \langle \dr_\alpha \Bbf(\xi), \g_{\Bbf(\xi)} (\dr_\alpha \Bbf(\xi)) \rangle \ddr \xi.   
\end{equation*} 
among all $\Bbf$ which have boundary values $\Abf_l$. In particular, $\bar{\Abf}$ satisfies the Euler-Lagrange equation \eqref{equation_euler_lagrange_matrix}. 
\end{prop}

\begin{proof}
We need to prove that, for any $\Bbf \in H^1(\Omega, (S^{++}_q(\R), \g))$ with boundary values $\Abf_l$ one has 
\begin{equation*}
\int_\Omega \frac{1}{2} \sum_{\alpha = 1}^p \langle \dr_\alpha \Bbf(\xi), \g_{\Bbf(\xi)} (\dr_\alpha \Bbf(\xi)) \rangle \ddr \xi \geqslant  \int_\Omega \frac{1}{2} \sum_{\alpha = 1}^p \langle \dr_\alpha \bar{\Abf}(\xi), \g_{\bar{\Abf}(\xi)} (\dr_\alpha \bar{\Abf}(\xi)) \rangle \ddr \xi 
= \Dir(\rho_{\bar{\Abf}}) = \Dir(\bar{\mubf}). 
\end{equation*}
To prove it, if we take any $\Bbf \in H^1(\Omega, (S^{++}_q(\R), \g))$ we can build $\mubf := \rho_\Bbf$ and we have, thanks to \eqref{equation_Dirichlet_Riemannian}, the identity
\begin{equation*}
\Dir(\mubf) =  \int_\Omega \frac{1}{2} \sum_{\alpha = 1}^p \langle \dr_\alpha \Bbf(\xi), \g_{\Bbf(\xi)} (\dr_\alpha \Bbf(\xi)) \rangle \ddr \xi.
\end{equation*} 
\emph{A priori}, $\mubf$ is valued in $\Prob(\R^q)$. If we denote by $P_D : \R^q \to D$ the Euclidean projection on $D$, then 
\begin{equation*}
\Dir(\bar{\mubf}) \leqslant \Dir(P_D \# \mubf) \leqslant \Dir(\mubf),
\end{equation*} 
where the first inequality comes from the optimality of $\bar{\mubf}$ (notice that $P_D \# \cdot$ leaves the boundary values unchanged) and the second one from the fact that $P_D \# \cdot$ is a contraction (Lemma \ref{lemma_composition_Lipschitz}).

To get the Euler-Lagrange equation it is actually easier if we take the covariance matrix and not its square root as the variable. In other words we define $\bar{\Cbf} := \bar{\Abf}^2$. As $\bar{\Abf}$ is never singular, this change of variables is smooth. We have $\dr_\alpha \bar{\Cbf} = L_{\bar{\Abf}} (\dr_\alpha \bar{\Abf})$ and in particular 
\begin{equation*}
\langle \dr_\alpha \bar{\Abf}, \g_{\bar{\Abf}} (\dr_\alpha \bar{\Abf}) \rangle = \langle \dr_\alpha \bar{\Cbf} , L_{\bar{\Cbf}}^{-1}(\dr_\alpha \bar{\Cbf}) \rangle. 
\end{equation*}  
If we take $\Dbf : \Omega \to S_q(\R)$ smooth and compactly supported on $\Omega$ and that we consider $\Bbf := \bar{\Cbf} + t \Dbf$ as a competitor for small $t$, we reach the conclusion that
\begin{equation*}
\sum_{\alpha = 1}^p \langle \dr_\alpha \Dbf, L_{\bar{\Cbf}}^{-1}(\dr_\alpha \bar{\Cbf}) \rangle + \frac{1}{2} \sum_{\alpha = 1}^p \left. \frac{\ddr }{\ddr t} \right|_{t= 0} \langle \dr_\alpha \bar{\Cbf} , L_{\bar{\Cbf} + t \Dbf}^{-1}(\dr_\alpha \bar{\Cbf}) \rangle = 0.
\end{equation*} 
A simple computation leads to 
\begin{equation*}
L_{\bar{\Cbf} + t \Dbf}^{-1}(\dr_\alpha \bar{\Cbf}) = L_{\bar{\Cbf}}^{-1}(\dr_\alpha \bar{\Cbf}) - t L_{\bar{\Cbf}}^{-1} \left[ \Dbf (L_{\bar{\Cbf}}^{-1} (\dr_\alpha \bar{\Cbf})) + (L_{\bar{\Cbf}}^{-1} (\dr_\alpha \bar{\Cbf} ))\Dbf \right] + o(t^2).
\end{equation*}
Using the properties of the Trace and the symmetry of $L_{\bar{\Cbf}}^{-1}$, we conclude that the Euler-Lagrange equation reads 
\begin{equation*}
\sum_{\alpha = 1}^p \langle \dr_\alpha \Dbf, L_{\bar{\Cbf}}^{-1}(\dr_\alpha \bar{\Cbf}) \rangle - \sum_{\alpha = 1}^p \langle \Dbf, (L_{\bar{\Cbf}}^{-1} (\dr_\alpha \bar{\Cbf}) )^2 \rangle = 0.
\end{equation*}
Coming back to $\bar{\Cbf} = \bar{\Abf}^2$ and $\dr_\alpha \bar{\Cbf} = L_{\bar{\Abf}} (\dr_\alpha \bar{\Abf})$, as $\Dbf$ is arbitrary we see that we get the weak formulation of \eqref{equation_euler_lagrange_matrix}. 
\end{proof}    

As far as the regularity issues are concerned, notice that $\bar{\Abf}$ is uniformly bounded from below as a symmetric matrix (this is (i) of Theorem \ref{theorem_dirichlet_problem_location_scatter}) and also bounded from above as a symmetric matrix (as $\rho_{\bar{\Abf}} \in \Prob(D)$ and $D$ is compact), hence the operators $L_{\bar{\Abf}(\xi)} : S_q(\R) \to S_q(\R)$ are bounded with a bounded inverse uniformly in $\xi \in \Omega$. In other words, the metric tensor $\g_{\bar{\Abf}(\xi)}$ is equivalent to the canonical scalar product uniformly in $\xi \in \Omega$. In particular, the regularity $\bar{\mubf} \in H^1(\Omega, \Prob(D))$ translates in $\bar{\Abf} \in H^1(\Omega, S_q(\R))$ where $S_q(\R)$ is endowed with its usual euclidean norm $| \cdot |$. 

Let us prove uniqueness. The first step is to identify the tangent velocity field to $\bar{\mubf}$ and a (at least formal) solution of the dual problem. 
  
\begin{prop}
\label{proposition_velocity_field_gaussian}
For any $\alpha \in \{ 1,2, \ldots, p \}$ we define $\bar{\Bbf}^\alpha := L_{\bar{\Abf}} L_{\bar{\Abf}^2}^{-1}(\dr_\alpha \bar{\Abf}) \in L^2(\Omega, S_q(\R))$ and we set
\begin{equation*}
\bar{\vbf}^{\alpha}(\xi, x) := \bar{\Bbf}^\alpha(\xi) x \in \R^q.
\end{equation*} 
for $\xi \in \Omega$ and $x \in D$. Then $\bar{\vbf} \in L^2_{\bar{\mubf}}(\Omega \times D, \R^{pq})$ is the tangent velocity field to $\bar{\mubf}$.
\end{prop}

\begin{proof}
Take $\psi \in C^1_c(\Omega \times D, \R^p)$ a test function. If we define $\tilde{\psi} \in H^1(\Omega, \R^p)$ by 
\begin{equation*}
\tilde{\psi}(\xi) := \int_D \psi(\xi, x) \bar{\mubf}(\xi, \ddr x) = \int_D \psi(\xi, \bar{\Abf}(\xi) x ) \rho(x) \ddr x,
\end{equation*} 
then we see that $\tilde{\psi}$ is compactly supported in $\Omega$, in particular the integral of $\nabla \cdot \tilde{\psi}$ over $\Omega$ vanishes. It reads 
\begin{equation*}
\iint_{\Omega \times D} (\nabla_\Omega \cdot \psi)(\xi, \bar{\Abf}(\xi) x ) \rho(x) \ddr x + \iint_{\Omega \times D} \sum_{\alpha = 1}^p (\dr_\alpha \bar{\Abf}(\xi) x ) \cdot (\nabla_D  \psi^\alpha)(\xi, \bar{\Abf}(\xi) x) \rho(x) \ddr x = 0.
\end{equation*}
By doing for a fixed $\xi \in \Omega$ the change of variables $y = \bar{\Abf}(\xi) x$, one can see that $(\bar{\mubf}, \wbf \bar{\mubf})$ satisfies the continuity equation where $\wbf : \Omega \times D \to \R^p$ is given by 
\begin{equation*}
\wbf^\alpha(\xi, y) := \dr_\alpha \bar{\Abf}(\xi) \bar{\Abf}(\xi)^{-1}y.  
\end{equation*} 
Notice that $\wbf(\xi, \cdot)$ is not a gradient because $\dr_\alpha \bar{\Abf}(\xi)$ and $\bar{\Abf}(\xi)^{-1}$ do not necessarily commute. On the contrary, as the matrices $\bar{\Bbf}^\alpha(\xi)$ for $\alpha \in \{ 1,2, \ldots, p \}$ are symmetric, $\bar{\vbf}(\xi, \cdot)$ is a gradient.   

Fix $\xi \in \Omega$ and $\alpha \in \{1,2, \ldots, p \}$. We claim that the velocity field $\bar{\vbf}^\alpha(\xi, \cdot)$ is the orthogonal projection in $L^2_{\bar{\mubf}(\xi)}(D, \R^q)$ of $\wbf^\alpha(\xi, \cdot)$ on the space of gradients (actually, this is exactly how $\bar{\vbf}^\alpha$ was chosen). Not to overburden the notations, we drop momentarily the dependence on $\xi$, i.e. $\bar{\Abf} := \bar{\Abf}(\xi)$, $\bar{\Bbf}^\alpha := \bar{\Bbf}^\alpha(\xi)$ and $\dr_\alpha \bar{\Abf} := \dr_\alpha \bar{\Abf}(\xi)$ are considered as given matrices. Take $f \in C^1(D)$ a test function defined on $D$ and compute:
\begin{align*}
\int_D \nabla f(x) \cdot (\wbf^\alpha(\xi, x)  - \bar{\vbf}^\alpha(\xi, x))  \bar{\mubf}(\xi, \ddr x) & =  \int_D (\nabla f)(\bar{\Abf} x) \cdot \left( (\dr_\alpha \bar{\Abf} \bar{\Abf}^{-1} - \bar{\Bbf}^\alpha) \bar{\Abf} x \right) \rho(x) \ddr x \\
& = \int_D (\nabla \tilde{f})(x) \cdot \left( \bar{\Abf}^{-1} (\dr_\alpha \bar{\Abf} \bar{\Abf}^{-1} - \bar{\Bbf}^\alpha) \bar{\Abf} x \right) \rho(x) \ddr x,
\end{align*} 
where $\tilde{f}(x) := f(\bar{\Abf} x)$. On the other hand, as the reader can check, $\bar{\Bbf}^\alpha$ is the projection on the set of symmetric matrices of $\dr_\alpha \bar{\Abf} \bar{\Abf}^{-1}$ where the scalar product between two matrices $C$ and $D$ is given by $\Tr(\bar{\Abf} ^tC D \bar{\Abf})$. In particular, the matrix $(\dr_\alpha \bar{\Abf} \bar{\Abf}^{-1} - \bar{\Bbf}^\alpha) \bar{\Abf}^2$ is skew-symmetric, thus the matrix  $\bar{\Abf}^{-1}(\dr_\alpha \bar{\Abf} \bar{\Abf}^{-1} - \bar{\Bbf}^\alpha) \bar{\Abf}$ is also skew-symmetric. As $\rho$ is radial, it implies that the function 
\begin{equation*}
x \in D \mapsto \left( \bar{\Abf}^{-1} (\dr_\alpha \bar{\Abf} \bar{\Abf}^{-1} - \bar{\Bbf}^\alpha) \bar{\Abf} x \right) \rho(x) 
\end{equation*}  
is divergence-free. It allows us to conclude that
\begin{equation*}
\int_D \nabla f(x) \cdot (\wbf^\alpha(\xi, x)  - \bar{\vbf}^\alpha(\xi, x))  \bar{\mubf}(\xi, \ddr x) = \int_D  \tilde{f}(x)  \nabla_D \cdot \left[ \left( \bar{\Abf}^{-1} (\dr_\alpha \bar{\Abf} \bar{\Abf}^{-1} - \bar{\Bbf}^\alpha) \bar{\Abf} x \right) \rho(x) \right] \ddr x = 0,
\end{equation*}
hence the claim is proved as $f$ is arbitrary.

The claim implies that $(\bar{\mubf}, \bar{\vbf} \bar{\mubf})$ also satisfies the continuity equation: for any $\psi \in C^1_c(\Omega \times D, \R^p)$,
\begin{equation*}
\iint_{\Omega \times D} \nabla_\Omega \cdot \psi \ddr \bar{\mubf} + \iint_{\Omega \times D} \nabla_D \psi \cdot \bar{\vbf} \ddr \bar{\mubf} = \iint_{\Omega \times D} \nabla_\Omega \cdot \psi \ddr \bar{\mubf} + \iint_{\Omega \times D} \nabla_D \psi \cdot \wbf \ddr \bar{\mubf} + \iint_{\Omega \times D} \nabla_D \psi \cdot (\bar{\vbf} - \wbf) \ddr \bar{\mubf} = 0,  
\end{equation*}
as the last integral vanishes because of the claim.

As $\bar{\vbf}(\xi, \cdot)$ is a gradient (because the $\bar{\Bbf}^\alpha$ are symmetric), Proposition \ref{proposition_characterization_v_tangent} implies that $\bar{\vbf}$ is the tangent velocity field to $\bar{\mubf}$. 
\end{proof}

Notice that if we define $\bar{\varphi} : \Omega \times D \to \R^p$ by, for any $\xi \in \Omega, x \in D$ and $\alpha \in \{ 1,2, \ldots, p \}$, 
\begin{equation*}
\bar{\varphi}^\alpha(\xi, x) := \frac{1}{2} \bar{\Bbf}^\alpha(\xi) x \cdot x;
\end{equation*}
then $\bar{\vbf} = \nabla_D \varphi$. More precisely, for a.e. $\xi \in \Omega$, $\bar{\varphi}(\xi, \cdot)$ (resp. $\bar{\vbf}(\xi, \cdot)$) is defined \emph{everywhere} on $D$ as a smooth function belonging to $C^1(D,\R^p)$ (resp. $C^1(D, \R^{pq})$). Moreover the Euler-Lagrange equation  \eqref{equation_euler_lagrange_matrix}, which can be written 
\begin{equation}
\label{equation_euler_lagrange_B}
\sum_{\alpha = 1}^p \dr_\alpha \bar{\Bbf}^\alpha + \sum_{\alpha = 1}^p (\bar{\Bbf}^\alpha)^2 = 0,
\end{equation}
translates at the level of $\bar{\varphi}$ in 
\begin{equation*}
\label{equation_varphi_solution_dual}
\nabla_\Omega \cdot \bar{\varphi} + \frac{1}{2} |\nabla_D \bar{\varphi}|^2 =0.
\end{equation*}
In fact, at least formally (because of the lack of smoothness of $\bar{\varphi}$), the function $\bar{\varphi}$ is a solution of the dual problem. For $\bar{\varphi}$ to be an actual solution, we would need the $\bar{\Bbf}^\alpha$ to be $C^1$ up to the boundary: even with the elliptic regularity proved below (i.e. point (iv) of Theorem \ref{theorem_dirichlet_problem_location_scatter}), we would not reach such a strong result if we just assume that $\dr \Omega$ and $\Abf_l$ are Lipschitz. We will use $\bar{\varphi}$ to show that the tangent velocity field of any other solution of the Dirichlet problem with boundary values $\mubf_l$ must coincide with $\bar{\vbf}$. About the smoothness of the objects involved, notice that for any $\alpha \in \{ 1,2, \ldots, p \}$ one has $\bar{\Bbf}^\alpha \in L^2(\Omega, S_q(\R))$ and, given \eqref{equation_euler_lagrange_B}, the function
\begin{equation*}
\sum_{\alpha = 1}^p \dr_\alpha \bar{\Bbf}^\alpha
\end{equation*}
belongs to $L^1(\Omega, S_q(\R))$. 

\begin{prop}
Let $\mubf$ a solution of the Dirichlet problem with boundary conditions $\mubf_l$ and $\vbf$ its tangent velocity field. Then, for a.e. $\xi \in \Omega$, one has $\vbf(\xi, x) = \bar{\vbf}(\xi, x)$ for $\mubf(\xi)$-a.e. $x$. 
\end{prop}

\begin{proof}
If $\varphi \in C^1(\Omega \times D, \R^p)$ then, as $\mubf$ and $\bar{\mubf}$ share the same boundary conditions, 
\begin{equation*}
\iint_{\Omega \times D} (\nabla_\Omega \cdot \varphi + \nabla_D \varphi \cdot \vbf ) \ddr \mubf = \BT_{\mubf_l}(\varphi) = \iint_{\Omega \times D} (\nabla_\Omega \cdot \varphi + \nabla_D \varphi \cdot \bar{\vbf} ) \ddr \bar{\mubf}. 
\end{equation*}
We claim that we can insert $\varphi = \bar{\varphi}$ even though $\bar{\varphi}$ is \emph{a priori} not regular enough. In other words, given \eqref{equation_varphi_solution_dual} and the fact that $\bar{\vbf} = \nabla_D \bar{\varphi}$, we claim that 
\begin{equation}
\label{equation_equality_boundary_terme_nonsmooth}
\iint_{\Omega \times D} \left( - \frac{1}{2} |\bar{\vbf}|^2 + \bar{\vbf} \cdot \vbf \right) \ddr \mubf = \iint_{\Omega \times D} \frac{1}{2} |\bar{\vbf}|^2 \ddr \bar{\mubf}.
\end{equation}
Notice that the r.h.s. is (formally) equal to both $\BT_{\mubf_l}(\bar{\varphi})$ and $\Dir(\bar{\mubf})$: it is not surprising as $\bar{\varphi}$ is a solution of the dual problem. 

To prove such an equality we regularize $\bar{\varphi}$ in the following way. For each $\alpha \in \{1,2, \ldots,p \}$ we apply to the matrix field $\bar{\Bbf}^\alpha$ the standard truncation and convolution procedure (see \cite[Theorem 3 of Section 4.2]{Evans1992}) to produce a sequence $(\bar{\Bbf}_n^\alpha)_{n \in \N}$ which belongs to $C^1(\Omega, S_q(\R))$ and which converges to $\bar{\Bbf}^\alpha$ in $L^2(\Omega, S_q(\R))$. Moreover, as derivatives commute with convolution, we can say that 
\begin{equation*}
\lim_{n \to + \infty} \sum_{\alpha = 1}^p \dr_\alpha \bar{\Bbf}^\alpha_n =  \sum_{\alpha = 1}^p \dr_\alpha \bar{\Bbf}^\alpha = - \sum_{\alpha = 1}^p (\bar{\Bbf}^\alpha)^2, 
\end{equation*} 
and the limit takes place in $L^1(\Omega, S_q(\R))$ as we already know that the r.h.s. belongs to such a space. In particular, up to extraction the convergences hold a.e. on $\Omega$. Then we set 
\begin{equation*}
\varphi_n^\alpha(\xi, x) := \frac{1}{2} \bar{\Bbf}_n^\alpha(\xi) x \cdot x.
\end{equation*}
for $\xi \in \Omega$ and $x \in D$. By construction $\varphi_n \in C^1(\Omega \times D, \R)$ so that 
\begin{equation}
\label{equation_equality_boundary_term_smooth}
\iint_{\Omega \times D} (\nabla_\Omega \cdot \varphi_n + \nabla_D \varphi_n \cdot \vbf ) \ddr \mubf = \BT_{\mubf_l}(\varphi_n) = \iint_{\Omega \times D} (\nabla_\Omega \cdot \varphi_n + \nabla_D \varphi_n \cdot \bar{\vbf} ) \ddr \bar{\mubf}. 
\end{equation} 
It remains to show that we can pass to the limit $n \to + \infty$. Given the convergence a.e. of the $\bar{\Bbf}^\alpha_n$ and of $\sum \dr_\alpha \bar{\Bbf}^\alpha_n$, we can assume that for a.e. $\xi \in \Omega$, the functions $\nabla_\Omega \cdot \varphi_n(\xi, \cdot)$ and $\nabla_D \varphi_n(\xi,\cdot)$ converge to respectively $- \frac{1}{2} |\bar{\vbf}|^2(\xi, \cdot)$ and $\bar{\vbf}(\xi,\cdot)$ in respectively $C(D)$ and $C(D, \R^{pq})$ respectively (notice that we use the fact that $D$ is bounded). Hence for a.e. $\xi \in \Omega$,
\begin{equation}
\label{equation_zz_aux_19}
\lim_{n \to + \infty} \int_D \left( \nabla_\Omega \cdot \varphi_n(\xi, x) + \nabla_D \varphi_n(\xi, x) \cdot \vbf(\xi,x) \right) \mubf(\xi, \ddr x) = \int_D \left( - \frac{1}{2} |\bar{\vbf}|^2(\xi, x) + \bar{\vbf}(\xi, x) \cdot \vbf(\xi,x) \right)  \mubf(\xi, \ddr x).
\end{equation}  
It remains to integrate this limit over $\Omega$. The natural upper bound for the l.h.s. of \eqref{equation_zz_aux_19} is obtained by Cauchy-Schwarz and the boundedness of $D$: for any $n \in \N$, 
\begin{align*}
\Bigg| \int_D \left( \nabla_\Omega \cdot \varphi_n(\xi, x) + \nabla_D \varphi_n(\xi, x) \cdot  \vbf(\xi,x) \right) &  \mubf(\xi, \ddr x) \Bigg| \\ 
& \leqslant C \left( \sum_{\alpha = 1}^p |\bar{\Bbf}_n^\alpha(\xi)|^2 + \sqrt{\int_D |\vbf(\xi,x)|^2 \mubf(\xi, \ddr x)} \sqrt{\sum_{\alpha = 1}^p |\bar{\Bbf}_n^\alpha(\xi)|^2}  \right),  
\end{align*}
where $C$ depends only on $D$. The r.h.s. is not bounded uniformly w.r.t. $n \in \N$ but on the other hand it converges in $L^1(\Omega)$ which is enough to say that the l.h.s. is uniformly integrable. Hence, up to extraction we can integrate \eqref{equation_zz_aux_19} w.r.t. $\Omega$:
\begin{equation*}
\lim_{n \to + \infty}  \iint_{\Omega \times D} (\nabla_\Omega \cdot \varphi_n + \nabla_D \varphi_n \cdot \vbf ) \ddr \mubf = \iint_{\Omega \times D} \left( - \frac{1}{2} |\bar{\vbf}|^2 + \bar{\vbf} \cdot \vbf \right) \ddr \mubf.
\end{equation*}
Of course, the result still holds if we take $(\mubf, \vbf) = (\bar{\mubf}, \bar{\vbf})$. Thus, passing in the limit in \eqref{equation_equality_boundary_term_smooth} we get \eqref{equation_equality_boundary_terme_nonsmooth}. 

Until now we have not used the optimality of $\mubf$. We notice that the r.h.s. of \eqref{equation_equality_boundary_terme_nonsmooth} is nothing else than $\Dir(\bar{\mubf})$ which coincides with $\Dir(\mubf) = \iint_{\Omega \times D} \frac{1}{2}  |\vbf|^2 \ddr \mubf$ by optimality of $\mubf$. From there, an algebraic manipulation leads to 
\begin{equation*}
\iint_{\Omega \times D} \frac{1}{2} |\vbf - \bar{\vbf}|^2 \ddr \mubf = 0,
\end{equation*}
which easily implies the result: recall that for a.e. $\xi \in \Omega$, the velocity field $\bar{\vbf}$ is continuous on $D$. 
\end{proof}

\begin{prop}
\label{proposition_uniqueness_location_scatter}
Let $\mubf$ a solution of the Dirichlet problem with boundary conditions $\mubf_l$. Then $\mubf = \bar{\mubf}$. 
\end{prop}

\begin{proof}
Take $\mubf$ a solution of the Dirichlet problem with boundary conditions $\mubf_l$ and define $\nubf = \mubf - \bar{\mubf}$. We extend $\nubf$ on $\R^p \bsl \Omega$ by $0$: with such a choice $\nubf \in L^2(\R^p, \M(D))$ is a (signed) measure-valued mapping defined on the whole space $\R^p$ which vanishes outside a compact set. We also define $\bar{\vbf}$ as a function $\R^p \times D \to \R^{pq}$ by extending it to $0$ outside $\Omega \times D$. If $\varphi \in C^1(\R^p \times D, \R^p)$ is any smooth function then 
\begin{align*}
\iint_{\R^p \times D} \left( \nabla_\Omega \cdot \varphi + \nabla_D \varphi \cdot \bar{\vbf} \right) \ddr \nubf &= \iint_{\Omega \times D} \left( \nabla_\Omega \cdot \varphi + \nabla_D \varphi \cdot \bar{\vbf} \right) \ddr \nubf \\
& = \iint_{\Omega \times D} \left( \nabla_\Omega \cdot \varphi + \nabla_D \varphi \cdot \bar{\vbf} \right) \ddr \mubf - \iint_{\Omega \times D} \left( \nabla_\Omega \cdot \varphi + \nabla_D \varphi \cdot \bar{\vbf} \right) \ddr \bar{\mubf} \\
& = \BT_{\mubf_l}(\varphi) - \BT_{\mubf_l}(\varphi) = 0,  
\end{align*}
where we have used the fact that both $(\mubf, \bar{\vbf} \mubf)$ and $(\bar{\mubf}, \bar{\vbf} \bar{\mubf})$ satisfy the continuity equation. In other words, $(\nubf, \bar{\vbf} \nubf)$ satisfy the continuity equation on the whole space $\R^p \times D$. 

We take an arbitrary direction in $\R^p$: we fix $\alpha = 1$. As we have seen in the proof of Proposition \ref{proposition_equivalence_sobolev_cube}, the (generalized) continuity equation implies that for a.e. $\xi \in \R^{p-1} = (e_\alpha)^\perp$, the curve $t \in \R \mapsto \nubf((t, \xi))$ satisfies the ($1$-dimensional) continuity equation with a velocity field given by $\wbf(t,x) = \bar{\vbf}^\alpha((t, \xi),x)$. Notice that for a fixed $t$ the velocity field $\wbf(t, \cdot)$ is Lipschitz and bounded with Lipschitz constant and upper bound controlled by $C \1_{(t,\xi) \in \Omega}  |\bar{\Bbf}^\alpha((t,\xi))|$ where $C < + \infty$ depends only on $D$. Given that $\bar{\Bbf}^\alpha \in L^2(\Omega)$, for a.e. $\xi \in \R^{p-1}$ one has that 
\begin{equation*}
\int_{\R} \1_{(t,\xi) \in \Omega}  |\bar{\Bbf}^\alpha((t,\xi))|  \ddr t < + \infty. 
\end{equation*}  
Hence for a.e. $\xi \in \R^{p-1}$ the assumptions of \cite[Proposition 8.1.7]{Ambrosio2008} are satisfied: the curve $t \in \R \mapsto \nubf((t, \xi))$ is solution of a continuity equation which has at most one solution. As the curve identically equal to $0$ is a solution (recall that $\nubf((t, \xi)) = 0$ for $|t|$ large enough), so must be $\nubf((\cdot, \xi))$. As this result holds for a.e. $\xi \in \R^{p-1}$, it implies that $\nubf$ is identically zero, which is the desired result.    
\end{proof}

Eventually, to prove regularity, following the theory of Schoen and Uhlenbeck \cite{Schoen1982, Schoen1983}, we only need to show that there is no \emph{minimizing tangent maps}, i.e. no Dirichlet minimizing mapping which is $0$-homogeneous. We start with the following result. 

\begin{prop}
\label{proposition_ishihara_gaussian}
Let $\Abf \in H^1(\Omega, S^{++}_q(\R))$ be a weak solution of \eqref{equation_euler_lagrange_matrix}, bounded from above and uniformly away from singular matrices, and $C \in S^{+}_q(\R)$ a semi-definite positive matrix. Then the (real-valued) mapping 
\begin{equation*}
f : \xi \in \Omega \to \langle \Abf(\xi)^2, C \rangle
\end{equation*}
is subharmonic.
\end{prop}

\noindent Actually, this is nothing else than the Ishihara property (Theorem \ref{theorem_ishihara}) for the functional $\mu \mapsto  \int_D \xi \cdot (C \xi) \mu(\ddr \xi)$, though in this simpler case we can show that it holds for any solution, as we can check it by a straightforward computation. 

\begin{proof}
As in Proposition \ref{proposition_velocity_field_gaussian}, for $\alpha \in \{ 1,2, \ldots, p \}$, we set $\Bbf^\alpha := L_{\Abf} L^{-1}_{\Abf^2}(\dr_\alpha \Abf)$. Thanks to the assumptions on $\Abf$, we know that $\Bbf^\alpha \in L^2(\Omega, S_q(\R))$: this regularity is enough to justify the following computations. Indeed, with this notation at hand, for any $\alpha \in \{1,2, \ldots,p \}$ 
\begin{equation*}
\dr_\alpha f = \langle L_{\Abf} (\dr_\alpha \Abf), C \rangle = \langle L_{\Abf^2} (\Bbf^\alpha), C \rangle = \langle  \Bbf^\alpha, L_{\Abf^2}(C) \rangle.
\end{equation*}
Hence, taking the derivative again and summing over $\alpha$, 
\begin{align*}
\Delta f & = \sum_{\alpha = 1}^p \left( \langle \dr_\alpha \Bbf^\alpha, L_{\Abf^2}(C) \rangle + \langle \Bbf^\alpha, L_{L_{\Abf}(\dr_\alpha \Abf)}(C) \rangle \right)  \\
& = \sum_{\alpha = 1}^p \left( \langle \dr_\alpha \Bbf^\alpha, L_{\Abf^2}(C) \rangle + \langle \Bbf^\alpha, L_{L_{\Abf^2}(\Bbf^\alpha)}(C) \rangle \right)  \\
& = \sum_{\alpha = 1}^p \left( \langle\dr_\alpha \Bbf^\alpha, L_{\Abf^2}(C) \rangle + \mathrm{Tr} \left( \left[ 2 \Bbf^\alpha \Abf^2 \Bbf^\alpha + (\Bbf^\alpha)^2 \Abf^2 + \Abf^2 (\Bbf^\alpha)^2  \right] C \right) \right).
\end{align*}
Now, using \eqref{equation_euler_lagrange_matrix} which reads $\sum_\alpha \dr_\alpha \Bbf^\alpha = - \sum_\alpha (\Bbf^\alpha)^2$, one reaches the conclusion that 
\begin{equation*}
\Delta f  = 2 \sum_{\alpha = 1}^p  \mathrm{Tr} \left(  \Bbf^\alpha \Abf^2 \Bbf^\alpha  C \right). 
\end{equation*}
The matrix $\Bbf^\alpha \Abf^2 \Bbf^\alpha$ belongs to $S^+_q(\R)$ because $\Abf$ does, and so does $C$ by assumption. As the trace of the product of two elements of $S^+_q(\R)$ is non negative, we deduce $\Delta f \geqslant 0$ which was the claim. 
\end{proof}

\noindent With this result, it is easy to see that there exists no non constant $0$-homogeneous tangent maps. Notice, by point (i) of Theorem \ref{theorem_dirichlet_problem_location_scatter}, and as $D$ is bounded, that any minimizing tangent map, if it were to exist, would be bounded from above and uniformly away from singular matrices.   

\begin{prop}
\label{proposition_gaussian_no_MTM}
Assume $\Omega = \U$ the unit ball of dimension $p$ and $\Abf \in H^1(\Omega, S^{++}_q(\R))$ is a weak solution of \eqref{equation_euler_lagrange_matrix}, bounded from above and uniformly away from singular matrices, which is $0$-homogeneous, meaning that $\Abf(\lambda \xi) = \Abf(\xi)$ for any $\lambda > 0$. Then $\Abf$ is constant. 
\end{prop}

\begin{proof}
According to Proposition \ref{proposition_ishihara_gaussian}, for any $C \in S^+_q(\R)$, the function 
\begin{equation*}
f : \xi \in \Omega \to \langle \Abf(\xi)^2, C \rangle
\end{equation*}
is subharmonic and $0$-homogeneous, hence it is constant by the maximum principle. But clearly, the scalar product between $\Abf$ and any given symmetric positive matrix is constant if and only if $\Abf$ is itself constant. 
\end{proof}

\appendix 

\section{Measurable selection of the \texorpdfstring{$\argmin$}{argmin}}

We want to show a result which states that if $F : X \times Y \to \R$ is a function which is measurable w.r.t. $X$, then one can find a selection $m : X \to Y$ such that $F(x,m(x)) = \min_Y F(x, \cdot)$ for every $x \in X$, i.e. such that $m(x) \in \argmin_Y F(x, \cdot)$. Recall the following result which can be found in \cite[Theorem 18.19]{Aliprantis2006}. 
\begin{prop}
\label{proposition_Aliprantis1819}
Let $X$ be a measured space and $Y$ a polish space. Let $F : X \times Y \to \R$ a function such that $F(x, \cdot) : Y \to \R$ is continuous for every $x \in X$, and $F(\cdot, y) : X \to \R$ is measurable for every $y \in Y$. Assume that for every $x \in X$, the function $F(x, \cdot)$ has a minimizer over $Y$. 

Then there exists $m : X \to Y$ a measurable function such that for all $x \in X$, 
\begin{equation*}
F(x, m(x)) = \min_{y \in Y} F(x,y).
\end{equation*}  
\end{prop}

However, in particular for Proposition \ref{proposition_optimality_conditions_approximate}, we need a case where $F(x, \cdot)$ is only l.s.c.. Thus, we prove some \emph{ad hoc} result relying on the particular structure of our problem which allows to treat lower semi-continuity. 

\begin{lm}
\label{lemma_min_continuous}
Let $X$ be a measured space and $Y$ a compact metrizable space. Let $F : X \times Y \to \R$ a function such that $F(x, \cdot) : Y \to \R$ is continuous for every $x \in X$, and $F(\cdot, y) : X \to \R$ is measurable for every $y \in Y$; and let $G : Y \to \R$ a l.s.c. function. 

Then the function $H : X \to \R$ defined by 
\begin{equation*}
H(x) := \min_y \{  F(x,y) + G(y) \ : \ y \in Y \}
\end{equation*} 
is measurable. 
\end{lm}   

\begin{proof}
Notice that $Y$ is separable as it is compact and metrizable. For any rational number $a$, the exists a sequence dense in $\{ y \in Y \ : \ G(y) \leqslant a \}$. Hence, we can construct a sequence $(y_n)_{n \in \N}$ such that for any rational number $a$ there is a subsequence of $(y_n)_{n \in \N}$ which is included and dense in $\{ y \in Y \ : \ G(y) \leqslant a \}$. 

Set $\tilde{H}(x) := \inf_n F(x,y_n) + G(y_n)$ which is measurable and larger than $H$. Let us prove that it is equal to $H$. Indeed, if $x \in X$, by standard arguments of calculus of variations, there exists $\bar{y}$ such that $H(x) = F(x,\bar{y}) + G(\bar{y})$. For any $a > G(\bar{y})$ rational, take a subsequence $(y_{n_k})_{k \in \N}$ which belongs to  $\{ y \in Y \ : \ G(y) \leqslant a \}$ and which converges to $\bar{y}$. By continuity of $F$, one has 
\begin{equation*}
\tilde{H}(x) \leqslant \liminf_{k \to + \infty} \left( F(x,y_{n_k}) + G(y_{n_k}) \right) \leqslant F(x, \bar{y}) + a. 
\end{equation*}     
As $a$ can be chosen arbitrary close to $G(\bar{y})$, we have that $\tilde{H}(x) \leqslant F(x,\bar{y}) + G(\bar{y}) = H(x)$.
\end{proof}

\begin{prop}
\label{proposition_measurable_selection_argmin}
Let $X$ be a measured space and $Y$ a compact metrizable space. Let $F : X \times Y \to \R$ a function such that $F(x, \cdot) : Y \to \R$ is continuous for every $x \in X$, and $F(\cdot, y) : X \to \R$ is measurable for every $y \in Y$; and let $G : Y \to \R$ a l.s.c. function. 

Then there exists $m : X \to Y$ a measurable function such that for any $x \in X$, 
\begin{equation*}
F(x,m(x)) + G(m(x)) = \min_y \{  F(x,y) + G(y) \ : \ y \in Y \}. 
\end{equation*}
\end{prop}

\begin{proof}
As in the previous lemma, define $H(x) := \min_y \{  F(x,y) + G(y) \ : \ y \in Y \}$, it is a measurable function valued in $\R$. Let $\Gamma$ be the mapping going from $X$ and valued in the compact subsets of $Y$ defined by $\Gamma(x) = \argmin_x (F(x, \cdot) + G(\cdot))$ which means 
\begin{equation*}
\Gamma(x) := \{ y \in Y \ : \ F(x,y) + G(y) = H(x) \}.
\end{equation*}
Notice that $\Gamma(x)$ is never empty thanks to standard arguments of calculus of variations. To prove the existence of a measurable selection of $\Gamma$, we rely on \cite[Theorem 18.13]{Aliprantis2006}: it is sufficient to show that $\Gamma$ is measurable, which means that $\{ x \in X \ : \ \Gamma(x) \cap Z \neq \emptyset \}$ is a measurable set of $X$ for any closed set $Z \subset Y$. But one can be convinced that, for a fixed $Z \subset Y$ closed, 
\begin{equation*}
\Gamma(x) \cap Z \neq \emptyset \ \Leftrightarrow \ H(x) = H_Z(x),   
\end{equation*}
where $H_Z(x) := \min_z \{  F(x,z) + G(z) \ : \ z \in Z \}$. Thanks to Lemma \ref{lemma_min_continuous}, both $H$ and $H_Z$ are measurable, thus the set on which they coincide is measurable, which concludes the proof.  
\end{proof}

\section{\texorpdfstring{$H^{1/2}$}{H1/2} determination of the square root}
\label{section_appendix_H1/2_square_root}

In this appendix we want to prove Lemma \ref{lemma_H1/2_no_jumps}, which states that, with $\Sph^1$ the unit circle of the complex plane $\C$ and $\U$ its unit disk, there is no function $f \in H^{1/2}(\Sph^1, \Sph^1)$ such that $f(\xi)^2 = \xi$ for a.e. $\xi \in \Sph^1$ (where the multiplication is understood as a complex multiplication). We take for granted that there is no continuous function $f \in C(\Sph^1, \Sph^1)$ such that $f(\xi)^2 = \xi$ for all $\xi \in \Sph^1$. Hence, it is enough to reason by contradiction and to prove that a function $f \in H^{1/2}(\Sph^1, \Sph^1)$ such that $f(\xi)^2 = \xi$ for a.e. $\xi \in \Sph^1$ admits a continuous representative. 

We start with some easy lemma which states that $H^{1/2}(\Sph^1, \U)$ is stable by composition with Lipschitz function. 

\begin{lm}
Let $u : \Sph^1 \to \R$ a Lipschitz function and $f \in H^{1/2}(\Sph^1, \Sph^1)$. Then $(u \circ f) \in H^{1/2}(\Sph^1, \R)$.  
\end{lm} 

\begin{proof}
It is well known (see \cite[Chapter 3]{McLean2000}) that there exists $\tilde{f} \in H^1(\U, \U)$ whose trace on $\Sph^1$ is $f$. Clearly, the function $u \circ \tilde{f}$ stays in $H^1(\U, \R)$, hence its trace, which is nothing else than $u \circ f$, is in $H^{1/2}(\Sph^1, \R)$.
\end{proof}

\noindent Then, let us prove that an $H^{1/2}$ function cannot have a jump. 

\begin{prop}
Let $f \in H^{1/2}([0,1], \R)$ such that $f(\xi) \in \{ 0,1 \}$ for a.e. $\xi \in [0,1]$. Then there is a representative of $f$ which is constant. 
\end{prop}

\begin{proof}
We reason by contraposition: we assume that $f$ is not constant, which translates in $0 < \int_0^1 f < 1$ and we want to show that $f \notin H^{1/2}([0,1], \R)$. Recall that it is sufficient to prove, given the definition of the $H^{1/2}$ norm \cite[Chapter 3]{McLean2000}, that
\begin{equation*}
\iint_{[0,1] \times [0,1]} \frac{|f(\eta) - f(\theta)|}{|\theta - \eta|^2} \ddr \eta \ddr \theta = + \infty.
\end{equation*}
Take $t > 0$ large enough. The function 
\begin{equation*}
\xi \mapsto \frac{1}{\sqrt{t}} \int_{\xi - t^{-1/2}/2}^{\xi + t^{-1/2}/2} f(\eta) \ddr \eta
\end{equation*}  
is continuous on $[t^{-1/2}/2, 1- t^{-1/2}/2]$ and has a means which belongs to $[c, 1-c]$, where $0 < c < 1$ is independent of $t$ (provided it is large enough) and is related to $0 < \int_0^1 f < 1$. Hence, there exists $\xi_t$ such that
\begin{equation*}
\int_{\xi_t - t^{-1/2}/2}^{\xi_t + t^{-1/2}/2} f(\eta) \ddr \eta \in \left[ \frac{c}{\sqrt{t}}, 1 - \frac{c}{\sqrt{t}} \right].
\end{equation*}  
Heuristically, $\xi_t$ is close to a point where $f$ ``jumps''. In particular, it implies that 
\begin{equation*}
\Leb_{[0,1]} \otimes \Leb_{[0,1]} \left( \left\{ (\eta, \theta) \in \left[ \xi_t - \frac{1}{2\sqrt{t}}, \xi_t + \frac{1}{2\sqrt{t}} \right]^2  \ : \ f(\eta) = 1 \text{ and } f(\theta) = 0 \right\} \right) \geqslant \frac{c^2}{t}.
\end{equation*}
As a consequence,  
\begin{equation*}
\Leb_{[0,1]} \otimes \Leb_{[0,1]} \left( \left\{ (\eta, \theta) \in [0,1]^2 \ : \ \frac{|f(\eta) - f(\theta)|}{|\theta - \eta|^2} \geqslant t \right\} \right) \geqslant \frac{c^2}{t}.
\end{equation*}
This estimate leads to 
\begin{equation*}
\iint_{[0,1] \times [0,1]}  \frac{|f(\eta) - f(\theta)|}{|\theta - \eta|^2} \ddr \eta \ddr \theta  
 = \int_0^{+ \infty} \left[ \Leb_{[0,1]} \otimes \Leb_{[0,1]} \left( \left\{ (\eta, \theta) \in [0,1]^2  \ : \ \frac{|f(\eta) - f(\theta)|}{|\theta - \eta|^2} \geqslant t \right\} \right) \right] \ddr t = + \infty. \qedhere 
\end{equation*}
\end{proof}

\noindent With these two lemmas, we can easily arrive to our conclusion. 

\begin{proof}[Proof of Lemma \ref{lemma_H1/2_no_jumps}]
Let $f \in H^{1/2}(\Sph^1, \Sph^1)$ such that $f(\xi)^2 = \xi$ for a.e. $\xi \in \Sph^1$. We want to show that $f$ is continuous. Take $X$ an arc of circle of $\Sph^1$. If $X$ is small enough, there are two continuous functions $f_0$ and $f_1$ (the complex square roots) defined on $X$ such that for all $\xi \in X$, $z^2 = \xi$ if and only if $z \in \{ f_0(\xi), f_1(\xi) \}$. Moreover, if $X$ is small, the ranges of $f_0$ and $f_1$ are far apart, hence we can find a Lipschitz function $u : \U \to \{ 0,1 \}$ such that $u \circ f_0 = 0$ and $u \circ f_1 = 1$ on $X$. Thus, $(u \circ f)(\xi) \in \{ 0,1 \}$ for $\xi \in X$. The previous lemmas allow us to conclude that the function is in $H^{1/2}(X, \{ 0,1\})$, hence constant, which means that $f$ is continuous on $X$. As $X$ is arbitrary, $f$ is continuous on $\Sph^1$, which is a contradiction.    
\end{proof}

\section*{Acknowledgments}

The author acknowledges the financial support of the French ANR ISOTACE (ANR-12-MONU-0013) and benefited from the support of the FMJH ``Program Gaspard Monge for Optimization and operations research and their interactions with data science'' and EDF via the PGMO project VarPDEMFG.

He also thanks Yann Brenier, Guillaume Carlier, Filippo Santambrogio and Thierry De Pauw for fruitful discussions, precious advice and pointing him out relevant references. In particular, Theorem \ref{theorem_lipschitz_extension} was provided by Filippo Santambrogio and Thierry de Pauw.

\end{document}